\documentclass[12pt,english,reqno]{amsart}
\usepackage{amscd,amssymb,amsmath,amstext,amsthm,amsfonts,bbold}
\usepackage[dvips]{graphicx}
\usepackage{mathrsfs}
\usepackage{float}
\usepackage[colorlinks=true,linkcolor=blue, urlcolor=red, citecolor=blue, backref]{hyperref}	
\usepackage[dvipsnames]{xcolor}
\usepackage[ansinew]{inputenc}
\usepackage{amssymb}
\usepackage{amscd}
\usepackage{amsfonts}
\usepackage{amsbsy}
\usepackage{marginnote}
\usepackage[normalem]{ulem}

\usepackage{a4wide}
\newtheorem{Theorem}{Theorem}[section]

\newtheorem{maintheorem}{Theorem}

\newtheorem{maincorollary}[maintheorem]{Corollary}

\newtheorem{T}{Theorem}[section]
\newtheorem{Corollary}[T]{Corollary}
\newtheorem{Proposition}[T]{Proposition}
\newtheorem{Lemma}[T]{Lemma}

\newtheorem{Remark}[T]{Remark}

\newtheorem*{remark}{Remark}
\newtheorem{Definition}[T]{Definition}

\newtheorem{Claim}{Claim}

\def \AA {{\mathbb A}}

\def \DD {{\mathbb D}}
\def \RR {{\mathbb R}}
\def \ZZ {{\mathbb Z}}
\def \NN {{\mathbb N}}
\def \PP {{\mathbb P}}

\def \LL {{\mathbb L}}

\def \XX {{\mathbb X}}

\def \cl {\mathcal{L}}
\def \cf {\mathcal{F}}
\def \cm {\mathcal{M}}

\def \cq {\mathcal{Q}}
\def \cp {\mathcal{P}}
\def \cc {\mathcal{C}}
\def \ch {\mathcal{H}}
\def \cs {\mathcal{S}}
\def \cu {\mathcal{U}}
\def \co {\mathcal{O}}
\def \cn {\mathcal{N}}

\def \cw {\mathcal{W}}
\def \ca {\mathcal{A}}

\def \cz {\mathcal{Z}}

\newcommand{\supp}{\operatorname{supp}}
\newcommand{\diam}{\operatorname{diam}}
\newcommand{\fix}{\operatorname{Fix}}
\newcommand{\per}{\operatorname{Per}}
\newcommand{\diameter}{\operatorname{diameter}}

\newcommand{\dist}{\operatorname{dist}}

\newcommand{\topp}{\operatorname{top}}
\newcommand{\leb}{\operatorname{Leb}}

\newcommand{\interior}{\operatorname{interior}}
\newcommand{\osc}{\operatorname{osc}}
\newcommand{\bc}{\operatorname{\bf c}}

\newcommand{\Ptop}{P_{\topp}}
\newcommand{\htop}{h_{\topp}}

\newcommand{\cE}{\mathcal{E}}

\newcommand{\cZ}{\mathcal{Z}}

\newcommand{\ce}{{\mathcal E}}

\begin{document}


\title{Thermodynamic formalism for expanding measures}

\author{Vilton  Pinheiro}
\address{Departamento de Matem\'atica, Universidade Federal da Bahia\\
Av. Milton Santos s/n, 40170-110 Salvador, Brazil.}
\email{viltonj@ufba.br}

\author{Paulo Varandas}
\address{Departamento de Matem\'atica, Universidade Federal da Bahia\\
Av. Milton Santos s/n, 40170-110 Salvador, Brazil.}
\email{paulo.varandas@ufba.br}

\subjclass[2010]{Primary: 37D35, 37A30, 37C40, 37D25.}


\begin{abstract}
In this paper we study the thermodynamic formalism of strongly transitive  endomorphisms $f$, focusing on the set all expanding measures.
In case $f$ is a non-flat $C^{1+}$ map defined on a Riemannian manifold, these are invariant probability measures with all its Lyapunov exponents positive.
Given a H\"older continuous potential $\varphi$ we prove the uniqueness of the equilibrium state among the space of expanding measures.
Moreover, we show that the existence  of an expanding measure $\mu$ maximizing the entropy on the the space of expanding measures implies the existence and uniqueness of  equilibrium state $\mu_{\varphi}$ on the space of expanding measures for any H\"older  continuous potential $\varphi$ with a small oscillation $\text{osc }\varphi=\sup\varphi-\inf\varphi$.
As some applications, we prove that Collet-Eckmann quadratic maps does not admit phase transition  for H\"older potential, and show that for Viana maps and every H\"older continuous potential of sufficiently small oscillation has a unique equilibrium state. 
\end{abstract}

\date{\today}


\maketitle 
\tableofcontents

\section{Introduction}

The theory of equilibrium states of smooth dynamical systems was initiated in the mid seventies by pioneering works of Sinai, Ruelle and Bowen~\cite{BR75,Bo75,Ru76b,Ru2,Si72}.
For uniformly hyperbolic diffeomorphisms and flows they proved that equilibrium states exist and are unique for every Hölder continuous potential, restricted to every basic piece of the non-wandering set.
The basic strategy to prove this remarkable fact was to (semi)conjugate the dynamics to a subshift of finite type, via a Markov partition.
However, to extend the theory beyond the scope of uniform hyperbolicity one must overcome some important difficulties as the absence of finite Markov partitions and the presence of critical points, singular points or discontinuities.
In fact,  equilibrium states may actually fail to exist if the dynamical system is not hyperbolic, exhibits critical points or singularities (see e.g. ~\cite{Buz01,Misiurewicz}).

\smallskip
In this article we focus on the set of expanding measures for a strongly transitive map.
Strongly transitive maps appear profusely in dynamical systems.
All continuous transitive piecewise monotone interval maps, expanding Lorenz maps, Viana maps, minimal homeomorphisms and uniformly expanding maps on a connected Riemannian manifold, are strongly transitive 
$($\footnote{ One can produce many examples of strongly transitive invariant sets using the fact that when the support of an ergodic expanding invariant probability measure has nonempty interior, then it contains an open, dense and forward invariant subset, with full measure, where the dynamics is strongly transitive (see Proposition~\ref{Propositioniougoygh} in Appendix~I).}$)$.     
In the context of non-flat $C^{1+}$ maps acting on a compact Riemannian manifold, an expanding measure is an invariant probability measure with all its Lyapunov exponents positive. 

It is worth mentioning that, nevertheless, the concepts of Lyapunov exponents and expanding measures can be defined for maps acting on more general compact metric spaces $\XX$, including Cantor sets.

\smallskip
The starting point for our analysis is a fundamental fact about any expanding measure $\mu$, namely, the existence of an induced Markov map $F$ with full branches such that $\mu$ is $F$-liftable \cite{Pi11}. 
For strongly transitive maps, we use this fact to approximate $\mu$ by another expanding measure $\overline{\mu}$ with full support.
Moreover, we prove that all such probability measures $\overline{\mu}$ can be lifted to a same induced Markov 
map $F$ with full branches.
This allows us to reduce the study of the thermodynamic formalism of expanding measures for strongly transitive maps to the study of strongly transitive countable Markov shifts, which are also strongly transitive, constrained to issues relative to the integrability of the inducing time.
One should refer that the unconstrained thermodynamic formalism for countable Markov shifts has been 
extensively studied and had important contributions by Aaronson, Buzzi, Denker, Mauldin, Sarig, Urbanski and Yuri ~\cite{ADU,BS03,MU,Sa99,Sa03,Sa09,Yu03}, among many others.
Using this special induced Markov map we prove that for each Hölder continu\-ous potential there is at most one equilibrium state among the expanding measures (see Theorem~\ref{Maintheorem00}).
This is crucial in the proof of the existence and uniqueness of equilibrium states for non-singular maps and 
Hölder continuous expanding potentials (Theorem~\ref{TheoOLDCorollaryMAINknlielMI}).
Furthermore, if an expanding measure $\mu$ is the unique invariant probability measure maximizing the entropy then 
$f$ has a unique equilibrium state $\mu_{\varphi}$ (among all invariant probability measures) for any Hölder continuous potential $\varphi$ with a small oscillation, meaning that $\osc\varphi=\sup\varphi-\inf\varphi$ is small enough (cf. Theorem~\ref{Maintheorem11}~and~\ref{TheoCorollaryU8883jljlEMI}
for the precise statements). This fact evidences the propagation of non-uniformly expanding behavior and can be interpreted as a robustness phenomenon on the existence and uniqueness of expanding equilibrium states.

\smallskip
In Section~\ref{SectionApplications} we illustrate the multiple applications of our results for some important classes of dynamical systems, including non-uniformly expanding interval maps, $C^1$-open sets of non-uniformly expanding local diffeomorphisms 
and $C^3$-open sets of multidimensional maps with critical points (Viana maps).
Two of such applications are highlighted in Corollaries~\ref{CorollaryCE} and ~\ref{CorollaryVIANA} below. 

\smallskip
From our point of view, this work initiates the study of the thermodynamic formalism for expanding measures in a systematic way. We hope that it can be a useful in the study of phase transitions for non-uniformly hyperbolic maps 
and to the thermodynamic formalism of partial hyperbolic diffeomorphisms.

\subsection*{Historical background and related works}

In the last four decades, the thermodynamic formalism outside the classical setting of uniformly hyperbolic maps has been exhaustively studied, with particular relevance given to absolutely continuous invariant measures and to measures with maximal entropy (which correspond to equilibrium states associated to minus log of the Jacobian and to constant potentials, respectively).
Dynamical systems which admit a unique measure of maximal entropy are referred to as intrinsically ergodic. 
Some important classes of intrinsically ergodic transformations include certain classes of piecewise monotonic interval transformations considered by \cite{Ho}, piecewise expanding maps \cite{Bu99,BPS01,BM02,BS03}, 
countable Markov shifts ~\cite{BS03,Sa99,Sa01,Sa03,Yu03}, some diffeomorphisms with homoclinic tangencies and no dominated splittings \cite{LR2,WY01},
some one-dimensional real and complex maps ~\cite{BR06,BrK98,BT,DKU90,DU91b,DPU96,IT10, Li, PR10,PSZ07,Ur98}, the non-uniformly expanding local homeomorphisms
 \cite{OV08,VV}, partially hyperbolic diffeomorphisms  \cite{BFSV,HHTU10} and
 transitive $C^{\infty}$ surface diffeomorphisms \cite{BCS,Sa13}, just to mention a few.

\smallskip
Closely related with our work, one should mention the results about non-uniformly expanding measures (NUE-measures). Given a $C^{1+}$-map acting on a compact Riemannian manifold $M$, we say that a forward invariant set 
$\Lambda\subset M$ is \emph{non-uniformly expanding} if there exists $\lambda>0$ such that $\Lambda\subset\ch(\lambda)$, where 

\begin{equation}\label{asdfg}
  \ch(\lambda)=\Big\{ x\in M  \colon \limsup_{n\to\infty} \frac{1}{n}\sum_{n=0}^{n-1}\log\|(Df (f^j(x))^{-1}\|^{-1}\ge\lambda\Big\}.
\end{equation}

An invariant probability measure $\mu$ is called {\em non-uniformly expanding} (NUE) if $\mu(\ch(\lambda))=1$ for some $\lambda>0$.
However, while  all NUE probability measures are clearly expanding ones (as the NUE hypothesis implies that all Lyapunov of $\mu$ are bigger or equal than $\lambda$), the converse is not true.
The thermodynamic formalism for  non-uniformly expanding sets and hyperbolic potentials have been studied by many authors 
(see \cite{AOS,AMS06,BCV,CV13, OV08,RV17,RS17,Sa20,Sa21,VV} and references therein).

\section{Statement of the main results}\label{SectionSMS}

Let $M$ be a compact Riemannian manifold (possibly with boundary), let $\cc\subset M$ be a closed subset with empty interior. Let $f:M\setminus\cc\to M$ be a $C^{1+}$ local diffeomorphism such that $\#f^{-1}(x)<+\infty$ for every $x\in\XX$. 
The set $\cc$, possibly empty, will be called the critical/singular set of $f$.
Let $\dist(x,y)$ be the geodesic distance on $M$, $\dist(x,\cc):=\inf\{\dist(x,y)\,;\,y\in\cc\}$ and $T_x^1M=\{v\in T_xM\,;\,|v|=1\}$ denote the unit tangent space at $x$.
The critical/singular set $\cc$ is called {\bf\em non-degenerated} when there exist constants $B,\beta>0$ such that: 
\begin{enumerate}
\item[(\,C1\,)] \hspace{0.3cm} $|\log \|Df(x)v\|\,|\le B+\beta|\log\dist(x,\cc)|$ \; $\forall v\in T_x^1M$ and $x\notin\cc$;
\item[(\,C2\,)] \hspace{0.3cm} $|\log\frac{\|Df(x)^{-1}\|}{\|Df(y)^{-1}\|} \,|\le \frac{B}{\dist(x,\cc)^\beta}\dist(x,y)$ \; $\forall x,y\notin\cc$ with $\dist(y,x)<\frac{1}{2}\dist(x,\cc)$.
\end{enumerate}
We say that $\cc$ is {\bf\em non-flat} provided that there exist constants $A<B$ and $0<\alpha<\beta$ such that:  
\begin{enumerate}
\item[(C1')] \hspace{0.3cm} $A+\alpha|\log \dist(x,\cc)|\le|\log \|Df(x)v\|\,|\le B+\beta|\log \dist(x,\cc)|$ \; $\forall v\in T_x^1M$ and $x\notin\cc$.
\item[(\,C2\,)] \hspace{0.3cm} $|\log\frac{\|Df(x)^{-1}\|}{\|Df(y)^{-1}\|}
\,|\le \frac{B}{\dist(x,\cc)^\beta}\dist(x,y)$ \; $\forall y,x\notin\cc$ with $\dist(y,x)<\frac{1}{2}\dist(x,\cc)$.
\end{enumerate}

\smallskip
In order to obtain uniqueness of equilibrium states we need some form of transitivity.
Recall that $f$ is transitive if $\alpha_f(x)=M$ for a dense set of points $x\in M\setminus\cc$, where 
$$
\alpha_f(x)=\bigcap_{n\ge0}\overline{\bigcup_{k\ge0}f^{-k}(f^{-n}(x))}
$$ 
is the $\alpha$-limit set of $x$. 
If $\alpha_f(x)=M$ for every $x\in M\setminus\cc$ then we say that $f$ is {\bf\em strongly transitive}.

For $\delta>0$, consider the truncated distance $\dist_{\delta}(x,y)=\min\{\dist(x,y),\delta\}$.

\begin{Definition}
 An $f$-invariant pro\-bability measure is called an {\bf\em expanding measure} for the differentiable map  $f$ if
following slow recurrence condition to the critical set holds
\begin{equation}\label{Equationiyiy545bsA}
  \int_{M}\log\dist_1(x,\cc)d\mu(x)>-\infty
\end{equation}
  and all the Lyapunov exponents of $\mu$ are positive, that is, \begin{equation}\label{Equationiyiy545bsB}
0<\lim_{n\to\infty}\frac{1}{n}\log\|Df^n(x)v\|<+\infty\;\text{ for $\mu$-almost every $x\in M$ and every } v\in T^1_xM.
\end{equation}
Let us denote by $\ce(f)$ the set of all $f$-invariant expanding probability measures of $f$. 
\end{Definition}

\begin{Remark}
Observe that the definition of expanding measure involves the Lyapunov exponents instead of the average along co-norms of the derivative cocycle (compare to  
\eqref{asdfg}). In particular, there is no uniform lower bound requirement on the Lyapunov exponents of measures in 
$\cE(f)$. 
\end{Remark}

\begin{Remark}\label{Remarkjgiu91qa}
It follows directly from {\em (C1')} above that if\, $\cc$ is non-flat then any invariant probability measure $\mu$ having only finite Lyapunov exponents satisfies \eqref{Equationiyiy545bsA}. In particular, in this context $\ce(f)$ coincides with the set of all invariant probability measures having only positive Lyapunov exponents.
\end{Remark}

\begin{Remark}\label{Remarkiovy223}
If $\cc$ is non-degenerated and either $\lim_{x\to c}|\det Df(x)|=0$ for every $c\in\cc$ or $\lim_{x\to c}|\det Df(x)|=+\infty$ for every $c\in\cc$ then   \eqref{Equationiyiy545bsA} holds for every invariant probability measure $\mu$ having only finite Lyapunov (see Lemma~\ref{LemmaDoLIftAndS} in Appendix~I). Thus, also in this case,  $\ce(f)$ is the set of all invariant probability measures having only positive Lyapunov exponents.
\end{Remark}

Assume that $$h(\ce(f)):=\sup\{h_{\mu}(f)\,;\, \mu\in\ce(f)\}<+\infty.$$
In case the map $\bar f$ has finite topological entropy it is clear that $h(\ce(f)) \le \htop(\bar f)<+\infty$.
Given a continuous function $\varphi:M\to\RR$, we say an $f$-invariant probability measure $\mu$  is an {\bf\em expanding equilibrium state for $\varphi$} if $\mu$ is an expanding measure and 
\begin{equation}\label{cvp}
h_{\mu}(f)+\int\varphi d\mu=P_{\mathcal{E}(f)}(\varphi):=\sup\bigg\{h_{\nu}(f)+\int\varphi d\nu\,:\,\nu\in\mathcal{E}(f)\bigg\},
\end{equation} 
where $P_{\ce(f)}(\varphi)$ is called the {\bf\em expanding pressure} of $\varphi$ $($\footnote{ Since  $h(\ce(f))<+\infty$, $M$ compact and $\varphi$ continuous, we have that $P_{\ce(f)}(\varphi)\le h(\ce(f))+\max|\varphi|<+\infty$.}$)$.
In other words, an expanding equilibrium state is an expanding measure which satisfies the previous constrained variational principle (second equality in ~\eqref{cvp}).
The {\bf\em oscillation} of a continuous function $\varphi:M\to\RR$ is $$osc(\varphi)=\sup_{x\in M}\varphi(x)-\inf_{x\in M}\varphi(x).$$
We say that $\mu\in\ce(f)$ is an {\bf\em measure of maximal expanding entropy} if it is an expanding equilibrium state for the null potential $\varphi\equiv0$, that is, $\mu\in\ce(f)$ and $
h_{\mu}(f)=h(\ce(f))$.

A continuous function $\varphi: M\to\RR$ is called {\bf\em expanding potential} if 
\begin{equation}\label{Equationhouh567aai}
	h_{\mu}(f)+\int\varphi d\mu<P_{\ce(f)}(\varphi),\;\:\forall\mu\in\cm^1(\bar{f}\,)\setminus\ce(f),
\end{equation}
where $\cm^1(\bar{f}\,)$ is the set of all invariant probability measures for the continuous extension $\bar{f}$.
Our first main result ensures, under the previous general assumptions, that there exists at most one expanding equilibrium state for every Hölder continuous potential.  

\begin{maintheorem}\label{Maintheorem00}
If $f$ is strongly transitive and $\cc$ is non-degenerated then $f$ has at most one expanding equilibrium state $\mu_{\varphi}$ for any Hölder continuous potential $\varphi$.
\end{maintheorem}

The statement in Theorem~\ref{Maintheorem00} concerns only on the uniqueness of expanding equilibrium states and such equilibrium states may not exist (see e.g. Example~\ref{se:MP}).
For that reason it is important to provide sufficient conditions under which such expanding equilibrium states do exist.  
The next theorem ensures that somewhat surprisingly, even if this is not the case, the single existence of a probability measure of maximal expanding entropy implies on the existence of expanding equilibrium states for nearby potentials. More precisely:

\begin{maintheorem}\label{Maintheorem11}
Suppose that $f$ is strongly transitive, $\cc$ is non-degenerated and $\sup_{x\in M\setminus\cc}\|D f(x)\|<$ $+\infty$.
If $f$ admits a probability measure of maximal expanding entropy then
there exists $\delta_0>0$ such that  $f$ has a unique expanding equilibrium state $\mu_{\psi}$ for any Hölder continuous potential $\psi$ with oscillation smaller that $\delta_0$.
\end{maintheorem}

The existence and uniqueness of equilibrium states, not only expanding equilibrium states, 
can be obtained for expanding potentials and maps with further regularity. 
This extra regularity demands that $f$ is piecewise injective and it can be extended to a $C^1$ map.
Recall that $f$ is {\bf\em piecewise injective} if there exists a finite open cover of $M\setminus\cc$ by 
sets $\{M_1,\cdots,M_n\}$ such that $f|_{M_j}$ is injective for every $1\le j\le n$.

\begin{maintheorem}\label{TheoOLDCorollaryMAINknlielMI}
 Suppose that $f$ is strongly transitive and $\cc$ is non-degenerated. 
If $f$ is piecewise injective and it can extended to a $C^1$ map  $\bar{f}:M\to M$, then $f$ has a unique equilibrium state $\mu_{\varphi}$ for any Hölder continuous and expanding potential $\varphi$. 
Moreover, $\mu_{\varphi}$ is an expanding measure.
\end{maintheorem}

The previous theorem evidences that existence and uniqueness of equilibrium states hold Hölder continuous expanding potentials.
We observe there exist relevant classes of examples where the space of expanding potentials has non-empty interior (see e.g. \cite{VV}).

\begin{maintheorem}
\label{TheoCorollaryU8883jljlEMI}
  Suppose that $f$ is strongly transitive and $\cc$ is non-degenerated. 
If $f$ is piecewise injective and it can extended to a $C^1$ map  $\bar{f}:M\to M$, then the following statements are true.
\begin{enumerate}
\item If $h_{\mu}(f)<h(\ce(f))$ $\forall\mu\in\cm^1(\bar{f}\,)\setminus\ce(f)$ then
$\exists\delta_0>0$ such that  $f$ has a unique expanding equilibrium state $\mu_{\psi}$ for any Hölder continuous potential $\psi$ with oscillation smaller than $\delta_0$.
\item If $\sup\{h_{\mu}(f)\,;\,\mu\in\cm^1(\bar{f}\,)\setminus\ce(f)\}<h(\ce(f))$ then
$\exists\delta_0>0$ such that  $f$ has a unique equilibrium state $\mu_{\psi}$ for any Hölder continuous potential $\psi$ with oscillation smaller than $\delta_0$.
\end{enumerate}
\end{maintheorem}

The proof of the previous results will appear in in Sections~\ref{sec:proofsmain1} and ~\ref{sec:proofsmain2}.
We should mention that for singular maps, as in the expanding Lorenz maps, there may exist invariant measures of arbitrarily 
large topological entropy whose top Lyapunov exponent is $+\infty$ (cf. \cite{PP22}).
This, together with the previous results, suggest that the hypothesis that $\cc$ is non-singular (i.e. 
derivatives are uniformly bounded above on $M\setminus \cc$) may be optimal.  

\smallskip
The previous results have an interesting application to the study of phase transitions for quadratic maps.
For a real parameter $c$, consider the family of quadratic polynomials $f_c(x)=x^2+c$, $x\in \mathbb R$. If $c$ is close to $-2$
then the interval $I_c=[c,f_c(c)]$ is invariant by $f_c$ and the map restricted to the interval $I_c$ is topologically transitive.
The quadratic map $f_c$ is called {Collet-Eckmann}  if 
\begin{equation}\label{eq:CEq}
\chi(c):= \liminf_{n\to\infty} \frac1n \log |(f_c^n)'(f_c(0))| > 0
\end{equation}
It is known that~\eqref{eq:CEq} implies that all $f_c$-invariant probability measures have positive Lyapunov exponent (cf. 
\cite{PRS03}).
In consequence, every continuous potential satisfies ~\eqref{Equationhouh567aai}, hence it is an expanding potential.
Nevertheless, there exist Collet-Eckmann quadratic maps $f_c$ having first order phase transitions for the piecewise Hölder continuous geometric potential \cite{CRL}: 
there exists $t_*>0$ so that the pressure function $\mathbb R\ni t \mapsto \Ptop(f,-t\log |f_c'|)$ is not differentiable at $t=t_*$.
Our previous results ensure that the unboundedness of the geometric potential is essential. More precisely, combining Theorem~\ref{TheoOLDCorollaryMAINknlielMI} with the known equivalence between Gateaux differentiability of pressure function and the existence and uniqueness of equilibrium states
(see e.g. \cite{BCMV}) we obtain the following:

\begin{maincorollary}\label{CorollaryCE}
Let $f_c$ be a Collet-Eckmann quadratic map and $\varphi: I_c \to \mathbb R$ be a Hölder continuous potential. For every $t\in \mathbb R$
there exists a unique equilibrium state $\mu_{t\varphi}$ for $f_c$ with respect to the potential $t\varphi$. 
In particular the pressure function $\mathbb R \ni t \mapsto \Ptop(f_c, t\varphi)$ is $C^1$-smooth.
\end{maincorollary}

\smallskip

Let us derive an application of the previous results in a context of multidimensional non-uniformly expanding maps. 
In \cite{Vi}, Viana introduces an open class of $C^3$ maps $f:J_f\to J_f$, where $J_f$ is a compact subset of $S^1\times\RR$ containing $S^1\times\{0\}$  in its interior, the critical set $\cc_f:=(\det Df)^{-1}(0)$ is non-flat and it is a compact manifold $C^2$ close to $S^1 \times\{0\}$, and Lebesgue almost every $x\in J_f$ has all its Lyapunov exponents positive and $\log d <\sup\{h_{\mu}(f) \colon \mu\in\cm_{erg}^1(f)\setminus\ce(f)\}<h_{top}(f)$, for some $d\ge 16$ (see Section~\ref{SectionVmaps} for details), where $\cm_{erg}^1(f)\subset\cm^1(f)$ is the set of all ergodic $f$-invariant probability measures and $\cm^1(f)$ is the set of all $f$-invariant probability measures of $f$.
Thus, as a consequence of our previous main results we obtain the following:

\begin{maincorollary}\label{CorollaryVIANA}
Let $f:J_f\to J_f$ be a Viana map. If $\varphi$ is a Hölder continuous potential such that 
\begin{enumerate}
\item $\sup\varphi-\inf\varphi<h_{top}(f)-\log d$ or 
\item $\int\varphi d\mu<P(f,\varphi)-\log d$ for every $\mu\in\cm^1(f)$,
\end{enumerate}
then $f$ has a unique equilibrium state $\mu_{\varphi}$ with respect to $\varphi$.
Moreover, $\mu_{\varphi}\in\ce(f)$ and $\supp\mu_{\varphi}=J_f$.
In particular, $f$ has a unique measure of maximal entropy and it is an expanding measure. 
\end{maincorollary}

\section{Basic definitions}\label{sec:def}

Let $\XX=(\XX,\dist)$ be a {\bf\em separable Baire metric space} 
and denote the {\bf\em ball of radius $r>0$ 
centered at $p\in\XX$} by $B_r(p)=\{y\in\XX\,:\,\dist(x,y)<r\}$.
Assume that for every $x\in\XX$ there is $\gamma_x>0$ such that 
\begin{equation}\label{BasicH}
  \overline{B_{\gamma_x}(x)}\text{ is compact and } B_{\varepsilon}(x)\text{ is connected and for every }0<\varepsilon\le\gamma_x.
\end{equation}

We should note that the hypothesis \eqref{BasicH} if necessary to construct {\em zooming nested ball} $B_{r}^{\star}(x)$ (see Definition~5.9 and Theorem~4 of \cite{Pi11}) which is used  to construct zooming return maps (Definition~\ref{def:return} in Section~\ref{Sectiongfdfghj}) well adapted to a given zooming measure (Proposition~\ref{Propositioiuiouiu877} in Section~\ref{Sectionjbutgyui6}).

A map $f:\XX\setminus\cc\to\XX$ is 
{\bf\em bi-Lipschitz local homeomorphism} if the following three below condition hold. 
\begin{enumerate}
\item[(I)] The set $\cc$, called the {\bf\em critical/singular} set of $f$, is a closed set with empty interior.
\item[(III)] For each $p\in\XX\setminus\cc$ there are $r\in(0,\gamma_p)$ and $K=K(p)\ge1$ such that $f(B_r(p))$ is an open subset of $\XX$ and 
$$
K^{-1} \dist(x,y)\le \dist(f(x),f(y))\le K \dist(x,y) \quad\text{for every} \; x,y\in B_r(p).
$$
\item[(III)]\label{Cond3f} $\#f^{-1}(x)<+\infty$ for every $x\in\XX$.
\end{enumerate}

Let $2^{\XX}$ be the power set of $\XX$, that is, the set for all subsets of $\XX$, including the empty set.
Define $f^*:2^{\XX}\to2^{\XX}$ by
$$f^*(U)=\begin{cases}
	\emptyset & \text{ if }U\subset\cc\\
	f(U\setminus\cc)=\{f(x)\,:\,x\in U\setminus\cc\} & \text{ if }U\not\subset\cc
\end{cases}.$$
We say that $U\subset\XX$ is {\bf\em forward invariant} if $f^*(U)\subset U$, and it is called {\bf\em backward invariant} if $f^{-1}(U)=U$.
Let $\co_f^+(U)=\bigcup_{n\ge0}(f^*)^n(U)$ be the saturation of $U\subset\XX$ by the forward orbit of $f^*$ and let
$\co_f^-(U)=\bigcup_{n\ge0} f^{-n}(U)$ be the backward saturated set of $U$.
For short, we write $(f^*)^n(x)$, $f^{-n}(x)$, $\co_f^+(x)$ and $\co_f^-(x)$ instead of $(f^*)^n(\{x\})$, $f^{-n}(\{x\})$, $\co_f^+(\{x\})$ and $\co_f^-(\{x\})$ respectively.
The omega-limit of a point $x$, $\omega_f(x)$, is set of accumulation points of the forward orbit of $x\in\XX$, that is, 
$$\omega_f(x)=\bigcap_{n\ge0}\overline{\co_f^+((f^*)^n(x))}.$$

Let  $\cm(\XX)$ be the set of $\sigma$-finite  Borel measures on $\XX$ and, as before in Section~\ref{SectionSMS}, the set of all $f$-invariant Borel probability measures is denoted by $\cm^1(f)$ and $\cm_{erg}^1(f)$ is the of all ergodic elements of $\cm^1(f)$.
It is a standard fact that if $f$ is continuous then the set $\cm^1(f)$ is locally compact in the weak$^*$ topology.

\subsection{Zooming measures}
Here the concept of non-uniform expansion will be topological and defined in terms of sequences of functions that determine rates of backward contraction.
Assume throughout that $f:\XX\setminus\cc\to\XX$ is 
a bi-Lipschitz local homeomorphism with $$\sup\{h_{\mu}(f)\,;\,\mu\in\cm^1(f)\}<+\infty.$$ 

A {\bf\em zooming contraction} is a sequence $\alpha=\{\alpha_n\}_{n\in\mathbb N}$ of functions $\alpha_n \colon [0,+\infty) \to [0,+\infty)$ satisfying, for every $m,n \in \NN$, the following conditions: (i) $\alpha_n(r)< r$, (ii) $\alpha_n(r)\le\alpha_n(r')$ whenever $r\le r'$, (iii) $\alpha_n \circ \alpha_{m}(r) \le \alpha_{m+n}(r) $
and (iv) $\sup_{r\in[0,1]} \sum_{n\ge 1} \alpha_n(r)<\infty$.

A zooming contraction $\alpha=\{\alpha_n\}_{n\in\NN}$ is called {\bf\em exponential} if $\alpha_n(r)= e^{-\lambda n}r$ for some $\lambda>0$, and it is called {\bf\em Lipschitz} if $\alpha_n(r)= a_n r$ for some real valued sequence $\{a_n\}_{n\in \mathbb N}$. In particular, every exponential zooming contraction is Lipschitz. 

\begin{Definition}
Given a zooming contraction $\alpha=\{\alpha_n\}_{n\in\NN}$, $\delta>0$ and $\ell\in\NN$, we say that $n$ is a {\bf\em $(\alpha,\delta,\ell)$-zooming time} for a point $p\in\XX$ if there exists
an open neighborhood $V_n(\alpha,\delta,\ell)(p)$ of $p$ such that $f^{\ell\,n}: V_n(\alpha,\delta,\ell)(p) \to B_{\delta}(f^{\ell\,n}(p))$ is a homeomorphism that extends
continuously to the boundary, and
\begin{equation}\label{eq:zoomcontr}
\dist(f^{\ell\,j}(x),f^{\ell\,j}(y)) \le \alpha_{n-j} ( \dist(f^{\ell\,n}(x),f^{\ell\,n}(y)) ),
	\quad \forall \;0 \le j\le n-1 \text{ and } x,y\in V_n(\alpha,\delta,\ell)(p).
\end{equation}
We refer to the neighborhood $V_n(\alpha,\delta,\ell)(p)$ as a {\bf\em $(\alpha,\delta,\ell)$-zooming pre-ball} of order $n$ center on $p$ and $B_{\delta}(f^{\ell\,n}(p))$ is called a {\bf\em $(\alpha,\delta,\ell)$-zooming ball}.
We denote by $\cz_n(\alpha,\delta,\ell)$ the set of points having $n$ as a $(\alpha,\delta,\ell)$-zooming time.
\end{Definition}

The notion of $(\alpha,\delta,\ell)$-zooming time is a generalization of the concept of hyperbolic times 
(see e.g. \cite{ABV,Pi11}) which is defined for topological dynamical systems and where
the contraction rate by inverse branches need not be exponential.
Here $\alpha$ refers to the zooming contraction, 
$\delta$ is the scale of growth and the last term $\ell$ refers that the iteration is by the local homeomorphism $f^\ell$. 
It is also worth noticing that the contraction rates in ~\eqref{eq:zoomcontr} depend exclusively on the distance between 
the iterates of the points $x,y$ and not on the point $p$ at the center of the ball.

\begin{Definition}
Let $\mu$ be an $f$-invariant and $\sigma$-finite Borel measure, $\alpha$ be a zooming contraction, $\delta>0$ and 
$\ell\in\NN$.
We say that $\mu$ is a {\bf\em $(\alpha,\delta,\ell)$-weak zooming measure} if 
$$
\mu\Big(\XX\setminus\limsup_{n\to\infty}\cz_n(\alpha,\delta,\ell)\Big)=0.
$$
If, additionally, 
$$\limsup_{n\to\infty}\frac{1}{n}\#\{1\le j\le n\,:\,x\in\cz_n(\alpha,\delta,\ell)\}>0 \quad\text{for $\mu$-a.e. $x\in\XX$}
$$
we say that $\mu$ is a {\bf\em $(\alpha,\delta,\ell)$-zooming measure}. 
The set of all $(\alpha,\delta,\ell)$-zooming Borel $f$-invariant probability measures will be denoted by $\cm_{\cZ}^1(\alpha,\delta,\ell)$. 
\end{Definition}

For simplicity, we shall say that $\mu$ is a {\bf\em zooming measure} if it is a $(\alpha,\delta,\ell)$-zooming measure for some zooming contraction $\alpha$ and constants $\delta>0$ and $\ell\in\NN$.
For each $\ell\in\NN$, let $\ce(f,\ell)$ be the set of all $f$-invariant probability measures which are $(\alpha,\delta,\ell)$-zooming for some exponential zooming contraction $\alpha$ and $\delta>0$.
A zooming measure with an exponential zooming contraction is called an {\bf\em expanding measure}. The set of {\bf\em all expanding  probability measures} is denoted by  $$\ce(f)=\bigcup_{\ell\in\NN}\ce(f,\ell)\subset\cm^1(f).$$

As in Section~\ref{SectionSMS} for maps on manifolds, we are assuming throughout the text that
\begin{equation}\label{Eqfinitentexp}
  h(\ce(f)):=\sup\{h_{\mu}(f)\,;\,\mu\in\ce(f)\}<+\infty.
\end{equation}

\subsection{Induced maps and measures}\label{Sectiongfdfghj}

An {\bf\em induced map} is a measurable map $F:A\to B$ where $A, B\subset\XX$ and $F$ is given by $F(x)=f^{R(x)}(x)$ for some measurable map $R:A\to\NN$ ($R$ is called the {\bf\em induced time} of $F$). While there exists no special requirement on the sets $A$ and $B$,  it is common that $A\subset B$ in some folkore constructions of induced maps.  

In order to consider the thermodynamic formalism of induced maps it is also needed to induce potentials.
More precisely, given an induced map $F:A\to B$ with induced time $R$, to each potential $\varphi:\XX\to\RR$  
one can associate the potential $\overline{\varphi}:A\to\RR$ given by $\overline{\varphi}(x)=\sum_{j=0}^{R(x)-1}\varphi\circ f^j(x)$. 
We shall say that $\overline\varphi$ is the {\bf\em $F$-lift} of $\varphi$.

An ergodic $f$-invariant probability measure $\mu$ is called {\bf\em $F$-liftable} if there exists an $F$-invariant probability measure $\nu\ll\mu$ (called the {\bf\em $F$-lift} of $\mu$) such that $\int R \,d\nu<+\infty$.
As not every $f$-invariant probability measure is liftable to a specific inducing scheme, we denote the set all $F$-liftable $f$ invariant probability measures by $\cm^1(f,F)$.
It is well known that if $\nu$ is the $F$-lift of $\mu$ then
\begin{equation}\label{Eq78987yh}
\mu=\frac{1}{\int R \,d\nu}\sum_{n\ge1}\sum_{j=0}^{n-1} \; f^j_*\big(\nu|_{\{R=n\}}\big)
	=\frac{1}{\int R \,d\nu}\; \sum_{j\ge0}f^j_*\big(\nu|_{\{R>j\}}\big)
\end{equation}

A {\bf\em full induced Markov map}
is a triple $(F,B,\cp)$ where $B\subset\XX$ is an open set, $\cp$ is a (countable) collection of disjoints open subsets of $B$ and  $F:A:=\bigcup_{P\in\cp}P\to B$ is an $f$ induced map satisfying: 
\begin{enumerate}
	\item for each $P\in\cp$, $F|_P$ is a homeomorphism  between $P$ and $B$ and it can be extended to a homomorphism sending $\overline{P}$ onto $\overline{B}$;
	\item $\lim_{n\to\infty} \diam(\cp_n(x))=0$ for every $x\in\bigcap_{n\ge1}F^{-n}(B)$, where $\cp_n=\bigvee_{j=0}^{n-1}F^{-j}(\cp)$ and $\cp_n(x)$ is the element of $\cp_n$ containing $x$.
\end{enumerate}

\smallskip
The following concept will play a key role in the construction of equilibrium states for induced maps.
A {\bf\em mass distribution} on $\cp$ is a map $m:\cp\to[0,1]$ such that $\sum_{P\in\cp}m(P)=1$.
Every such mass distribution $m$ generates an {\em $F$-invariant probability measure} $\mu$ (generated by the mass distribution $m$) which is the {\bf\em $F$-invariant Bernoulli probability measure} defined by $$\mu(P_1\cap F^{-1}(P_2)\cap\cdots\cap F^{n-1}(P_n))=m(P_1) m(P_2)\cdots m(P_n),$$ for each $P_1\cap F^{-1}(P_2)\cap\cdots\cap F^{n-1}(P_n)\in\bigvee_{j=0}^{n-1}F^{-j}(\cp)$.

\begin{Definition}[$(\alpha,\delta,\ell)$-zooming return map] \label{def:return}
A full induced map $(F,B,\cp)$ with induced time $R$ is called a {\bf\em $(\alpha,\delta,\ell)$-zooming return map} if 
\begin{enumerate}

	\item $1\le R(x)\le\min \big\{n\in\NN\,:\,x\in\cz_n(\alpha,\delta,\ell)\text{ and }f^{\ell\,n}(x)\in B\big\}$ for every $\,x\in \bigcup_{P\in\cp}P$; 
	\item for each $P\in\cp$
there is an $(\alpha,\delta,\ell)$-pre-ball $V_n(\alpha,\delta,\ell)(x_P)$ with respect to $f$ 
such that 
$$
P=(f^{\ell\,n}|_{V_n(\alpha,\delta,\ell)(x_P)})^{-1}(B) \subset B \subset B_{\delta}(f^{\ell\,n}(x_P)),
$$
where $n={R}(P)$. In particular,
$$\dist(f^{\ell\,j}(x),f^{\ell\,j}(y))\le\alpha_{n-j}(\dist(F(x),F(y)))\,\,\,\forall\,x,y\in P\text{ and }0\le j<n.$$
\end{enumerate}
\end{Definition}

\section{Zooming measures on strongly transitive spaces}

 In this section we assume some familiarity of the reader with 
the notions of zooming measures and inducing schemes, and will refer the reader to \cite{Pi11} for some definitions and proofs.

\subsection{Zooming sets, pre-images and induced maps}
The first lemma relates zooming times for $f$ with zooming times for its iterates. As usual, 
$\limsup_n\cz_n({\alpha},\delta,\ell)$ stands for the set of points with infinitely many 
$({\alpha},\delta,\ell)$-zooming times.

\begin{Lemma}\label{Lemma2121112}
Fix $\ell\in\NN$, $\delta>0$ and a Lipschitz zooming contraction $\alpha=\{\alpha_n\}_n$, $\alpha_n(r)=a_n r$ with $\sum_n(a_n)^{1/2}<+\infty$. 
If $p\in\limsup_n\cz_n(\alpha,\delta,1)$ then there exists $q\in\{p,\cdots,f^{\ell-1}(p)\}$ such that $\co_{f^{\ell}}^-(q)\subset\limsup_n\cz_n(\widetilde{\alpha},\delta,\ell)$, where $\widetilde{\alpha}=\{\widetilde{\alpha}_n\}_n$ with $\widetilde{\alpha}_n(r)=(a_{\ell n})^{1/2}r$. 
Furthermore,
if $p\in\limsup_n\cz_n(\alpha,\delta,\ell)$ then $\co_{f^{\ell}}^-(p)\subset\limsup_n\cz_n(\widetilde{\alpha},\delta,\ell)$, 
\end{Lemma}

\begin{proof}
	Let $p\in \XX$ 
	be such that $\#\cz(p,1)=+\infty$, where $\cz(p,j):=\{n\in\NN\,:\,p\in\cz_n(\{\alpha_{j n}\}_n,\delta,j)\}$ $\forall j\in\NN$. 
Given $n\in\cz(p,1)$, as $n=a \ell +b$ with $0\le b<\ell$, we have that $f^b(p)\in\cz_{a\ell}(\{\alpha_n\}_n,\delta,1)\subset\cz_a(\{\alpha_{\ell n}\}_n,\delta,\ell)$.
As $\#\cz(p,1)=+\infty$, we can conclude by the pigeonhole principle that exists $q\in\{p,\cdots,f^{\ell-1}(p)\}$ such that $\#\cz(q,\ell)=+\infty$.

Write $g=f^{\ell}$. 
Given $y\in\co_{g}^-(q)$, say $g^{t}(y)=q$, and using that $f\mid_{\XX\setminus \mathcal C}$ is a  bi-Lipschitz local homeomorphism, one guarantees that there exists 
 $K=K(y)\ge1$ and $\delta_y>0$ such that $g^t(B_{\delta_y}(y))$ is an open set and 
\begin{equation}\label{Equation00010101x}
  \dist(x,z)/K\le \dist(f^j(x),f^j(z))\le K \dist(g^t(x),g^t(z))
\end{equation}
for every $x,z\in B_{\delta_y}(y)$ and every $0\le j<\ell t$. 
Noting that $B_{\delta_y/K}(q)\subset g^{t}(B_{\delta_y}(y))$, if $n_0\ge1$ is large enough then $\overline{V_n(\{\alpha_{\ell n}\}_n,\delta,\ell)(q)}\subset\overline{B_{\delta_y/K}(q)}\subset g^{t}(B_{\delta_y}(y))$ whenever $\cz(q,\ell)\ni n\ge n_0$.

Hence, taking $n_1\ge n_0$ large enough, using the properties of zooming contraction and (\ref{Equation00010101x}), we get that $Ka_{\ell n}<(a_{\ell(n+t-j)})^{1/2}$ for every $0\le j<t$ and $n\ge n_1$.
Moreover, as $a_i<1$ $\forall i$, we have that $a_{\ell (n+t-j)}\le (a_{\ell(n+t-j)})^{1/2}$ $\forall t\le j<n+t$, proving that 
$$\dist(g^j(x),g^j(z))\le (a_{\ell(n+t-j)})^{1/2}\dist(g^{n+t}(x),g^{n+t}(z))$$
for every $x,z\in V_{n+t}'(y)$ $:=$ $(g^t|_{B_{\delta_y}(y)})^{-1}(V_n(\{\alpha_{\ell n}\}_n,\delta,\ell)(q))$, $0\le j<n+t$ and $n_1\le n\in\cz(q,\ell)$.
In particular $V_{n+t}'(y)=V_{n+t}(\widetilde{\alpha},\delta,\ell)(y)$ and the latter means that 
$$
n_1\le n\in\cz(q,\ell)\; \implies\;  n+t\in\widetilde{\cz}(y,\ell):=\{i\in\NN\,:\,y\in\cz_{i}(\widetilde{\alpha},\delta,\ell)\}.
$$ 
Hence, $\#\widetilde{\cz}(y,\ell)=+\infty$ for every $y\in\co_g^-(q)$, proving that $\co_g^-(q)\subset\limsup_n\cz_n(\widetilde{\alpha},\delta,\ell)$, which is the first assertion in the lemma.

The second claim is a direct consequence of the previous argument. Indeed, if $p\in\limsup_n\cz_n(\alpha,\delta,\ell)$, we chose $q=p$ in the proof above and conclude that  $\co_{f^{\ell}}^-(p)\subset\limsup_n\cz_n(\widetilde{\alpha},\delta,\ell)$. 
This proves the lemma.
\end{proof}

\subsection{Fat-induced probability measures}

The classical uniformly expanding maps admit finite Markov partitions, through which it is possible 
to semiconjugate the dynamics to a subshift of finite type. In particular all invariant measures, hence the
equilibrium states, 
are associated to a fixed Markov structure. Thus, as non-uniformly expanding maps often have countable Markov structures,  
conceptually one can expect equilibrium states to be determined by these structures. 
Yet, while every invariant measure having only positive Lyapunov exponents has a Markov structure adapted to
it (see \cite{Pi11}), in order to characterize equilibrium states one needs a Markov structure that is adapted to 
a broad class of such measures, that we now define.

\begin{Definition}
A probability measure $\mu$ on $\XX$ is called 
a {\bf\em $(\alpha,\delta,\ell)$-zooming fat-induced probability measure} if $\mu$ is $f$-invariant, ergodic and there exists a $(\alpha,\delta,\ell)$-zooming return map $(F,B,\cp)$ and a $F$-lift $\nu$ of $\mu$ such that $\supp\nu=\overline{B}$.
\end{Definition}

{ In rough terms, zooming fat-induced probability measures are those that lift to a full supported 
probability measure on an induced 
zooming return map on a disk  (recall Definition~\ref{def:return}).}
Given $\ell\in\NN$, let $\ce^*(f,\ell)$ be the set of all $(\alpha,\delta,\ell)$-zooming fat-induced probability measures associated
to some exponential zooming contraction $\alpha$ and $\delta>0$. Denote by 
$$\ce^*(f)=\bigcup_{\ell\in\NN}\ce^*(f,\ell)$$ the set of all {\bf\em expanding fat-induced  probability measures}.

An induced map $F:A\to B$ is called {\bf\em orbit-coherent} if 
$$
		\co_f^+(x)\cap\co_f^+(y)\ne\emptyset\Longleftrightarrow\co_F^+(x)\cap\co_F^+(y)\ne\emptyset
$$
for every $x,y\in\bigcap_{j\ge0} F^{-j}(A)$.

\begin{Remark}\label{Remark9876543456789} 
	All the $(\alpha,\delta,1)$-zooming return maps  constructed in \cite{Pi11} are orbit-coherent. Indeed, the orbit coherence is asked explicitly in Theorem~1 in  \cite{Pi11} and this result is the one used to assure that an invariant probability measure is liftable to the induced maps constructed in that paper.
In \cite[Theorem~A]{Pi20}, the first author extended the lift result \cite[Theorem~1]{Pi11} to any induced map $F$ in a metric space, and 
proved that if $F$ is orbit-coherent then the 
$F$-lift is unique and ergodic (cf. \cite[Theorem~1]{Pi11}).
\end{Remark}

\subsection{Existence and liftability}\label{Sectionjbutgyui6}

 We are now in a position to state an instrumental liftability result, namely that under a mild topological
assumption every zooming measures (with Lipschitz zooming contractions) is liftable to some orbit-coherent  zooming
return map (depending on the measure) so that the domain is dense in the range of the induced map (see item (4) below). More precisely:

\begin{Proposition}\label{Propositioiuiouiu877}
Let $\mu\in\cm^1(f)$ be an ergodic $(\alpha,\delta,\ell)$-zooming measure, for some $\ell\in\NN$, $\delta>0$ and a Lipschitz zooming contraction $\alpha=\{\alpha_n\}_n$, $\alpha_n(r)=a_n r$ with $\sum_n(a_n)^{1/2}<+\infty$.
If $f^{\ell}$ is strongly transitive and $\widetilde{\alpha}=\{\widetilde{\alpha}_n\}_n$ with $\widetilde{\alpha}_n(r)=(a_{\ell n})^{1/2}r$ then, given any sufficiently small $0<\varepsilon<\delta/2$, there is a $(\widetilde{\alpha},\delta,\ell)$-zooming return map $(F,B,\cp)$ such that
\begin{enumerate}
\item $F$ is orbit-coherent;
\item $F:A\to B$, where $A=\bigcup_{P\in\cp}P$, $B$ is a connected open set with $B_{\varepsilon/2}(p)\subset B\subset B_{\varepsilon}(p)$ for some {$p\in\XX$}; 
\item $\#\{P\in\cp\,:\,R(P)=n\}<+\infty$ for every $n\in\NN$, where $R$ is the inducing time of $F$;
\item $A$ is an open and dense subset of $B$;
\item $\mu$ has a unique $F$-lift $\nu$. 
\item $\nu$ is $F$-ergodic, { $\nu\ll \mu\mid_B$ and the Radon-Nikodym derivative $\frac{d\nu}{d\mu\mid_B}$  is bounded.}
\end{enumerate}
\end{Proposition}
\begin{proof} Let $g=f^{\ell}$. As $\mu$ is ergodic and $f$-invariant, it is well known that there is $1\le s\le\ell$, with $\ell/s\in\NN$, such that $\XX$ can be decomposed on $s$ $\mu$-ergodic components with respect to $g$. More
precisely, there are pairwise disjoint sets $U_1,\cdots,U_s\subset\XX$ so that $\mu(U_j)>0$, $\XX=\bigcup_jU_j\mod\mu$ and $\mu|_{U_j}$ is ergodic with respect to $g$. 
Let $U$ be one of the $\mu$-ergodic components with respect to $g$ and set $\mu'=\frac{1}{\mu(U)}\mu|_{U}$.
Note that $\mu'$ is a $(\alpha,\delta,1)$-zooming probability measure for $g$.
Given $x\in\XX$ and $V\subset\XX$, let  
$$
\tau_{x,\alpha,\delta}(V)
	=\limsup_{n\to\infty}\frac{1}{n}\#\Big\{1\le j\le n\,:\,x\in\cz_j(\alpha,\delta,\ell)\text{ and }g^j(x)  \in V\Big\}
$$
and 
$$
\omega_{\alpha,\delta,\ell}(x)=\big\{y\in\XX\,:\,\tau_{x,\alpha,\delta}(B_{\varepsilon}(y))>0\text{ for every }\varepsilon>0\big\}.
$$
It follows from Lemma~3.9 of \cite{Pi11} that there exists a compact set $\ca\subset\supp\mu'$ such that $\omega_{\alpha,\delta,\ell}(x)=\ca$ for {$\mu'$-almost every $x\in\XX$.} 
Choose a point $p\in\ca$.

It follows from item (\ref{Cond3f}) of the definition of a {\em bi-Lipschitz local homeomorphism} that $\#f^{-\ell}(x)<\infty$ for every $x\in\XX$ and so,  $g$ is a backward separated map in the sense of \cite[Definition~5.11]{Pi11}.
Thus, for every $0<\varepsilon<\delta/2$  sufficiently small, the $(\widetilde{\alpha},\delta,\ell)$-zooming nested ball $B_{\varepsilon}^{\star}(p)$ is a well defined open connected set containing $B_{\varepsilon/2}(p)$ (see Definition~5.9 and Lemma~5.12 of \cite{Pi11}).
Observe that, by construction, $p\in\omega_{\alpha,\delta,\ell}(x)\subset\omega_g(x)$ for $\mu'$-almost every $x$.
In particular $\mu'(B_{\varepsilon/2}(p))>0$
and one can choose $q\in B_{\varepsilon/2}(p)$ such that $p\in\omega_{\alpha,\delta,\ell}(q)$. 

It follows from $f^{\ell}$ being strongly transitive and Lemma~\ref{Lemma2121112} that $U\subset \overline{\co_g^-(q)}$ 
and $\co_g^-(q)\subset\limsup_n\cz_n(\widetilde{\alpha},\delta,\ell)$.
Therefore, using that 
$$
\cz_n(\alpha,\delta,\ell)\subset\cz_n(\widetilde{\alpha},\delta,\ell) \qquad \forall\,n\in\NN,
$$ 
$\overline{\co_g^-(q)}\supset B_{\varepsilon/2}(p)$  and $p\in\omega_{\widetilde{\alpha},\delta,\ell}(y)$ for every $y\in\co_g^-(q)$, we get that 
$$
\Lambda:=\{x\in B_{\varepsilon/2}(p)\,:\,p\in\omega_{\widetilde{\alpha},\delta,\ell}(x)\}
$$ 
is  dense on $B_{\varepsilon/2}(p)$ and $\mu'(\Lambda)=\mu'(B_{\varepsilon/2}(p))>0$.

From Theorem~4 (page 914) of \cite{Pi11}, the induced Markov map $F:A\subset B_{\varepsilon}^{\star}(p)\to B_{\varepsilon}^{\star}(p)$ associated to the ``first $(\widetilde{\alpha},\delta,\ell)$-zooming return time to $B_{\varepsilon}^*(p)$ with respect to $g$''  is a $(\widetilde{\alpha},\delta,\ell)$-zooming return map $(F,B,\cp)$, where $B=B_{\varepsilon}^*(p)$, 
$$
A:=\bigcup_{n\in\NN,\,y\in\cw_n}(f^n|_{V_n(\widetilde{\alpha},\delta,\ell)(y)})^{-1}(B),$$
where 
$$\cw_n=\{y\in\cz_n(\widetilde{\alpha},\delta,\ell)\text{ such that }f^n(V_n(\widetilde{\alpha},\delta,\ell)(y))\supset B\}.
$$
 and $\cp$ is the collection of connected components of $A$.

Notice that $A\supset\{x\in B\,:\, { \omega_{\widetilde{\alpha},\delta,\ell}(x)\cap B  \ne \emptyset} \}$, hence $\overline{A}=\overline{B}$.
Furthermore, Theorem~4 in \cite{Pi11} guarantees the existence of a $F$-invariant measure $\nu_0\le\mu'\ll\mu$, with $\nu_0(\XX)>0$ and $\int R \,d\nu_0<+\infty$.
In other words, $\mu$ is $F$-liftable to { $\nu:=\frac{1}{\nu_0(\XX)}\nu_0\le C\mu|_{B}$, where $C=1/\nu_0(\XX)$.}

As $\#f^{-1}(x)<+\infty$ for every $x\in\XX$ and $f$ is measurable, it follows from a result due to Purves \cite{Pu} that $f$ is bimeasurable, hence any image of a measurable set is measurable. 
As $F$ is orbit-coherent (see Remark~\ref{Remark9876543456789} above), it follows from Theorem~A of \cite{Pi20} that $\nu$ is $F$-ergodic, it is the unique $F$-lift of $\mu$.
For last, item (3) follows from the fact that $\#g^{-1}(x)<+\infty$ for every $x\in\XX$.
\end{proof}

\begin{Lemma}\label{Lemma001009110}
Suppose that $f$ is transitive and $(F,B,\cp)$ is a $(\alpha,\delta,1)$-zooming return map such that $B\subset B_{\varepsilon}(p)$ for some $p\in\XX$ and $\varepsilon\in(0,\delta/2)$.
If $\nu$ is a $F$-invariant, ergodic probability measure with $\int R\,d\nu<+\infty$ then $\mu=\frac{1}{\int R \,d\nu}\sum_{n\ge1}\sum_{j=0}^{n-1}f^j_*(\nu|_{\{R=n\}}),$ is a $f$-invariant, ergodic $(\alpha,\delta/2,1)$-zooming probability measure.
Moreover, if $\interior(\supp\nu)\ne\emptyset$ then $\supp\mu=\XX$.
\end{Lemma}
\begin{proof}
Let $\nu$ be a $F$-invariant, ergodic probability measure such that $\int R \,d\nu<+\infty$.
It is clear that $\mu=\frac{1}{\int R \,d\nu}\sum_{n\ge1}\sum_{j=0}^{n-1}f^j_*(\nu|_{\{R=n\}}),$ is a $f$-invariant, ergodic
probability measure.
It is also clear that, as $f$ is transitive, then $\supp\mu=\XX$ when $\interior(\supp\nu)\ne\emptyset$.
It remains to prove that $\mu$ is a zooming measure.

Given $x\in\bigcap_jF^{-j}(B)$ and $n\ge1$ there are $x_0,\cdots,x_{n-1}\in\XX$ such that 
$$
x_j\in\cz_{R\circ F^{j}(x)}(\alpha,\delta,1) \quad\text{ and }\quad F^{j}(x)\in V_{R\circ F^{j}(x)}(\alpha,\delta,1)(x_j)
$$ 
for every $0\le j<n$.

In consequence, as $F^n(x)\in B \subset B_{\varepsilon}(p)\subset B_{\delta/2}(p)$,
one deduces that 
\begin{align*} 
V_{s_n(x)}'(x) & :=(f^{R(x)}|_{V_{R(x)}(\alpha,\delta,1)(x_0)})^{-1}\circ\cdots\circ(f^{R(F^{n-1}(x))}|_{V_{R(F^{n-1}(x))}(\alpha,\delta,1)(x_{n-1})})^{-1}\big(B_{\delta/2}(F^n(x))) \\
		& =(f^{R(x)}|_{V_{R(x)}(\alpha,\delta,1)(x_0)})^{-1}\circ\cdots\circ(f^{R(F^{n-1}(x))}|_{V_{R(F^{n-1}(x))}(\alpha,
		\delta,1)(x_{n-1})})^{-1}\big(B_{\delta/2}(f^{s_n(x)}(x)))
\end{align*}
is the $(\alpha,\delta/2,1)$-zooming pre-ball of order $s_n(x)$ centered at $x$, that is
$V_{s_n(x)}'(x)=V_{s_n(x)}(\alpha,\delta/2,1)(x)$, where  $s_n(x)=\sum_{j=0}^{n-1}R\circ F^j(x)$.
Setting 
$J(x)=\big\{s_n(x)\,:\,n\in\NN\},$
we get that 
$$
\lim_{n\to\infty}\frac{1}{n}\#\{1\le j\le n\,:\, j\in J(x)\}=\frac{\nu(B)}{\int R \,d\nu}=\frac{1}{\int R \,d\nu}>0.
$$
Now observe that $x\in\cz_n(\alpha,\delta/2,1)$ whenever $n\in J(x)$. 
This ensures that $\nu$-almost every $x\in B$ has positive frequency of $(\alpha,\delta/2,1)$-zooming times.
As $\nu\ll\mu$, we get that $\mu(B)>0$ and so, it follows from the ergodicity of $\mu$ that $\mu$-almost every point has positive frequency of $(\alpha,\delta/2,1)$-zooming times, proving that $\mu$ is a $(\alpha,\delta/2,1)$-zooming measure. 
\end{proof}

\subsection{A prelude to thermodynamic formalism}

The following instrumental result, saying that entropy and space averages of any expanding measure 
can be approximated by those of fat-induced zooming measures,  will give fat-induced zooming measures 
a key role in the thermodynamic formalism of expanding measures.

\begin{Lemma}\label{LemmaLiftFunctions}
Let $F:A\to B$ be an induced map with induced time $R$. Suppose that an ergodic $f$-invariant probability measure $\mu$ has a $F$-lift $\nu$ that is $F$ ergodic.
If $\overline{\varphi}$ is the $F$-lift of a measurable and $\mu$-integrable map $\varphi:\XX\to\RR$ then 
$$\int\varphi \,d\mu=\frac{\int\overline{\varphi} \,d\nu}{\int R \,d\nu}.$$
\end{Lemma}
\begin{proof} Note that $\nu(A_0)=1$, where $A_0:=\bigcap_{n\ge0} F^{-n}(\XX)\subset A$.
As $\nu\ll\mu$, it follows from Birkhoff's ergodic theorem that there exists $U\subset A_0$, with $\nu(U)=1$, such that $\lim_n\frac{1}{n}\sum_{j=0}^{n-1}\varphi\circ f^j(x)=\int\varphi \,d\mu$ and $\lim_n\frac{1}{n}\sum_{j=0}^{n-1}\overline{\varphi}\circ F^j(x)=\int \overline{\varphi} \,d\nu$ for every $x\in U$.
Setting, for $x\in U$, $r_n(x)=\sum_{j=0}^{n-1}R\circ F^j(x)$, it follows also from the ergodic theorem that 
$$\lim_{n\to\infty}\frac{1}{n}r_n(x)=\lim_{n\to\infty}\frac{1}{n}\sum_{j=0}^{n-1}R\circ F^j(x)=\int R \,d\nu.$$

So, for each $x\in U$ we get that  $$\sum_{j=0}^{r_n(x)-1}\varphi\circ f^j(x)=\sum_{j=0}^{n-1}\sum_{k=0}^{R(F^j(x))-1}\varphi\circ f^k(F^j(x))=\sum_{j=0}^{n-1} \overline{\varphi}\circ F^j(x)$$
and as a consequence, for $\mu$-almost every $x$, 
$$\int\varphi \,d\mu=\lim_{n\to\infty}\frac{1}{r_n(x)}\sum_{j=0}^{r_n(x)-1}\varphi\circ f^j(x)=\lim_{n\to\infty}\frac{\frac{1}{n}\sum_{j=0}^{n-1} \overline{\varphi}\circ F^j(x)}{\frac{1}{n}r_n(x)}=\frac{\int \overline{\varphi} \,d\nu}{\int R \,d\nu}.$$
\end{proof}

Let $\mathfrak{L}=\{\psi_1,\psi_2,\psi_3,\cdots\}\subset C^0(\XX,[0,1])$ be a countable set of Lipschitz function such that 
\begin{equation}\label{def:metricdM}
d(\nu,\eta):=\sum_{n\in\NN}\frac{1}{2^n}\bigg|\int\psi_n d\nu-\int\psi_n d\eta\bigg|\le1
\end{equation} 
define a metric on $\cm^1(\XX)$ compatible with weak* topology (see, for instance, Theorem Portmanteau at \cite{OV16}).
Given $n\in\NN$, let $C_n>0$ be a Lipschitz constant for $\psi_n$, hence $|\psi_n(x)-\psi_n(y)|\le C_n\dist(x,y)$ for any $x,y\in \XX$.

Due to the presence of the critical region $\cc$, it is natural to include in our analysis potentials that are regular just outside of $\cc$, which motivates the following definition.
\begin{Definition}\label{DefinitionPiecewiseHolder}
Given $C,\gamma>0$, we say that $\varphi:\XX\setminus\cc\to\RR$ is a {\em $(C,\gamma)$-piecewise Hölder continuous potential} if  $\sup_{x\in\XX\setminus\cc}|\varphi(x)|<+\infty$ and $$|\varphi(x)-\varphi(y)|\le C\dist(x,y)^{\gamma}$$ for every $x,y\in\XX\setminus\cc$ in a same connected component of $\XX\setminus\cc$.
\end{Definition}

A {\bf\em piecewise Hölder continuous potential} is a $(C,\gamma)$-piecewise Hölder continuous potential for some $C,\gamma>0$.

\begin{Proposition}\label{prop:reduction-III}
Suppose that $f^{\ell}$, $\ell\ge1$, is strongly transitive and 
let $\varphi:\XX\to\RR$ be a piecewise Hölder continuous potential. 
If $\mu\in\ce(f,\ell)$ is ergodic then, for any small $\varepsilon>0$ there exists  $\overline{\mu}\in\ce^*(f,\ell)$ such that:
\begin{enumerate}
\item  $|\int\varphi \, d\overline{\mu}|<\varepsilon$ when $\int \varphi \,d\mu=0$,
\item $\big|\frac{\int\varphi \, d\overline{\mu}}{\int\varphi \,d\mu}-1\big|<\varepsilon$ when $\int \varphi \,d\mu\ne0$,
\item  $h_{\overline{\mu}}(f)>(1-\varepsilon)h_{\mu}(f)$, 
\item $d(\mu,\overline{\mu})<2\varepsilon$.
\end{enumerate}
\end{Proposition}
\begin{proof}
Fix $\varepsilon>0$ and let $t\in\NN$ be such that $\sum_{n>t}2^{-n}<\varepsilon$. 
Suppose that $\mu$ is a $(\alpha,\delta,\ell)$-zooming measure, for some $\delta>0$ and a Lipschitz zooming contraction $\alpha=\{\alpha_n\}_n$ with $\alpha_n(r)=e^{-\lambda n} r$ and $\lambda>0$.
Let also $C,a>0$ be such that $|\varphi(x)-\varphi(y)|\le C \dist(x,y)^a$ for every $x,y\in\XX\setminus\cc$ in a same connected component of $\XX\setminus\cc$. 
Let $0<\tau<\delta/2$ be such that 
\begin{equation}\label{eq:CaM}
\frac{C_0}{1-e^{-\lambda a/2}}\tau^{a}<\frac{\varepsilon M}{8}, \quad\text{ where }\quad
M= \frac{1}{1+|\int \varphi \,d\mu|}\le1\text{ and }C_0=\max\{C,C_1,\cdots,C_t\},
\end{equation}
where the constants $C_i$ above are the Lipschitz constants of the functions $\psi_i$ used to define the metric $d$ in ~\eqref{def:metricdM}.
Now, let $(F,B,\cp)$ be the $(\widetilde{\alpha},\delta,\ell)$-zooming return map given by Proposition~\ref{Propositioiuiouiu877} with $\mbox{diam}(B)<\tau$ and   $\nu$ the $F$-lift of $\mu$.
Note that $\widetilde{\alpha}=\{e^{-\lambda n/2}r\}_n$, hence it is also an exponential zooming contraction.

Let  $A=\bigcup_{P\in\cp}P$ and $\overline{\varphi}(x)=\sum_{j=0}^{ R(x)/\ell -1}\varphi\circ f^{j\ell}(x)$.
For each $P\in\cp$ and $x,y\in P$,
\begin{align*}
|\overline{\varphi}(x)-\overline{\varphi}(y)|
	 & \le\sum_{j=0}^{ R(P)/\ell -1}|\varphi(f^{j\ell}(x))-\varphi(f^{j\ell}(y))|\le\sum_{j=0}^{ R(P)/\ell -1}C  \dist(f^{j\ell}(x),f^{j\ell}(y))^a
	 \\
	 & \le C\sum_{j=0}^{ R(P)/\ell -1} \widetilde{\alpha}_{ R(P)/\ell -j}(\dist(F(x),F(y)))^a
		\le C \, \frac{1}{1-e^{-\lambda a/2}}\, \dist(F(x),F(y))^{a}\\
	& \le C_0 \, \frac{1}{1-e^{-\lambda a/2}}\, \dist(F(x),F(y))^{a}.
\end{align*}
Analogously, we get that 
$$|\overline{\psi}_n(x)-\overline{\psi}_n(y)|\le C_n \, \frac{1}{1-e^{-\lambda/2}}\, \dist(F(x),F(y))\le C_0 \, \frac{1}{1-e^{-\lambda a/2}}\, \dist(F(x),F(y))^{a}$$ for every $1\le n\le t$, where $\overline{\psi}_n(x)=\sum_{j=0}^{ R(x)/\ell -1}\psi_n\circ f^{j\ell}(x)$
As $\diameter(B)<\tau$, the choice ~\eqref{eq:CaM} implies that
\begin{equation}\label{Equationjbhbprih39h}
\max\big\{\;  |\overline{\varphi}(x)-\overline{\varphi}(y)|\; ,\; |\overline{\psi}_n(x)-\overline{\psi}_n(y)| \; \big\}<\frac{\varepsilon M}{8},\;\;\;\forall x,y\in P, \;\; \forall P\in\cp\text{ and }1\le n\le t.
\end{equation}

{ The idea is now to use the mass distributions to approximate the measure $\mu$, both in its
average on $\varphi$ and entropy, by a fat-induced expanding measure.}
Write $\{n_1,n_2,n_3,\cdots\}=\{n\in\NN\,:\,\{R=n\}\ne\emptyset\}$, with $1\le n_1<n_2<n_3<\cdots$. 
For each $P\in\cp$ choose $x_P\in P$ and define   
$$
m_0(P)=
\begin{cases}\hspace{1.5cm}
0 & \text{ if }R(P)\notin\{n_1,n_2,\cdots\}\\
\frac{1}{W}\frac{2^{-n_j}}{\#\{Q\in\cp\,:\,R(Q)=n_j\}}  & \text{ if }R(P)=n_j,
\end{cases}
$$ 
where $W=\sum_{j=1}^{+\infty}\frac{2^{-n_j}}{\#\{Q\in\cp\,:\,R(Q)=n_j\}}\in(0,1]$.
For $0<\gamma<\gamma_{\varphi}$, where 
 \begin{equation}\label{eq:auxP}
\gamma_{\varphi}=\bigg(\frac{\varepsilon M}{4}\bigg)\bigg(\frac{1}{1+\sum_{P\in\cp}|\overline{\varphi}(x_P)|(\nu(P)+m_0(P))}\bigg)\bigg(\frac{W}{\sum_{j=1}^{+\infty}n_j2^{-n_j}}\bigg),
 \end{equation} 
let $m_{\gamma}$ be the mass distribution
$$
m_{\gamma}(P)=(1-\gamma)\,\nu(P)+\gamma \,m_0(P) \quad\text{ for every}\quad P\in\cp.
$$
Observe that the $F$-invariant, ergodic probability measure $\overline{\nu}_{\gamma}$ generated by $m$ verifies 
$$\int R \,d\overline{\nu}_{\gamma}\le(1-\gamma)\int R \,d\nu+\gamma W \sum_{j=1}^{+\infty} n_j 2^{-n_j}<(1-\gamma)\int R \,d\nu+2\gamma<+\infty.$$ 
The choice of $\overline{\nu}_{\gamma}$ above and the fact that $R\ge 1$ actually ensures that 
\begin{align*}
\bigg|\int R \,d\nu-\int R \,d\overline{\nu}_{\gamma}\bigg|
	& =\bigg|\gamma\int R \,d\nu-\gamma\sum_{j=1}^{+\infty}\frac{n_j2^{-n_j}}{W\#\{Q\in\cp\,;\,R(Q)=n_j\} }\bigg| \\
	& \le  \gamma \bigg( \frac{\sum_{j=1}^{+\infty}n_j2^{-n_j}}{W}  +\int R \,d\nu\bigg)
	\le\frac{\varepsilon M}{4}+ \frac{\varepsilon M}{4}\int R \,d\nu 
	\le \frac{\varepsilon M}{2}\int R \,d\nu.
\end{align*} 
As $\supp\overline{\nu}_{\gamma}=\overline{B}$, by construction the $f$-invariant, ergodic probability measure $\overline{\mu}_{\gamma}$ given by $$\overline{\mu}_{\gamma}=\frac{1}{\int R \,d\overline{\nu}}\sum_{j\ge0}f^j_*(\overline{\nu}|_{\{R>j\}})$$
is a $(\widetilde{\alpha},\delta,\ell)$-zooming fat-induced probability measure.
As $\widetilde{\alpha}$ is an exponential zooming contraction, we get by Lemma~\ref{Lemma001009110} that $\overline{\mu}_{\gamma}\in\ce^*(f,\ell)$.
Additionally, using (\ref{Equationjbhbprih39h}) we get 
$$
\max\bigg\{\bigg|\int\overline{\varphi} \,d\nu-\sum_{P\in\cp}\overline{\varphi}(x_P)\nu(P)\bigg|, \; 
\;\bigg|\int\overline{\varphi} \,d\overline{\nu}_{\gamma}-\sum_{P\in\cp}\overline{\varphi}(x_P)\overline{\nu}_{\gamma}(P)\bigg| \bigg\}
	 {  < \frac{\varepsilon M}{8}}.
$$
Hence, by triangular inequality,
\begin{align*}
\bigg|\int\overline{\varphi} \,d\nu-\int\overline{\varphi} \,d\overline{\nu}_{\gamma}\bigg|
	& \le \bigg|\int\overline{\varphi} \,d\nu-\sum_{P\in\cp}\overline{\varphi}(x_P)\nu(P)\bigg|+\bigg|\sum_{P\in\cp}\overline{\varphi}(x_P)\nu(P)-\sum_{P\in\cp}\overline{\varphi}(x_P)\overline{\nu}_{\gamma}(P)\bigg| \\
	& \qquad +\bigg|\sum_{P\in\cp}\overline{\varphi}(x_P)\overline{\nu}_{\gamma}(P)-\int\overline{\varphi} \,d\overline{\nu}_{\gamma}\bigg| \\
	&  \le { \frac{\varepsilon M}{4}} 
	+\bigg|\sum_{P\in\cp}\overline{\varphi}(x_P)\nu(P)-\sum_{P\in\cp}\overline{\varphi}(x_P)\overline{\nu}_{\gamma}(P)\bigg| \\
	&  \le { \frac{\varepsilon M}{4} 
	+\sum_{P\in\cp}\, |\overline{\varphi}(x_P)|\, \big|\nu(P)-m_{\gamma}(P)\big|}\\
	& = { \frac{\varepsilon M}{4} + \gamma \cdot \sum_{P\in\cp}\, |\overline{\varphi}(x_P)|\, \big|\nu(P)-m_0(P)\big|} 
	  \le { \frac{\varepsilon M}{2} }
\end{align*}

This can now be reformulated for the original dynamics as follows. 
As $F$ is orbit-coherent, it follows from \cite[Theorem~A]{Pi20} that $\nu$ is $F$ ergodic. Therefore that $\int \varphi d\mu=\int\overline{\varphi} \,d\nu/\int R \,d\nu$ and $\int \varphi \,d\overline{\mu}_{\gamma}=\int\overline{\varphi} \,d\overline{\nu}_{\gamma}/\int R \,d\overline{\nu}_{\gamma}$ (recall Lemma~\ref{LemmaLiftFunctions}) and triangular inequality
\begin{align*}
\bigg|\int \varphi \,d \mu-\int \varphi \,d\overline{\mu}_{\gamma}\bigg| 
	& \le 
	\bigg|\frac{\int \overline{\varphi} \,d\nu}{\int R \,d\nu}-\frac{\int \overline{\varphi} \,d\nu}{\int R \,d\overline{\nu}_{\gamma}}\bigg|
	+\bigg|\frac{\int \overline{\varphi} \,d\nu}{\int R \,d\overline{\nu}_{\gamma}}-\frac{\int \overline{\varphi} \,d\overline{\nu}_{\gamma}}{\int R \,d\overline{\nu}_{\gamma}}\bigg| \\
	& < \bigg|\frac{\int\overline{\varphi} \,d\nu}{\int R \,d\nu}\bigg|\bigg|\frac{\int R \,d\nu-\int R \,d\overline{\nu}_{\gamma}}{\int R \,d\overline{\nu}_{\gamma}}\bigg|+\frac{\varepsilon M}{\int R \,d\overline{\nu}_{\gamma}} \\
	& < \bigg(\frac{\varepsilon M}{2}\bigg)\bigg|\int\varphi \,d\mu\bigg|+\varepsilon M=\varepsilon M\bigg(\frac{1}{2}
	\bigg|\int\varphi \,d\mu\bigg|+1\bigg)<\varepsilon,
\end{align*}
proving items (1) and (2).

By the same reasoning, we get that, taking $0<\gamma<\min\{\gamma_{\varphi}, \gamma_{\psi_1},\cdots,\gamma_{\psi_t}\}$
$$\bigg|\int\psi_n \,d \mu-\int \psi_n \,d\overline{\mu}_{\gamma}\bigg|<\varepsilon\text{ for every }1\le n\le t,$$
and so, $d(\mu,\overline{\mu}_{\gamma})<\big(\varepsilon\sum_{n=1}^t2^{-n}\big)+\varepsilon<2\varepsilon$, proving item $(4)$.

Finally, it remains to compare the entropies of $\mu$ and $\overline{\mu}_{\gamma}$.
Take the function $H:[0,1]\to[0,1]$ given by $H(0)=0$ and $H(x)=-x\log x$ for $x>0$. 
 As $\cp$ is a generating  partition for $F$ and $\overline{\nu}_{\gamma}$ is a Bernoulli measure it follows that
$$
h_{\nu}(F)=\inf_{n\ge 1} \frac1n \sum_{P\in\cp^{(n)}}H(\nu(P)) 
	\le \sum_{P\in\cp}H(\nu(P)) \quad\text{and}\quad h_{\overline{\nu}_{\gamma}}(F)=\sum_{P\in\cp}H(m_{\gamma}(P))>0.
$$
In particular, as $H$ is concave,
$$h_{\overline{\nu}_{\gamma}}(F)\ge (1-\gamma)\bigg(\sum_{P\in\cp}H(\nu(P))\bigg)+\gamma\bigg(\sum_{P\in\cp}H(m_0(P))\bigg)>(1-\gamma)h_{\nu}(F).$$

Combining the previous estimates with Abramov's formula for the entropy of induced maps (cf. \cite[Theorem 5.1]{Zwei}) one concludes that 
\begin{align*}
h_{\overline{\mu}_{\gamma}}(f)
	& =\frac{h_{\overline{\nu}_{\gamma}}(F)}{\int R \,d\overline{\nu}_{\gamma}}
	>(1-\gamma)\frac{\int R \,d\nu}{\int R \,d\overline{\nu}_{\gamma}}\;\frac{h_{\nu}(F)}{\int R \,d\nu}
	>(1-\gamma)\bigg(1-\frac{\varepsilon M}{2}\bigg)\;\frac{h_{\nu}(F)}{\int R \,d\nu} \\
	& >\Big(1-\frac\varepsilon2\Big)^2 h_{\mu}(f)>(1-\varepsilon)h_{\mu}(f).
\end{align*}
This proves item (3) and completes the proof of the proposition. 
\end{proof}

{
We finish this subsection noting that Proposition~\ref{prop:reduction-III} paves the way to the thermodynamic formalism of expanding measures through the analysis of the fat-induced expanding measures. 
Indeed, defining as in Section~\ref{SectionSMS} the  expanding pressure by 
$$P_{\ce(f)}(\varphi)=\sup\bigg\{h_\mu(f)+\int\varphi \,d\mu\,:\,\mu\in\ce(f)\bigg\},$$
Proposition~\ref{prop:reduction-III} has the following immediate consequence.

\begin{Corollary}\label{Corollarykbiohg4k234}
If $f$ is strongly transitive then $$
P_{\ce(f)}(\varphi)=\sup\bigg\{h_\mu(f)+\int\varphi \,d\mu\,:\,\mu\in\ce^*(f)\bigg\}
$$
for every piecewise Hölder continuous potential $\varphi$.
\end{Corollary}
\begin{proof}
If $\ce(f)=\emptyset$ there are nothing to prove.
On the other hand, if $\ce(f)\ne\emptyset$ then $\per(f)\ne\emptyset$.
In this case, by Proposition~\ref{PropositionTotTransPer} in Appendix~I, replacing $f$ by $g:=f^{\ell}|_V$, where $\ell=\min\{j\ge1\,;\,\fix(f^j)\ne\emptyset\}$ and $V$ is an open set such that $(f^*)^{\ell}(V)\subset V$, we get that $g^n$ is strongly transitive for every $n\ge1$ and so, it follows from Proposition~\ref{prop:reduction-III} that 
\begin{equation}\label{eq:PSlel}
P_{\ce(g)}(S_\ell \,\varphi)=\sup\bigg\{h_\mu(g)+\int S_\ell \,\varphi \,d\mu\,:\,\mu\in\ce^*(g)\bigg\},
\end{equation}
where $S_\ell \,\varphi=\sum_{j=0}^{\ell-1} \varphi\circ f^j$.
 By Proposition~\ref{PropositionTotTransPer}, 
$V\cup f^*(V)\cup\cdots\cup (f^*)^{\ell-1}(V)\supset\XX\setminus\cc$ and $\mu(\cc)=0$ for every $\mu\in\ce(f)$, and so we have that $$\ce(g)\ni\mu\mapsto \frac{1}{\ell}\sum_{j=0}^{\ell-1}\mu\circ f^{-j}\in\ce(f)$$ is a surjection. 
This, together with \eqref{eq:PSlel} and the fact that $P_{\ce(f^\ell)}(S_\ell\, \varphi) = \ell \, P_{\ce(f)}(\varphi)$ proves the corollary. 
\end{proof}

\subsection{Special zooming induced maps}

The thermodynamic formalism involves a selection of invariant measures according to their entropy or free energy. 
Note that while every expanding measure is liftable to a zooming induced map (recall Proposition~\ref{Propositioiuiouiu877})
there is no guarantee that all relevant expanding measure, namely those with large entropy or large pressure,  
can be lifted to induced map, hence comparable. 
The goal of this subsection is to present a special induced map which can be used to compare {\em all} fat-induced zooming measures.

We define the (upper) {\bf\em conformal derivative} of $f$ at $p\in\XX\setminus\cc$ as $$\DD f(p)=\limsup_{x,y\to p}\frac{\dist(f(x),f(y))}{\dist(x,y)}.$$
As $f$ is a bi-Lipschitz local homeomorphism, we have that $0<\DD f(p)<+\infty$ for every $p\in\XX\setminus\cc$.

\begin{Theorem}\label{Theoremuhbkj100201}
Consider an ergodic expanding probability measure $\mu_0\in\ce(f,1)$ and let $\lambda,\delta>0$ be such that $\mu_0$ is a $(\gamma,\delta,1)$-zooming measure, where $\gamma=\{\gamma_n\}_n$ with $\gamma_n(r)=e^{-2\lambda\, n}r$.
Let $\beta=\{\beta_n\}_n$ be the zooming contraction given by $\beta_n(r)=e^{-\lambda\sqrt{n}}r$.

If $f^n$ is strongly transitive for every $n\ge1$ then every $\mu\in\ce^*(f,1)$ is a $(\beta,\delta/2,1)$-zooming measure. Moreover, there is 
$\varepsilon_0>0$ and $p\in\XX$ such that, for any $0<\varepsilon<\varepsilon_0$ there is a $(\beta,\delta/2,1)$-zooming return map $(F,B,\cp)$, where $\beta=\{\beta_n\}_n$ and $\beta_n(r)=e^{-\lambda\sqrt{n}}r$, satisfying the following properties:
\begin{enumerate}
\item $F$ is orbit-coherent.
\item $F:A\to B$, where $A=\bigcup_{P\in\cp}P$, $B$ is a connected open set with $B_{\varepsilon/2}(p)\subset B\subset B_{\varepsilon}(p)$.
\item $\#\{P\in\cp\,:\,R(P)=n\}<+\infty$ for every $n\in\NN$, where $R$ is the induced time of $F$.
\item $A$ is an open and dense subset of $B$.
\item $\mu_0$ is $F$-liftable.
\item Each ergodic $\mu\in\ce^*(f,1)$  has a unique $F$-lift $\overline{\mu}$;
moreover,
\begin{itemize}
	\item $\overline{\mu}$ is $F$-ergodic;
	\item $\overline{\mu}\le C\mu|_B$ for some constant $C\ge1$.
\end{itemize}
\item If $\sup_{x\notin\cc}\DD f(x)<+\infty$, then every $\mu\in\ce^*(f)$ is $(\beta,\delta/2,1)$-zooming. Furthermore, each ergodic $\mu\in\ce^*(f)$  has a unique $F$-lift $\overline{\mu}$ and 
\begin{itemize}
	\item $\overline{\mu}$ is $F$-ergodic;
	\item $\overline{\mu}\le C\mu|_B$ for some constant $C\ge1$.
\end{itemize}
\end{enumerate}
\end{Theorem}
\begin{proof}
As in the proof of Proposition~\ref{Propositioiuiouiu877},
given $x\in\XX$ and $V\subset\XX$, let
$$
\omega_{\gamma,\delta,\ell}(x)=\big\{y\in\XX\,:\,\tau_x(B_{\varepsilon}(y))>0\text{ for every }\varepsilon>0\big\},
$$ where  
$$
\tau_{x}(V)
	=\limsup_{n\to\infty}\frac{1}{n}\#\Big\{1\le j\le n\,:\,x\in\cz_j(\gamma,\delta,1)\text{ and }f^j(x) \in V\Big\}.
$$
It follows from Lemma~3.9 of \cite{Pi11} that exists and compact $\ca\subset\supp\mu_0$ such that $\omega_{\gamma,\delta,1}(x)=\ca$ for $\mu_0$-almost every $x\in\XX$.
Choose a point $p\in\ca$.

Let $p_0\in\limsup_n\cz_n(\gamma,\delta,1)$ be a $\mu_0$-generic point (in particular $p\in\omega_{\gamma,\delta,1}(p_0)$).
As $(e^{-2\lambda\ell n})^{1/2}=e^{-\lambda\ell n}\le e^{-\lambda n}$, it follows from Lemma~\ref{Lemma2121112} that 
$$
\co_f^-(p_0)\subset\limsup\cz_n(\alpha,\delta,1), \quad\text{where}\quad \alpha=\{e^{-\lambda n}r\}_n.
$$

As $\#f^{-1}(x)<+\infty$ $\forall\,x\in\XX$ $($\footnote{  Item (\ref{Cond3f}) in the definition of a {\em bi-Lipschitz local homeomorphism}.}$)$, the map $f$ is backward separated (cf. Definition~5.11 of \cite{Pi11}), hence there is $\varepsilon_1\in(0,\delta/2)$ such that, for every $0<\varepsilon<\varepsilon_1$, the $(\beta,\delta,1)$-zooming nested ball $B_{\varepsilon}^{\star}(p)\subset B_{\varepsilon}(p)$ is a well defined open connected set containing $B_{\varepsilon/2}(p)$ (see Definition~5.9 and Lemma~5.12 of \cite{Pi11}).
Hence, choose $0<\varepsilon<\varepsilon_0:=\min\{\varepsilon_1,\delta/4\}$, take $B:=B_{\varepsilon}^{\star}(p)$ and
$$
A:=\bigcup_{n\in\NN,\,y\in\cw_n}(f^n|_{V_n(\beta,\delta,1)(y)})^{-1}(B),$$
where 
$$\cw_n=\{y\in\cz_n(\beta,\delta,1)\text{ such that }f^n(V_n(\beta,\delta,1)(y))\supset B\}.
$$
Let also $\cp$ be the collection of connected components of $A$
and $F:A\subset B\to B$ be the full induced Markov map associated to the first $(\beta,\delta,1)$-zooming return time to $B$ with respect to $f$ (cf. Theorem~4 (page 914) of \cite{Pi11}).
Hence, $(F,B,\cp)$ is a $(\beta,\delta,1)$-zooming return map.
Moreover, as $\co_f^-(p_0)$ is dense on $\XX$, $\co_f^-(p_0)\subset\limsup_n\cz_n(\alpha,\delta,1)\subset\limsup_n\cz_n(\beta,\delta,1)$ and $p\in\omega_{\alpha,\delta,1}(p_0)=\omega_{\alpha,\delta,1}(y)\subset\omega_{\beta,\delta,1}(y)$ for every $y\in\co_f^-(p_0)$, we get that $A\supset\co_f^-(p_0)\cap B$.
This ensures that $A$ is an open and dense subset of $B$.
The fact that $\#\{P\in\cp\,:\,R(P)=n\}<+\infty$  for every $n\in\NN$ follows from $\#f^{-1}(x)<+\infty$ for every $x\in\XX$.
In particular the induced map $F$ satisfies items (1)-(4).
As $p\in\omega_{\gamma,\delta,1}(x)\subset\omega_{\alpha,\delta,1}(x)$ for $\mu_o$ almost every $x$, we get from  Theorem~4 in \cite{Pi11} that $\mu_0$ is $F$-liftable, proving item (5).

Let $\mu\in\ce^*(f,\ell)$, 
$\ell\ge1$.
By definition, there are $\sigma,a>0$ and a $(\eta,a,\ell)$-zooming return map $(F_0,B_0,\cp_0)$ having a $F_0$-lift $\nu$ of $\mu$ such that $\supp\nu=\overline{B_0}$, where $\eta=\{e^{-\sigma n}r\}_n$.

Set $K_1=1$ and $K_{\ell}:=\sup\{\DD f(x)\,:\,x\in \XX\setminus\cc\}<+\infty$ for $\ell\ge2$.
As $\mu$ is an expanding measure, there exists $x\in\XX\setminus\cc$ so that 
$\DD f(x)\ge1$ and, consequently, $K_{\ell}\ge1$ for every $\ell\in\NN$. 
Take $n_0=\max\big\{\big(\frac{(\ell-1)\log (Ke^{\sigma})}{\lambda}\big)^2,\big(\frac{2\lambda}{\min\{\lambda,\sigma\}}\big)^2\big\}$ and note that  
$($\footnote{
 Write $\alpha=(\ell-1)\log (K_{\ell}e^{\sigma})$ and $\beta=\min\{\lambda,\sigma\}$.
 If $n\ge n_0$ then $n+m\ge(\alpha/\lambda)^2$ and so, $\alpha+\lambda\sqrt{n+m}\le 2\lambda\sqrt{n+m}$.
On the other hand, for $n\ge n_0$, we also have that $n+m\ge n\ge (2\lambda/\beta)^2$ and so, $2\lambda\le \beta\sqrt{n+m}$.
This implies that $2\lambda\sqrt{n+m}\le\beta(n+m)$.
Thus, $\alpha+\lambda\sqrt{n+m}\le 2\lambda\sqrt{n+m}\le\beta(n+m)\le\lambda n+\sigma m$.
Finally, \eqref{Equationjolugotufo76} follows from $\alpha+\lambda\sqrt{n+m}\le \lambda n+\sigma m$ $\iff$ $(K_{\ell}e^{\sigma})^{\ell-1}e^{-\lambda n-\sigma m}\le e^{-\lambda\sqrt{n+m}}$.}$)$

\begin{equation}\label{Equationjolugotufo76} 
e^{-\lambda n-\sigma m}\le (K_{\ell}e^{\sigma})^{\ell-1}\,e^{-\lambda n-\sigma m}\le  e^{-\lambda\sqrt{n+m\;}}\text{ for every }m\ge0\text{ and }n\ge n_0.
\end{equation}
Recall that $B_0$ is a connected open set and let $R_0$ be the induced time of $F_0$ and choose $p_\mu\in B_0\cap\co_f^-(p_0)$.
As $p\in\omega_{\alpha,\delta,1}(y)$ for every $y\in\co_f^-(p_0)$, one can choose $n_1\ge n_0$ large enough 
so that $p_\mu\in\cz_{n_1}(\alpha,\delta,1)$, that
 $V_{n_1}(\alpha,\delta,1)(p_\mu)\subset B_0$ and $f^{n_1}(p_{\mu})\in B$.
Let  
$$
V=(f^{n_1}|_{V_{n_1}(\alpha,\delta,1)(p_\mu)})^{-1}(B)\subset (f^{n_1}|_{V_{n_1}(\alpha,\delta,1)(p_\mu)})^{-1}(B_{\delta}(f^{n_1}(p_\mu)))
$$
and consider the set 
$
\NN_x=\{n\in\NN\,:\,F_0^n(x)\in V\}.
$ 
As $\nu$ is $F_0$-ergodic and $\supp\nu=\overline{B_0}$, there exists $U\subset B_0$ with $U=B_0 (\mbox{mod}\,\nu)$ such that 
$$
\lim_{n\to\infty} \frac{1}{n}\#\big(\{1,\cdots,n\}\cap\NN_x\big)=\nu(V)>0
\quad\text{for every $x\in U$.}
$$

Given $x\in U$ and $n\ge1$, let $\cp_{0,n}=\bigvee_{j=0}^{n-1}F_0^{-j}(\cp_0)$ and let $\cp_{0,n}(x)$ denote
the element of the partition $\cp_{0,n}$ containing $x$.
Take $r_n=\frac{1}{\ell}\sum_{j=0}^{n-1}R_0\circ F_0^j(x)$ and $g=f^{\ell}$.
By construction, there exists $x_n\in\cz_{r_n}(\eta,a,\ell)$ such that $\cp_{0,n}(x)=(g^{r_n}|_{V_{r_n}(\eta,a,\ell)(x_n)})^{-1}(B_0)$.
Moreover, since $F_0^n(x)\in V$ for every $n\in\NN_x$, we get that $f^{n_1}(F_0^n(x))\in  B$ and $B_{\delta/2}(f^{n_1}(F_0^n(x)))\subset B_{\delta}(f^{n_1}(p_\mu))$,  since  $\diam(B)<\delta/2$ and $f^{n_1}(p_{\mu})\in B$. 
Hence, 
setting 
$$V_{\ell r_n+n_1}(x)=(F_0^n|_{\cp_{0,n}(x)})^{-1}\circ(f^{n_1}|_{V_{n_1}(\alpha,\delta,1)(p_{\mu})})^{-1}(B_{\delta/2}(f^{n_1}\circ F_0^n(x)))$$
when $n\in\NN_x$, we have that 
$f^{\ell r_n+n_1}$ sends $V_{\ell r_n+n_1}(x)$ diffeomorphically to $B_{\delta/2}(f^{\ell r_n+n_1}(x))$.
Moreover, as $f^{\ell r_n}(V_{\ell r_n+n_1}(x))\subset V_{n_1}(\alpha,\delta,1)(p_\mu)$ we get that
$$\dist(f^j(y),f^j(z))\le e^{-\lambda(\ell r_n+n_1-j)}\dist(f^{\ell r_n+n_1}(y),f^{\ell r_n+n_1}(z))\le$$
$$\le e^{-\lambda\sqrt{\ell r_n+n_1-j}}\dist(f^{\ell r_n+n_1}(y),f^{\ell r_n+n_1}(z))$$
for every $\ell r_n\le j<\ell r_n+n_1$ and $y,z\in V_{\ell r_n+n_1}(x)$.
That is, 
\begin{equation}\label{Equation10010010011}
	\dist(f^j(y),f^j(z))\le \beta_{\ell r_n+n_1-j}\big(\dist(f^{\ell r_n+n_1}(y),f^{\ell r_n+n_1}(z))\big)\;\;\;\forall\,\ell r_n\le j<\ell r_n+n_1\end{equation}
and every $y,z\in V_{\ell r_n+n_1}(x)$.

In what follows we consider the cases $\ell=1$ and $\ell\ge2$ separately.
Assume first that $\ell=1$.
In this case,  using that $V_{r_n+n_1}(x)\subset V_{r_n}(\eta,a,1)(x_n)$ and (\ref{Equationjolugotufo76}) we get that
$$
\dist(f^j(y),f^j(z)) \le e^{-\sigma(r_n-j)}\dist(f^{r_n}(y),f^{r_n}(z))\le  e^{-\sigma(r_n-j)}e^{-\lambda n_1}\dist(f^{r_n+n_1}(y),f^{r_n+n_1}(z))\le$$
$$\le e^{-\lambda\sqrt{r_n+n_1-j}}\dist(f^{r_n+n_1}(y),f^{r_n+n_1}(z))$$
for every $0\le j<r_n$ and $y,z\in V_{r_n+n_1}(x)$.
In other words, 
\begin{equation}\label{Equation0o01110}
	\dist(f^j(y),f^j(z))\le\beta_{r_n+n_1-j}\big(\dist(f^{r_n+n_1}(y),f^{r_n+n_1}(z))\big)\;\;\;\forall\,0\le j<r_n
\end{equation}
and every $y,z\in V_{\ell r_n+n_1}(x)$.
As $\ell=1$, equations (\ref{Equation10010010011})  and (\ref{Equation0o01110}) imply that $r_n+n_1$ is a $(\beta,\delta/2,1)$-zooming time for $x$ and $V_{r_n+n_1}(x)=V_{r_n+n_1}(\beta,\delta/2,1)(x)$ whenever $n\in\NN_x$.

Let us consider now the case $\ell\ge2$. Given $0\le j<\ell r_n$, write $j=m\ell +r$ with $0\le r<\ell$.
Using that $V_{\ell r_n+n_1}(x)\subset V_{r_n}(\eta,a,\ell)(x_n)$, with $n_1\ge n_0$, and \eqref{Equationjolugotufo76}, we obtain that
$$\dist(f^j(y),f^j(z))=\dist(f^{m\ell+r}(y),f^{m\ell+r}(z))=\dist(f^{\ell-r}(f^{(m+1)\ell}(y)),f^{\ell-r}(f^{(m+1)\ell}(z))\le
$$
\begin{equation}\label{Equationig75329000b}
  \le  (K_{\ell})^{\ell-r} \dist(f^{(m+1)\ell}(y),f^{(m+1)\ell}(z))\le (K_{\ell})^{\ell-r} e^{-\sigma(r_n-m-1)\ell}\dist(f^{r_n\ell}
(y),f^{r_n\ell}(z))=
\end{equation}
$$=(K_{\ell})^{\ell-r} e^{+\sigma(\ell-r)}e^{-\sigma(r_n\ell-j)\ell} \dist(f^{r_n\ell}
(y),f^{r_n\ell}(z))\le (K_{\ell}e^{\sigma})^{\ell-1}e^{-\sigma(r_n\ell-j)}\dist(f^{r_n\ell}(y),f^{r_n\ell}(z))\le$$ $$\le (K_{\ell}e^{\sigma})^{\ell-1} e^{-\sigma(r_n\ell-j)}e^{-\lambda n_1}\dist(f^{r_n\ell+n_1}(y),f^{r_n\ell+n_1}(z))\le e^{-\lambda\sqrt{r_n\ell+n_1-j}}\dist(f^{r_n\ell+n_1}(y),f^{r_n\ell+n_1}(z))$$
for every $0\le j<r_n$ and $y,z\in V_{r_n+n_1}(x)$. 
That is, 
\begin{equation}\label{Equation0o01109910}
	\dist(f^j(y),f^j(z))\le\beta_{r_n\ell+n_1-j}\big(\dist(f^{r_n\ell+n_1}(y),f^{r_n\ell+n_1}(z))\big)\;\;\;\forall\,0\le j<r_n\ell
\end{equation}
and every $y,z\in V_{\ell r_n+n_1}(x)$.
Therefore, (\ref{Equation10010010011}) and (\ref{Equation0o01109910}) together imply that 
 $r_n\ell+n_1$ is a $(\beta,\delta/2,1)$-zooming time for $x$ and $V_{r_n\ell+n_1}(x)=V_{r_n\ell+n_1}(\beta,\delta/2,1)(x)$ for each $n\in\NN_x$.
Observe that we use in \eqref{Equationig75329000b}, for $\ell\ge2$, that $K_{\ell}=\sup_{x\notin\cc}\DD f(x)<+\infty$ (for $\ell=1$ this is not necessary as $K_1=1$).

Now, let $\LL_x=\{r_n\ell+n_1\,:\,n\in\NN_x\}$
and note that 
$$
\lim_{m\to\infty}\frac{1}{m}\#\Big\{1\le j\le m\,:\,j\in\LL_x\Big\} 
	= \frac{1}{\int R_0 \,d\nu}\lim_{m\to\infty}\frac{1}{m}\#\Big\{1\le j\le m\,:\,j \in\NN_x\Big\} 
	= \frac{\nu(V)}{\int R_0 \,d\nu}>0.
$$
As $f^n(x)\in B$ for every $x\in U$ and $n\in\LL_x$, this means that every point $x\in U$ has positive frequency of $(\beta,\delta/2,1)$-zooming times $n$ such that $f^n(x)\in B$.
Moreover, using that $\mu$ is ergodic, we get that $\mu$ almost every $x\in\XX$ has positive frequency of $(\beta,\delta/2,1)$-zooming times $n\in\NN$ such that $f^n(x)\in B$. In particular,
$\mu$ is a $(\beta,\delta/2,1)$-zooming measure when 
$$\begin{cases}
	\mu\in\ce^*(f,1)\text{ or }\\
	\mu\in\ce^*(f,\ell)\text{ for }\ell\ge2\text{ and }\sup_{x\notin\cc}\DD f(x)<+\infty\end{cases}$$

Furthermore, as $F$ is orbit coherent and a full induced map, we can use Theorem~1 of \cite{Pi11} 
(or Theorem~A of \cite{Pi20}) to conclude that $\mu$ is $F$-liftable to some probability measure $\overline{\mu}$ and $\overline{\mu}\le C\mu|_{B}$ for some constant $C\ge1$.
Finally, it follows from Theorem~B of \cite{Pi20}, $\overline{\mu}$ is $F$-ergodic and it is the unique $F$-lift of $\mu$.
This proves items (6) and (7), completing the proof of the theorem.
\end{proof}

As defined in Section~\ref{SectionSMS}, the map $f:\XX\setminus\cc\to\XX$ is called {\bf\em piecewise injective} when there exists a finite cover of $\XX\setminus\cc$ by open sets $\{X_1,\cdots,X_t\}$ such that $f|_{X_j}$ is injective for every $1\le j\le t$.

\begin{Corollary}\label{CorollaryUPPERSEMI}
Suppose that $f$ is strongly transitive and piecewise injective.
Assume also that  $\sup_{x\notin\cc}\DD f(x)<+\infty$ and $f$ can be extended continuously to a map $\bar{f}:\XX \to \XX$.
If $\mu_n\in\ce(f)$, $n\in\NN$, is a sequence converging to some $\bar{f}$-invariant probability measure $\mu_0$ with $\mu_0(\cc)=0$, then $
  h_{\mu_0}(\bar{f}\,)\ge\limsup_n h_{\mu_n}(f)$.
\end{Corollary} 
\begin{proof}
One can use Proposition~\ref{PropositionTotTransPer} of Appendix~I and follow the same argument of the proof of Corollary~\ref{Corollarykbiohg4k234} to conclude that, without loss of generality, changing $f$ by $f^{\ell}$ if necessary, we may assume that $f^n$ is strongly transitive for every $n\ge1$.
Up to replace $f$ by $f^t$ if $\mu_1\in\ce(f,t)$, we can assume that $\mu_1\in\ce(f,1)$.
That is, we can assume that $\mu_1\in\ce(f,1)$ and $f^n$ is strongly transitive for every $n\ge1$.

Let $\lambda,\delta>0$ be such that $\mu_1$ is a $(\gamma,\delta,1)$-zooming measure, where $\gamma=\{\gamma_n\}_n$ and $\gamma_n(r)=e^{-2\lambda\,n}r$.
It follows from Theorem~\ref{Theoremuhbkj100201} that every $\nu\in\ce^*(f)$ is a $(\beta,\delta/2,1)$-zooming measure, where $\beta=\{\beta_n\}_n$ and $\beta_n(r)=e^{-\lambda\sqrt{n}}r$.

By Proposition~\ref{prop:reduction-III}, one can choose a sequence $\mu_n^*\in\ce^*(f)$ such that $\mu_n^*\to\mu_0$ and $\lim_n| h_{\mu_n^*}(f)-h_{\mu_n}(f)|=0$.
As $\mu_n^*$ is a sequence of $(\beta,\delta/2,1)$-zooming measures for both $f$ and  $\bar{f}$, it follows from Lemma~\ref{Lemmadfghjye32zzx} of Appendix~I that $h_{\mu_0}(\bar{f}\,)\ge\liminf_n h_{\mu_n^*}(f)=\liminf_n h_{\mu_n}(f)$. 
\end{proof}

\section{Measure of maximal  entropy for a full induced Markov map}\label{sec:thermodynamics}

For each $r\ge1$,  define $$\AA_r=\big\{\{a_n\}_{n\ge1}\,;\,a_n\ge0\,\forall\,n\text{ and }\sum_n a_n\le1\le r\le\sum_n n a_n\big\}$$
and set ${H}:[0,1]\to[0,1]$ as
$$H(x)=\begin{cases}
0 & \text{ if }x=0\\
x\log(1/x) & \text{ if }x>0
\end{cases}.$$

\begin{Lemma}\label{LemmaCont987yg}
\begin{equation}\label{Equationtfyf}
  \lim_{r\to\infty}\bigg(\sup\bigg\{\lim_{m\to\infty}\bigg(\frac{\sum_{n=1}^m{H}(a_n)}{\sum_{n=1}^m n a_n}\bigg)\,;\,\{a_n\}\in\AA_r\bigg\}\bigg)=0.
\end{equation}
\end{Lemma}
\begin{proof}
For each $n\in\NN$, let us set $\cu_n=\{j\in\NN\,;\,j^{-(n-1)}\ge a_j>j^{-n}\}$. Define  $\NN_0=\{j\in\NN\,;\,a_j\ne0\}$. Notice that $\{\cu_n\,;\,n\in\NN_0\}$ is a partition of $\NN_0$, that is, $\NN_0=\bigcup_{n\in\NN_0}\cu_n$  and $\cu_n\cap\cu_m=\emptyset$ whenever $n\ne m$.
For each $n\in\NN_0$ set $u_n=\min\cu_n$. This means that the first integer $j$
for which $a_j\in (\frac{1}{j^n},\frac{1}{j^{n-1}})$ is $j=u_n$. Write $\{n_1,n_2,n_3,...\}=\NN_0$ with $n_1<n_2<n_3\cdots$. As a consequence, $u_{n_j}\ge j$ $\forall\,j$.

In order to estimate ~\eqref{Equationtfyf} we bound first its numerator. Firstly consider the case when $n_k\le3$ and $j\in\cu_{n_k}$. In this case, ${H}(a_j)=a_j\log(1/a_j)\le n_k a_j\log(j)\le 3a_j\log(j)$. Let $\Gamma_3:=\{k\in\NN\,;\,n_k\le3\}$ and observe that $\#\Gamma_3\le3$. Therefore,
\begin{equation}\label{Eqkjhg56sr}
\sum_{k\in\Gamma_3}\bigg(\sum_{j\in\cu_{n_k}\cap\{1,\cdots,m\}}{H}(a_j)\bigg)\le 9\sum_{j=1}^{m}a_j\log(j).
\end{equation}

Now, suppose that $n_k\ge4$, that is $k\in\NN_0\setminus\Gamma_3$. In this case, for any $j\in\cu_{n_k}$, we have that ${H}(a_j)=a_j\log(1/a_j)\le n_k\log(j)j^{-(n_k-1)}\le n_k j^{-(n_k-2)}$. Thus,
$$\sum_{j\in\cu_{n_k}}{H}(a_j)\le\sum_{j\in\cu_{n_k}}\frac{n_k}{j^{n_k-2}}\le\sum_{j\ge u_{n_k}}\frac{n_k}{j^{n_k-2}}\le n_k\bigg(\frac{1}{(u_{n_k}+1)^{n_k-1}}+\int_{x=u_{n_k}}^{\infty}\frac{1}{x^{n_k-2}}dx\bigg)=$$
$$=n_k\bigg(\frac{1}{(u_{n_k}+1)^{n_k-1}}+\frac{1}{(n_k-3)(u_{n_k})^{n_k-3}} \bigg)\le2\frac{n_k}{n_k-3}\bigg(\frac{1}{u_{n_k}}\bigg)^{n_k-3}\le8\bigg(\frac{1}{u_{n_k}}\bigg)^{n_k-3}.$$
As $u_{n_k}\ge k$, we get $8\big(\frac{1}{u_{n_k}}\big)^{n_k-3}\le8\big(\frac{1}{k})^{n_k-3}\le 8\big(\frac{1}{k})^{k-3}$, we have
\begin{equation}\label{Eq0000gh}
\sum_{k\in\NN_0\setminus\Gamma_3}\bigg(\sum_{j\in\cu_{n_k}}{H}(a_j)\bigg)\le 8\sum_{k\in\NN_0\setminus\Gamma_3}\bigg(\frac{1}{k}\bigg)^{n_k-3}\le8\sum_{k=1}^{\infty}\bigg(\frac{1}{k}\bigg)^{k-3}\le 8\bigg(1+\sum_{k=2}^{\infty}\bigg(\frac{1}{2}\bigg)^{k-3}\bigg)=40.
\end{equation}
Putting together (\ref{Eqkjhg56sr}) and (\ref{Eq0000gh}), we get for any $\{a_n\}\in\AA_r$ that 
\begin{equation}\label{Eqijnjnono6}
  \sum_{n=1}^m{H}(a_n)=\sum_{n\in\Gamma_3\cap\{1,...,m\}}{H}(a_n)+\sum_{n\in\{1,...,m\}\setminus\Gamma_3}{H}(a_n)
\le9\sum_{n=1}^m\log(n) a_n+40
\end{equation}
From \eqref{Eqijnjnono6}, we get that 
$$\sup\bigg\{\lim_{m\to\infty}\frac{\sum_{n=1}^m{H}(a_n)}{\sum_{n=1}^m n a_n}\,;\,\{a_n\}\in\AA_r\bigg\}\le40/r+9\sup\bigg\{\lim_{m\to\infty}\frac{\sum_{n=1}^m \log(n)a_n}{\sum_{n=1}^m n a_n}\,;\,\{a_n\}\in\AA_r\bigg\}.$$
Thus, to finish the proof we need to control the ratio $\frac{\sum\log(n)a_n}{\sum n a_n}$, for all $\{a_n\}\in\AA_r$, when $r$ goes to infinite. That is, we need only to prove that
\begin{equation}\label{Eq9h98sn811}
\lim_{r\to\infty}\bigg(\sup\bigg\{\lim_{m\to\infty}\bigg(\frac{\sum_{n=1}^m\log(n) a_n}{\sum_{n=1}^m n a_n}\bigg)\,;\,\{a_n\}\in\AA_r\bigg\}\bigg)=0.
\end{equation}

Notice that, for every $1<m_0<m$ and $\{a_n\}\in\AA_r$, we have $$\sum_{n=1}^m\log(n) a_n\le\log(m_0)\bigg(\underbrace{\sum_{n=1}^{m_0-1}a_n}_{\;\;\;\le1}\bigg)+\sum_{n=m_0}^{m}\log(n) a_n\le$$
$$\le\log(m_0)+\sum_{n=m_0}^{m}\frac{\log(n)}{n}n a_n\le\log(m_0)+\frac{\log(m_0)}{m_0}\sum_{n=m_0}^{m}n a_n.$$
Therefore, 
$$\frac{\sum_{n=1}^m\log(n) a_n}{\sum_{n=1}^m n a_n}\le\frac{\log(m_0)}{\sum_{n=1}^m n a_n}+\frac{\log(m_0)}{m_0}.$$

Given any $\varepsilon>0$ let $m_0>1$ be such that $\frac{\log(m_0)}{m_0}<\varepsilon$.  So, for any $m\ge m_0$ and any $\{a_n\}\in\AA_r$, we have $\frac{\sum_{n=1}^m\log(n) a_n}{\sum_{n=1}^m n a_n}<\log(m_0)/r+\varepsilon$. Thus, $$\lim_{r\to\infty}\bigg(\sup\bigg\{\lim_{m\to\infty}\bigg(\frac{\sum_{n=1}^m\log(n) a_n}{\sum_{n=1}^m n a_n}\bigg)\,;\,\{a_n\}\in\AA_r\bigg\}\bigg)\le\varepsilon\,\,\forall\varepsilon>0.$$
Thus, we get (\ref{Eq9h98sn811}), finishing the proof of the proposition.
\end{proof}

Let $(F,B,\cp)$ is a full induced Markov map for $f$ with induced time $R$.
Given $\ell\in\NN$, let $\cp_{\ell}$ be the cylinders of order $\ell$ of $\cp$, that is,
\begin{equation}\label{Eqcylin}
  \cp_{\ell}=\bigvee_{n=1}^{\ell-1}F^{-n}(\cp)=\{P_0\cap F^{-1}(P_1)\cap\cdots\cap F^{-(\ell-1)}(P_{\ell-1})\,;\,P_0,\cdots,P_{\ell-1}\in\cp\}.
\end{equation}
Note that $(F^{\ell},B,\cp_{\ell})$ is also a full induced Markov map for $f$ with induced time $$R_{\ell}(x)=\sum_{n=0}^{\ell-1}R\circ F^n(x).$$

\begin{Definition}
Given $\nu\in\cm^1(F)$ and $r:\cu\to\NN$ with $\#\{P\in\cu\,;\,r(P)\le n\}<+\infty$ $\forall n\in\NN$, let $$H_{\nu}(\cu,r)=\limsup_{n\to\infty}\frac{\sum_{U\in\cu,\,r(U)\le n}H(\nu(U))}{\sum_{U\in\cu,\,r(U)\le n}r(U)\nu(U)}.$$
Define the {\bf\em normalized  entropy} of $\nu\in\cm^1(F)$ as  
$$h_{\nu}(F,R):=\inf\{H_{\nu}(\cp_{\ell},R_{\ell})\,;\,\ell\in\NN\}.$$
\end{Definition}

\begin{remark}
The notion of normalized entropy, inspired by Abramov's formula \cite{Zwei}, takes into consideration
the refined partitions $\cp_{\ell}$ and the value of the roof function associated to such partition elements.
Moreover, observe that if $\sum_{U\in\cu}H(\nu(U))<+\infty$  then 
$$H_{\nu}(\cu,r)=\frac{\sum_{U\in\cu}H(\nu(P))}{\sum_{U\in\cu}r(U)\nu(U)}\in\RR$$
even if $\sum_{U\in\cu}r(U)\nu(U)=+\infty$ (in which case $H_{\nu}(\cu,r)=0$).
\end{remark}

Denoting $H_{\nu}(\cp)=\sum_{P\in\cp}H(\nu(P))$ and setting
$$\cm_*^1(F)=\bigg\{\nu\in\cm^1(F)\,;\,H_{\nu}(\cp)<+\infty\bigg\},$$
we have that 
\begin{equation}\label{Equationnytr5}
 H_{\nu}(\cp_{\ell},R_{\ell})=\frac{H_{\nu}(\cp_{\ell})}{\int R_{\ell}d\nu}=\frac{\frac{1}{\ell}H_{\nu}(\cp_{\ell})}{\int Rd\nu}\;\;\forall\ell\in\NN,\forall \nu\in\cm_*^1(F).
\end{equation}
Thus, using that $\cp$ is a generating partition,  the fact that the sequence $\rho_{\ell}:=H_{\nu}(\cp_{\ell})$  is subadditive and \eqref{Equationnytr5}, we get that   
\begin{equation}\label{Eqjhgfjkl}
h_{\nu}(F,R)=\frac{\inf_{\ell}\frac{1}{\ell}H_{\nu}(\cp_{\ell})}{\int R d\nu}=\frac{h_{\nu}(F)}{\int R d\nu},\;\;\forall\,\nu\in\cm_*^1(F).
\end{equation}
Given $\nu\in\cm^1(F)$, we have that $\{H_{\nu}(\cp_{\ell},R_{_{\ell}})\,;\,\ell\in\NN\}\supset\{H_{\nu}(\cp_{n\ell},R_{_{n\ell}})\,;\,\ell\in\NN\}$ and so, \begin{equation}\label{eqkjhvfytefgbuyg}
h_{\nu}(F,R)\le h_{\nu}(F^n,R_n)\;\;\forall n\in\NN\text{ and }\nu\in\cm^1(F).
\end{equation}

\begin{Lemma}\label{Lemmammama}
$h_{\nu}(F,R)\le H_{\nu}(\cp,R)\le h(f,F)$ $\forall\,\ell\in\NN$ and $\nu\in\cm^1(F)$.
\end{Lemma}
\begin{proof}
Given $\nu\in\cm^1(F)$ and $\ell\in\NN$, let $\nu_n$ be the $F$-invariant Bernoulli probability measure generated  by the mass distribution $\mathfrak{M}(P)=\frac{\nu(P\cap\{R\le n\})}{\nu(\{R\le n\})}$, $P\in\cp$, and note that 
$$H_{\nu_n}(\cp)=\sum_{P\in\cp}H(\nu_n(P))=\sum_{P\in\cp,\,R(P)\le n}H(\nu_n(P))=$$
$$=\sum_{P\in\cp,\,R(P)\le n}\frac{\nu(P)}{\nu(\{R\le n\})}\bigg(\log\bigg(\frac{1}{\nu(P)}\bigg)-\log\bigg(\frac{1}{\nu(\{R\le n\})}\bigg)\bigg)=$$
$$=\bigg(\frac{1}{\nu(\{R\le n\})}\sum_{P\in\cp,\,R(P)\le n}H(\nu(P))\bigg)-\log\bigg(\frac{1}{\nu(\{R\le n\})}\bigg)$$
$$\int R d\nu_n=\frac{1}{\nu(\{R\le n\})}\int_{\{R\le n\}}R d\nu.$$
As a consequence, 
\begin{equation}\label{eqljutf335aa}
\frac{H_{\nu_n}(\cp)}{\int R d\nu_n}=\frac{\sum_{P\in\cp,\,R(P)\le n}H(\nu(P))}{\int_{\{R\le n\}}R d\nu}-\frac{H(\nu(\{R\le n\}))}{\int_{\{R\le n\}}R d\nu}.
\end{equation}

Since 
$$\lim_{n}\frac{H(\nu(\{R\le n\})}{\int_{\{R\le n\}}R d\nu}\le \lim_{n}H(\nu(\{R\le n\})=0,$$
we can take the $\limsup$ at \eqref{eqljutf335aa} to prove that 
$$H_{\nu}(\cp,R)=\limsup_{n\to\infty}\frac{H_{\nu_n}(\cp)}{\int R d\nu_n}.$$

Since $\int R d\nu_n\le n<+\infty$,
 $\nu_n$ is the $F^{\ell}$-lift of the $f$-invariant probability measure $$\mu_n:=\frac{1}{\int R \,d\nu_n}\; \sum_{j\ge0}f^j_*\big(\nu_n|_{\{R>j\}}\big)\in\cm^1(f,F),$$
and so, by Abramov's formula \cite[Theorem 5.1]{Zwei},  $h_{\nu_n}(F)=h_{\mu_n}(f)\int Rd\nu_n<+\infty$.
As $\nu_n$ is a $F$-invariant Bernoulli probability measure of the generating partition $\cp$, we get that $h_{\nu_n}(F)=H_{\nu_n}(\cp)$ and so,
$$\frac{H_{\nu_n}(\cp)}{\int R d\nu_n}=\frac{h_{\nu_n}(F)}{\int Rd\nu_n}=h_{\mu_n}(f)\le h(f,F)\;\;\forall n\in\NN,$$
proving that $h_{\nu}(F,R)=\inf_kH_{\nu}(F^k,R_k)\le  H_{\nu}(\cp,R)\le h(f,F)$. 
\end{proof}

A consequence of Lemma~\ref{Lemmammama} is that 
\begin{equation}
h(f,F)=\sup\{h_{\nu}(F,R)\,;\,\nu\in\cm^1(F)\}.
\end{equation} 
Indeed, recall that $h(f,F)=\sup\{h_{\mu}(f)\,;\,\mu\in\cm^1(f,F)\}$, where $\cm^1(f,F)$ is the set of all $F$-liftable $f$-invariant probability measures.
Thus, given $\mu\in\cm^1(f,F)$ there exists $\widetilde{\mu}\in\cm^1(F)$ with $\int Rd\widetilde{\mu}<+\infty$ and such that $\mu=\frac{1}{\int R \,d\nu}\; \sum_{j\ge0}f^j_*\big(\widetilde{\mu}|_{\{R>j\}}\big)$.
By Abramov's formula \cite[Theorem 5.1]{Zwei}, $h_{\widetilde{\mu}}(F)=h_{\mu}(f)\int Rd\widetilde{\mu}<+\infty$.
So, $\widetilde{\mu}\in\cm_*^1(F)$ and $h_{\mu}(f)=\frac{h_{\widetilde{\mu}}(F)}{\int R d\widetilde{\mu}}$.
Thus, we can use \eqref{Eqjhgfjkl} and Lemma~\ref{Lemmammama} to conclude that $$h(f,F)=\sup\{h_{\mu}(f)\,;\,\mu\in\cm^1(f,F)\}=\sup\bigg\{\frac{h_{\widetilde{\mu}}(F)}{\int R d\widetilde{\mu}}\,;\,\mu\in\cm^1(f,F)\bigg\}=$$
$$=\sup\{h_{\widetilde{\mu}}(F,R)\,;\,\mu\in\cm^1(f,F)\}\le\sup\{h_{\nu}(F,R)\,;\,\nu\in\cm^1(F)\}\le h(f,F).$$

\begin{Theorem}\label{theo.technical1}
If $(F,B,\cp)$ is a full induced Markov map then $\sum_{n\in\NN}\#\{R=n\}e^{-h(f,F)n}=1$ and there exists a unique $F$-invariant probability measure $\nu_0$ such that 
\begin{equation}\label{eq.tec}
h_{\nu_0}(F,R)
	= h(f,F).
\end{equation}
where $R$ is the induced time of $F$.
Furthermore,
\begin{enumerate}
\item $\nu_0$ is the  $F$-invariant Bernoulli probability measure given by $\nu_0(P)=e^{-h(f,F)R(P)}$ for every $P\in\cp$, in particular, $\supp\nu_0=\overline{\bigcup_{P\in \mathcal P} P}$;
\item if $\int R \, d\nu_0<+\infty$ then:
\begin{enumerate}
\item[(a)] 
$
\mu_0 = \frac1{\int R\, d\nu_0} \sum_{n\ge 1} f_*^n(\nu_0 \mid \{ R>n\})
$
is an $f$-invariant probability measure,
\item[(b)] 
 $\delta(F):=\frac{1}{\int R d\nu_0} \sum_{n\ge 1} H(\nu_0(\{R=n\}))\in\big(0,h(f,F)\big)$,
\item[(c)]  for each $\gamma>h(f,F)-\delta(F)$ there exists $C_{\gamma}>0$ such that $\int R d\overline{\mu}\le C_{\gamma}$ for the $F$-lift $\overline \mu$ of any $\mu\in\cm^1(f,F)$ with $h_{\mu}(f)\ge\gamma$;
\end{enumerate}
\item 
if $
\limsup_{n\to\infty} \frac1n \log \# \{R=n\}<h(f,F)$
then $\int R \, d\nu_0<+\infty$ and the $f$-invariant probability measure $\mu_0$ given by item (2a) has exponential decay of correlations.
\end{enumerate}
\end{Theorem} 
\begin{proof} Let $\nu$ be a $F$-invariant probability measure.
It follows from  the definitions of  $H_{\nu}(\cp,R)$ that 
\begin{equation}\label{eq:firsteq}
h_{\nu}(F,R)\le H_{\nu}(\cp,R_{\ell})=\limsup_{k\to\infty}\;\;\;\underbrace{\frac{\sum_{n\in\NN(k)} \sum_{R_{\ell}(P)=n,P\in\cp_{\ell}}H(\nu(P))}{\sum_{n\in\NN(k)}n \,\nu(\{R_{\ell}=n\})}}_{(\star)},  
\end{equation}
where $\NN(k)=\{1\le n<k+1\,;\,\{R=n\}\ne\emptyset\}$.
The ratio $(\star)$ above can be subdivided in two terms, which reflect the complexity in the comparison
between the level sets $\{R=n\}$ and the distribution in the level sets. More precisely, $(\star)$
is bounded by the  sum of 
\begin{equation}\label{eq.aux1}
\frac{\sum_{n\in\NN(k)}H(\nu(\{R=n\}))}{\sum_{n\in\NN(k)}  n \,
        \nu(\{R=n\})}
\end{equation}
and
\begin{equation}\label{eq.aux2}
\frac{\sum_{n\in\NN(k)}\left(\nu(\{R=n\}) \,
    \sum_{\substack{R(P)=n,P\in\cp}}H\big(\frac{\nu(P)}{\nu(\{R=n\})}\big)\right)}
		 {\sum_{n\in\NN(k)}  n \, \nu(\{R=n\})}.
\end{equation}
For short, let us denote $\#\{P\in\cp\,;\,R(P)=n\}$  by $\#\{R=n\}$.
As the logarithm is strictly concave, well known estimative ensure that   \eqref{eq.aux2} is bounded above by 
\begin{equation}\label{Equationojijnu2}
\frac{\sum_{n\in\NN(k)} \nu(\{R=n\}) \,
    \log \#\{R=n\}}
		 {\sum_{n\in\NN(k)}  n \, \nu(\{R=n\})},
\end{equation}
and that the equality holds if and only if 
\begin{equation}\label{eq.Lagmult}
\nu(P)
	=\frac{\nu(\{R=n\})}{\#\{R=n\}}\;\; \text{ for every }P\in\cp\text{ with }R(P)=n.
\end{equation}
Therefore, we get
\begin{equation}\label{Eqklbv56}
\eqref{eq.aux1}+\eqref{eq.aux2}
	\le\frac{ 
          	\sum_{n\in\NN(k)} \nu( \{R=n\} ) ( \log \#\{R=n\}-\log\nu(\{R=n\})\big) }
	 {\sum_{n\in\NN(k)}  n \, \nu(\{R=n\})},
\end{equation}
and the equality is attained for the probability measure obtained by mass distribution using \eqref{eq.Lagmult} 
on each level set $\{R=n\}$.

For $k\in\NN\cup\{+\infty\}$, let $$\AA(k):=\bigg\{\{a_n\}_{n\in\NN}\,;\, a_n\in[0,1]\, \forall n\in\NN, a_n=0\, \forall n\notin\NN(k)\text{ and }\sum_{n\in\NN(k)}a_n=1\bigg\}.$$
For each $A=\{a_n\}\in\AA(+\infty)$, let $\nu_A$ be the $F$-invariant Bernoulli probability measure generated  by the mass distribution $\mathfrak{m} (P)=\frac{a_n}{\#\{R=n\}}$ for every $P\in\cp$ with $R(P)=n$ and $n\in\NN(+\infty)$.
Hence, if we take 
$$
\bc:= 
\sup 
		\left\{  \limsup_{k\to\infty}\frac{\sum_{n\in\NN(k)} a_n (\log \#\{R=n\}-\log a_n)}
	        {\sum_{n\in\NN(k)} n \,a_n}  \,;\,\{a_n\}_{n}\in \AA(+\infty)\right\},
$$
it follows from \eqref{eq:firsteq} and  \eqref{Eqklbv56} that 
\begin{equation}\label{Equaznkhg}
  \bc=\sup\{H_{\nu_A}(\cp,R)\,;\,A\in\AA(+\infty)\}=\sup\{H_{\nu}(\cp,R)\,;\,\nu\in\cm^1(F)\}=h(f,F).
\end{equation}

Moreover,
\begin{equation}\sup\bigg\{\label{eq.Lagmult234a}\limsup_{k\to\infty}
\sum_{n\in\NN(k)} 
	a_n\big(\log \#\{R=n\}-\bc n -\log a_n\big)\,;\,\{a_n\}_n\in\AA(+\infty)\bigg\}=0
\end{equation}

Given $k\in\NN$, let $\gamma_k=\sum_{n=1}^ke^{-\bc n}\#\{R=n\}<+\infty$ and consider $\{\alpha_n\}_n\in\AA(+\infty)$ defined as 
$$\alpha_n=
\begin{cases}\frac{e^{-\bc n}\#\{R=n\}}{\gamma_k} & \text{ if }n\in\AA(k)\\
0 & \text{ if }n\notin\AA(k)
\end{cases}.$$
It follows from \eqref{eq.Lagmult234a} that  
$$0\ge \sum_{n\in\NN(k)} 
	\alpha_n\big(\log \#\{R=n\} - \bc n -\log \alpha_n\big)=\sum_{n\in\NN(k)} 
\alpha_n\log\gamma_k=\log\gamma_k,$$
proving that $\gamma_k\le 1$ for every $k\in\NN$.
This implies that $\gamma:=\sum_{n\in\NN}e^{-\bc n}\#\{R=n\}\le1$ and so, we can apply Proposition~\ref{PropositionPFSSIGMA} in Appendix~I for $\LL(k)=\NN(k)$ and $\beta_n=\#\{R=n\}e^{-\bc n}$.
As a consequence, if $\{a_n\}_n\in\AA(+\infty)$  then either 
$$\label{eq.Lagmult234}
\limsup_{k\to\infty}\sum_{n\in\NN(k)} 
	a_n\big(\log \#\{R=n\} - \bc n -\log a_n\big)=$$
$$=\limsup_{k\to\infty}\sum_{n\in\NN(k),\,a_n\ne0} a_n\log\bigg(\frac{\#\{R=n\}e^{-\bc n}}{a_n}\bigg)\le\log\bigg( \sum_{n\in\NN}  e^{-\bc n}\#\{R=n\}\bigg)=\log\gamma
$$
and $$\sum_{n\in\NN} 
	a_n\big(\log \#\{R=n\}-\bc n-\log a_n\big)=\log\gamma\iff a_n=\frac{\#\{R=n\}e^{-\bc n}}{\log\gamma}\;\forall\,n\in\NN.$$
In particular, $$\log\gamma=\max\bigg\{\sum_{n\in\NN} 
	a_n\big(\log \#\{R=n\}-\bc n-\log a_n\big)\,;\,\{a_n\}_n\in\AA(+\infty)\bigg\}=0.$$
Thus, 
$$\sum_{n\in\NN}\#\{R=n\}e^{-\bc n}=1,$$ 
and so, if $\{a_n\}\in\AA(+\infty)$ then 
\begin{equation}\label{equatubvtt}
\limsup_{k\to\infty}\sum_{n\in\NN(k)} 
	a_n\big(\log \#\{R=n\}-\bc n -\log a_n\big)=0 \iff a_n=\#\{R=n\}e^{-\bc n}\;\;\forall n\in\NN.
\end{equation}

A consequence of \eqref{eq.Lagmult}, \eqref{Eqklbv56} and \eqref{equatubvtt} is that, if $\nu\in\cm^1(F)$ then,  by \eqref{Equaznkhg},
\begin{equation}\label{Eqaukljgdi5}
  H_{\nu}(\cp_,R)=h(f,F)\iff\nu(P)=e^{-h(f,F) R(P)}\;\;\forall P\in\cp
\end{equation}

In particular, if we denote by $\nu_0$ the $F$-invariant probability measure defined by mass distribution $\mathfrak{m}(P)=e^{-h(f,F)\, R_{\ell}(P)}$, $\forall P\in\cp$, it follows from \eqref{Eqaukljgdi5} that  that $\nu_0$ is the unique $F$-invariant Bernoulli probability measure  satisfying  (\ref{eq.tec}), proving the claim below.
\begin{Claim}\label{Claimkhbihbojo4}
Given a full induced Markov map  $(\widetilde{F},\widetilde{B},\widetilde{\cp})$, the $\widetilde{F}$-invariant Bernoulli probability measure 
$\eta$ given by $\eta(P)=e^{-h(f,\widetilde{F})\,\widetilde{R}(P)}$ $\forall\,P\in\widetilde{\cp}$ is the unique $\widetilde{F}$-invariant Bernoulli probability measure 
satisfying 
$H_{\eta}(\widetilde{\cp},\widetilde{R})
	=h(f,\widetilde{F})
$, where $\widetilde{R}$ is the induced time of $\widetilde{F}$. 
\end{Claim}

We can use Claim~\ref{Claimkhbihbojo4} to prove the unicity of the measure satisfying (\ref{eq.tec}) over all $F$-invariant probability measures. Indeed, assume by contradiction that there exists other distinct measure $\mu\in\cm^1(F)$ that attains (\ref{eq.tec}).
As
$\mu\ne\nu_0$, there exists some $N\ge 1$ and an element 
$P \in\cp_{N}$ such that 
$\mu(P)\ne\nu_0(P)$, where $\cp_N=\bigvee_{n=0}^{N-1}F^{-n}(\cp)$. 
Let $\overline{\mu}$ be the $F^N$-invariant Bernoulli probability measure obtained from $\mu$ by mass distribution $\mathfrak{m}(P)=\mu(P)$ on the cylinders $P\in\cp_N$.
Note that $\nu_0$ is also 
a $F^N$-invariant Bernoulli probability measure, as it is a $F$-invariant Bernoulli probability measure.
It follows from \eqref{eqkjhvfytefgbuyg} that $h_{\mu}(F,R)\le h_{\mu}(F^N,R_N)$ and $h_{\nu_0}(F,R)\le h_{\nu_0}(F^N,R_N)$.
On the other hand, it follows directly from the formula that $H_{\mu}(\cp_N,R_N)=H_{\overline{\mu}}(\cp_N,R_N)$.
Thus, using Lemma~\ref{Lemmammama} allied to $F^N$ and that $h(f,F^N)=h(f,F)$, we get that 
$$h(f,F^N)=h(f,F)=h_{\mu}(F,R)\le h_{\mu}(F^N,R_N)\le H_{\mu}(\cp_N,R_N)= H_{\overline{\mu}}(\cp_N,R_N) \le h(f,F^N)$$
and 
$$h(f,F^N)=h(f,F)=h_{\nu_0}(F,R)\le h_{\nu_0}(F^N,R_N)\le H_{\nu_0}(\cp_N,R_N) \le h(f,F^N)$$
proving that 
\begin{equation}\label{Equat47ggu}
H_{\nu_0}(\cp_N,R_N)=h(f,F^N)=H_{\overline{\mu}}(\cp_N,R_N).
\end{equation}

By Claim~\ref{Claimkhbihbojo4} applied to $\widetilde{F}:=F^N$, there exists a unique $F^N$-invariant Bernoulli probability measure $\eta$ such that $H_{\eta}(\cp_N,R_N)=h(f,F^N)$ and this is a contradiction with \eqref{Equat47ggu} and the fact that $\overline{\mu}\ne\nu_0$  are $F^N$-invariant Bernoulli probability measures.
Therefore, $\nu_0$ is the unique $\nu\in\cm^1(F)$ satisfying $H_{\nu}(\cp,R)=h(f,F)$.

Now, suppose that $\int R d\nu_0<+\infty$. Let $\gamma\in(h(f,F)-\delta(F),h(f,F))$ and consider a sequence $\mu_\ell\in\cm^1(f,F)$ with $h_{\mu_{\ell}}(f)\ge\gamma$.
Let $\overline{\mu}_\ell\in\cm^1(F)$ be a $F$-lift of $\mu_\ell$.
It is easy to check that $0<\delta(F)\le \frac{h_{\nu_0}(F)}{\int R d\nu_0}=H_{\nu_0}(\cp,R)=h(f,F)$.
Thus, to prove item (2), note that, if $\lim_n\int R d\overline{\mu}_n=+\infty$ then, by Lemma~\ref{LemmaCont987yg}, 
$$\lim_{\ell\to\infty}\lim_{k\to\infty}\frac{\sum_{n\in\NN(k)}H(\overline{\mu}_\ell(\{R_{\ell}=n\}))}{\sum_{n\in\NN(k)}  n \,
        \overline{\mu}_\ell(\{R_{\ell}=n\})}=0.$$
        On the other hand, it follows from (\ref{eq:firsteq}), (\ref{eq.aux1}), (\ref{Equationojijnu2}) and the definition of $\nu_0$ that 
$$\gamma\le\lim_{\ell\to\infty}\frac{h_{\overline\mu_\ell}(F)}{\int R d\overline\mu_\ell}=0+\lim_{\ell\to\infty}\lim_{k\to\infty}\frac{\sum_{n\in\NN(k)} \overline\mu_\ell(\{R_{\ell}=n\}) \,
    \log \#\{R_{\ell}=n\}}
		 {\sum_{n\in\NN(k)}  n \, \overline\mu_\ell(\{R_{\ell}=n\})}\le$$
$$\le\lim_{k\to\infty}\frac{\sum_{n\in\NN(k)} \nu_0(\{R_{\ell}=n\}) \,
    \log \#\{R_{\ell}=n\}}
		 {\sum_{n\in\NN(k)}  n \, \nu_0(\{R_{\ell}=n\})}=\frac{h_{\nu_0}(F)}{\int R d\nu_0}-\delta(F)=h(f,F)-\delta(F)<
		 \gamma,$$ which is a  contradiction. 

\medskip
Finally, concerning item (3),
as $\sum_{n\in\NN}\#\{R_{\ell}=n\}e^{-h(f,F)n}=1$, we get that $$\limsup_{n\to\infty}\frac{1}{n}\log\#\{R_{\ell}=n\}\le h(f,F).$$
Moreover, since $\nu_0(\{R_{\ell}=n\})=\#\{R_{\ell}=n\}e^{-h(f,F)n}$, 
$$
\int R\, d\nu_0 = \sum_{n\in \NN} n \#\{R_{\ell}=n\}e^{-h(f,F) n}
$$
In particular, if 
$\gamma:=\limsup_{n}\frac{1}{n}\log\#\{R_{\ell}=n\}<h(f,F)$, taking $\varepsilon>0$ such that $\gamma<h(f,F)-\varepsilon$, there exists $n_0\in\NN$ such that 
\begin{equation}\label{eq:exptailss}
\nu_0(\{R>n\})
= \sum_{k>n} k e^{-h(f,F)k}\#\{R=k\}
	\le \sum_{k>n} k e^{-\varepsilon k}\le \sum_{k>n} e^{-\varepsilon k/2}=\frac{1}{1- e^{-\varepsilon/2}}\, e^{-\frac{\varepsilon}{2}n}
\end{equation}
for every $n\ge n_0$, and so $\int R\, d\nu_0<\infty$. Furthermore, the exponential tails condition ~\eqref{eq:exptailss} is well known to imply that $\mu_0$ has exponential decay of correlations.
\end{proof}

\section{Uniqueness of expanding equilibrium states}\label{sec:proofsmain1}

This section is devoted to the proof of the results concerning uniqueness of equilibrium states, stated as Theorem~\ref{Maintheorem00}. 
The proofs combine liftability issues dealt in previous sections with results on uniqueness of
equilibrium states for piecewise expanding Markov maps \cite{BS03}.

\subsection{Induced potential}

Given a continuous potential $\varphi:\XX\to\RR$ and a full induced map $(F,B,\cp)$, define the $F$-induced pressure of $\varphi$ as
$$P(\varphi,f,F):=\sup\bigg\{h_{\nu}(f)+\int\varphi d\nu\,;\,\nu\in\cm^1(f,F)\bigg\}.$$
Define the {\bf\em $n$-variation} of a function $\Psi:\bigcup_{P\in\cp}P\to\RR$, by $$V_n(\Psi)=\sup\{|\Psi(x)-\Psi(y)|\,;\,x,y\in Q\text{ and }Q\in\cp_n\},$$ where $\cp_n$ is the set of all $P_1\cap F^{-1}(P_2)\cdots\cap F^{-(n-1)}(P_n)$ for all possible $P_1,\cdots,P_n\in\cp$.

\begin{Proposition}\label{Propositionioboi78ihbbnala}
Let 
 $(F,B,\cp)$ be a full induced Markov map with induced time $R$.
 If $\varphi:\bigcup_{P\in\cp}\bigcup_{j=0}^{R(P)-1}f^j(P)\to\RR$ is a continuous function with $P(\varphi,f,F)<+\infty$,  $\sum_{n\in\NN}V_n(\overline{\varphi})<+\infty$, where $\overline{\varphi}$ is the $F$-lift of $\varphi$, then there exists at most one $\mu\in\cm^1(f,F)$ such that
\begin{equation}\label{Equationkjgi754}
  h_{\mu}(f)+\int\varphi d\mu=P(\varphi,f,F).
\end{equation}
Furthermore, If $\mu\in\cm^1(f,F)$  satisfies \eqref{Equationkjgi754} then $\mu$ has a unique $F$-lift $\nu$, $\nu$ is $F$-ergodic  and $\supp\nu=\overline{\bigcup_{P\in\cp}P}$.
\end{Proposition}
\begin{proof} We may suppose that $\exists\mu\in\cm^1(f,F)$ satisfying \eqref{Equationkjgi754}, otherwise there are nothing to prove.
Defining $\varphi_0=\varphi-P(\varphi,f,F)$, we have that $P(\varphi_0,f,F)=0$ and also that $$h_{\mu}(f)+\int\varphi_0  d\mu=P(\varphi_0,f,F)=0.$$
Let $\Psi$ be the $F$-lift of $\varphi_0$, that is, $\Psi(x)=\sum_{j=0}^{R(x)-1}\varphi_0\circ f^j(x)$.
According to Sarig (Theorem~2 of \cite{Sa99}) and Iommi, Jordan \& Todd (Theorem 2.10 of \cite{IJT}) the Gurevich pressure~$($\footnote{ See, for instance, Section~3.1.3 of \cite{Sa09} for the definition of Gurevich pressure.}$)$ of $\Psi$ is given by 

\begin{equation}\label{EquationSarigniy5}
  P_G(\Psi)=\sup\bigg\{h_{\eta}(F)+\int\Psi d\eta\,;\,\eta\in\cm^1(F)\text{ and }\eta(\{R\le n\})=1\text{ for some }n\in\NN\bigg\}=
\end{equation}
\begin{equation}\label{Equationlkhkgv101}
  =\sup\bigg\{h_{\eta}(F)+\int\Psi d\eta\,;\,\eta\in\cm^1(F)\text{ and }\int\Psi d\eta<+\infty\bigg\}.
\end{equation}

\begin{Claim}\label{Claimwerty45}
$P_G(\Psi)=0$
\end{Claim}
\begin{proof}[Proof of the claim]Let $\nu$ be a $F$-lift of $\mu$. 
As $\int R d\nu<+\infty$, we get that $$h_{\nu}(F)+\int\Psi d\nu=\int R d\nu\bigg(\underbrace{h_{\mu}(f)+\int\varphi_0 d\mu}_0\bigg)=0.$$ Hence, as  $\int \Psi d\nu=(\int R d\nu)(\int\varphi_0 d\mu)$ $=$ $(\int R d\nu)(\int\varphi d\mu-P(\varphi,f,F))$ $<$ $+\infty$, it follows from \eqref{Equationlkhkgv101} that $P_G(\Psi)\ge0$.

Assume, by contradiction, that $P_G(\Psi)>0$.
Letting $\ell=\min\{j\in\NN\,;\,\{R\le j\}\ne\emptyset\}$, we have that $\{R\le n\}$ is a compact full shift for every $n\ge\ell$ and so, $F$ has a unique equilibrium state $\nu_n\in\cm^1(F|_{\{R\le n\}})$ for $\Psi$ supported on 
$\bigcap_{j\geq 0} F^{-j}(\{R\le n\})$, for every $n\ge\ell$.
Thus, it follows from \eqref{EquationSarigniy5} that, $h_{\nu_n}(F)+\int\Psi d\nu_n\to P_G(\Psi)$.
If $P_G(\Psi)>0$ then there exists $n_0\ge\ell$ such that $h_{\nu_n}(F)+\int\Psi d\nu_n>0$ for every $n\ge n_0$.
As $\int R d\nu_n\le n<+\infty$,  we have that $$\mu_n:=\frac{1}{\int R \,d\nu_n}\; \sum_{j\ge0}f^j_*\big(\nu_n|_{\{R>j\}}\big)\in\cm^1(f,F)$$ and so, 
$$0<h_{\nu_n}(F)+\int\Psi d\nu_n=\underbrace{\int R d\nu_n}_{>0}\;\bigg(\underbrace{h_{\mu_n}+\int\varphi_0 d\mu_n}_{\le0}\bigg)\le0,$$
which is a contradiction. Hence  $0\le P_G(\Psi)\le0$.  
\end{proof}

\begin{Claim}\label{Calimkhvodiyb}
$\sup\Psi<+\infty$.
\end{Claim}
\begin{proof}[Proof of the claim]
As $P_G(F)=0<+\infty$ and $\sum_{n\in\NN}V_n(\Psi)=\sum_{n\in\NN}V_n(\overline{\varphi})<+\infty$, we get that $V_1(\Psi)<+\infty$ and also that $\Psi$ is Walters. Thus, it follows from Theorem~4.9 of \cite{Sa09} that $\Psi$ admits a $F$-invariant Gibbs probability measure $\eta$.
In particular, as $P_G(\Psi)=0$, there exists $K\ge1$ such that $$\frac{1}{K}\le\frac{\eta(P)}{\exp\Psi(x)}\le K,\;\forall x\in P\text{ and }P\in\cp.$$
Hence, $$\Psi(x)\le \log(K\eta(P))\le \log K<+\infty$$ for every $x$ in the domain of $F$.
\end{proof}

As $\sup\Psi<+\infty$, $\sum_{n\in\NN}V_n(\Psi)<+\infty$ and $P_G(\Psi)=0<+\infty$, it follows from Buzzi \& Sarig (Theorem~1.1 of \cite{BS03}) that  exists at most one $\xi\in\cm^1(F)$ such that $h_{\xi}(F)+\int\Psi d\xi$ is well defined and $$h_{\xi}(F)+\int\Psi d\xi=\sup\bigg\{h_{\gamma}(F)+\int\Psi d\gamma\,;\,\gamma\in\cm^1(F)\text{ and }h_{\gamma}(F)+\int\Psi d\gamma\text{ is well defined}\bigg\},$$ one can also see Theorem~4.6 of \cite{Sa09}.
Moreover, $$h_{\xi}(F)+\int\Psi d\xi=P_G(\Psi)=0,$$ and, by Theorem~1.2 of \cite{BS03}, $$\supp\xi=\overline{\bigcup_{P\in\cp}P}.$$
Therefore, $\nu=\xi$ for every $F$-lift of any measure $\mu$ satisfying \eqref{Equationkjgi754}.
As a consequence, $$\mu=\frac{1}{\int R \,d\xi}\; \sum_{j\ge0}f^j_*\big(\xi|_{\{R>j\}}\big)\in\cm^1(f,F)$$ is the unique measure of $\cm^1(f,F)$ satisfying \eqref{Equationkjgi754} and $\supp\nu=\overline{\bigcup_{P\in\cp}P}$ for the $F$-lift $\nu$ of $\mu$.
\end{proof}

We now show that the assumptions of the previous proposition are satisfied by the induced potential of a 
Hölder continuous potential.
The next lemma ensures that such an induced potential has summable variations.

\begin{Lemma}\label{Lemmalihbdcw334}
Let $\varphi$ be a $(C,\gamma)$-piecewise Hölder continuous potential and $(F,B,\cp)$ be a $(\alpha,\delta,1)$-zooming return map, where $\alpha=\{\alpha_n\}_n$ is a Lipschitz zooming contraction $\alpha_n(r)=a_n r$ with  $S:=C\sum_n(a_n)^{\gamma}<+\infty$.
If $\overline{\varphi}$ is the $F$-lift of $\varphi$ then
$$\sum_{n\in\NN}V_n(\overline{\varphi})\le S^2\diameter(B)^{\gamma}<+\infty.$$
\end{Lemma}
\begin{proof}Given $P\in\cp$ and $x,y\in P$, we have that $f^j(x)$ and $f^j(y)$ belongs to a same connected component of $\XX\setminus\cc$ for every $0\le j<R(P)$, where $R$ is the induced time of $F$.
Thus,  
$$|\overline{\varphi}(x)-\overline{\varphi}(y)|=\bigg|\sum_{j=0}^{R(P)-1}\varphi\circ f^j(x)-\sum_{j=0}^{R(P)-1}\varphi\circ f^j(y)\bigg|\le\sum_{j=0}^{R(P)-1}C\dist(f^j(x),f^j(y))^{\gamma}\le$$ $$\le C\sum_{j=0}^{R(P)-1}\big(a_{_{R(P)-j-1}}\dist(F(x),F(y))\big)^{\gamma}\le C\sum_{n=1}^{R(P)}(a_n)^{\gamma}\dist(F(x),F(y))^{\gamma}\le S \dist(F(x),F(y))^{\gamma}.$$
This implies that $|\overline{\varphi}(x)-\overline{\varphi}(y)|\le S \diameter(B)^{\gamma}(a_n)^{\gamma}$ for every $x,y\in P$ and $P\in\cp_n$.
Thus, $\sum_{n\in\NN}V_n(\overline{\varphi})\le S\,\diameter(B)^{\gamma}\sum_{n\in\NN}(a_n)^{\gamma}=S^2\diameter(B)^{\gamma}<+\infty$.
\end{proof}

\subsection{Lipschitz zooming measures}

In view of the previous constructions of zooming inducing schemes, one concludes that measures that are liftable have some backward zooming contraction. 
Consider the set $\cl\cz(f)$ of all  {\bf\em Lipschitz zooming measures of $f$}. 
Given the inclusion $\ce(f) \subseteq \cl\cz(f)$ in the space 
 of Lipschitz zooming measures it is natural to look for conditions which guarantee that $\ce(f)=\cl\cz(f)$, which ultimately will allow one to prove existence of equilibrium states among expanding measures, and this will be discussed in Corollary~\ref{Corolkbihvg45678} below.
We first need the following auxiliary lemma. 
\begin{Lemma}
\label{LemmaKO798uhb}
Let $M$ be a Riemannian manifold and $f:M\setminus\cc\to M$ be a $C^{1+}$ local diffeomorphism with $\cc$ being a non-degenerated critical/singular set.
If $\mu\in\cm^1(f)$ is an expanding measure, in the sense it satisfies \eqref{Equationiyiy545bsA} and \eqref{Equationiyiy545bsB}, with all its Lyapunov exponents bigger than $\lambda>0$ then: 
\begin{enumerate}
\item[(i)] there exists $\ell\ge1$ such that 
$$
\lim_{n\to\infty} \frac{1}{n}\sum_{j=0}^{n-1}\log\|(Df^{j\ell}(x))^{-1}\|^{-1}\ge\lambda/4
\quad \text{for $\mu$-a.e. $x\in M$;}
$$ 
\item[(ii)] $\mu$ is a zooming measure with an exponential zooming contraction.
\end{enumerate}
\end{Lemma}
\begin{proof} Note that, as $\cc$ is non-degenerated, we get from (C1) that $$|\log\|(Df(x))^{-1}\||\le \log B+\beta|\log\dist(x,\cc)|.$$
Thus, as \eqref{Equationiyiy545bsA} holds, we get also that $\log\|(Df(\cdot))^{-1}\|$ is $\mu$-integrable.
Now,   item (i)  follows from the proof of Lemma~3.5 in \cite{KO} (alternatively use Lemma~7.1 in \cite{Pi20}) applied to the sub-additive cocycle $\varphi(n,x)=\log\|(Df^n(x))^{-1}\|^{-1}$. Item (ii) follows from the non-uniform expansion stated at item (i) combined with Corollary~3.2 in \cite{ABV}.
Finally, the last conclusion follows from item (ii) combined with Remark~\ref{Remarkjgiu91qa}.
\end{proof}

\begin{remark}
The assumption \eqref{Equationiyiy545bsB} on finiteness of Lyapunov exponents is not of technical nature.
In fact, maps with singularities may have non-trivial invariant probability measures with infinite Lyapunov exponents and positive entropy (cf. \cite{PP22} for the case of  expanding Lorenz maps). 
\end{remark}

We can now state the main result of this subsection.

\begin{Proposition}\label{Propkbihvg45678}
Let $M$ be a Riemannian manifold and $\Lambda\subset M$ a closed meager subset. Assume that 
$g:M\setminus\Lambda\to M$ is a $C^{1}$ local diffeomorphism.
If $\mu\in\cm^1(g)$ is a zooming measure with a Lipschitz zooming contraction  then all its Lyapunov exponents are positive.  
\end{Proposition}
\begin{proof}
	Suppose that $\mu$ is a $(\alpha,\delta,\ell)$-zooming measure for $\alpha=\{\alpha_n\}_n$ Lipschitz zooming contraction, $\delta>0$ and $\ell\in\NN$. Let $1<a_n<1$ be such that $\alpha_n(r)=a_n r$ for every $n\in\NN$.
It follows Theorem~4 (page 914) in \cite{Pi11} that exists a $(\alpha,\delta,\ell)$-zooming induced return map $(F,B,\cp)$ such that $\mu$ is $F$-liftable to some probability measure $\nu\in\cm^1(F)$.

Let $R$ be the induced time for $F$, that is, $F(x)=g^{R(x)}(x)$. 
Note that $(F|_{P})^{-1}$ is a $a_1$-contraction for every $P\in\cp$, meaning that 
$\dist((F|_{P})^{-1}(x),(F|_{P})^{-1}(y))\le a_1\dist(x,y)$ for every $x,y\in B$.
Thus, given $x\in B_0:=\bigcap_{n\ge 1} F^{-n}(B)$ and $v\in T_xM$, we have that
$$\frac{1}{r_n(x)}\log|Dg^{r_n(x)}(x)v|=\frac{n}{r_n(x)}\frac{1}{n}\log|DF^n(x)v|\ge\frac{n}{r_n(x)}\frac{1}{a_1},$$ where $r_n(x):=\sum_{j=0}^{n-1}R\circ F^j(x)$.
As $\lim_{n\to\infty} \frac{n}{r_n(x)}=\frac{1}{\int R d\nu}>0$ for $\nu$-almost every $x\in B_0$ then the same holds for $\mu$-almost every $x\in B_0$.
Using the ergodic theorem one deduces that $\mu(B_0)>0$, hence there is a positive $\mu$-measure set such that 
$$
\limsup_{n\to\infty}\frac{1}{n}\log|Df^n(x)v|\ge\lambda:=\frac{1}{a_1\int R d\nu}>0 \quad \text{for every $v\in T_xM$.}
$$ 
Thus, by the ergodicity and invariance of $\mu$, we conclude the all Lyapunov exponents of $\mu$ are bigger or equal than $\lambda$, concluding the proof of the proposition.
\end{proof}

\begin{Corollary}\label{Corolkbihvg45678}
Let $M$ be a Riemannian manifold. If $f:M\setminus\cc\to M$ is a $C^{1+}$ local diffeomorphism with $\cc$ being a non-degenerated critical/singular set then  $\cl\cz(f)=\ce(f)$.
\end{Corollary}

\subsection{Uniqueness of expanding equilibrium states}

Using the fact that one considers induced maps there could be at least two sources of problems in analyzing expanding equilibrium states. 
First,  there could exist expanding equilibrium states that are not liftable to the induced scheme that allows to compare all fat-induced expanding measures. Second, the probability measure of largest topological pressure at the induced level could generate a zooming but non-expanding 
probability measure. 
The following auxiliary lemma ensures that the first situation does not occur, that is to say that all expanding equilibrium states are actually fat-induced expanding measures, hence liftable.

\begin{Lemma}\label{LemmaInduhbiuty6}
Suppose that $f$ is strongly transitive and  $\varphi:\XX\to\RR$ is a piecewise Hölder continuous potential. If $\mu\in\ce(f,\ell)$ is an ergodic expanding equilibrium state for $\varphi$, $\ell\ge1$, then $\mu\in\ce^*(f,\ell)$.
\end{Lemma}
\begin{proof}
Let $\mu\in\ce(f,\ell)$ be an ergodic expanding equilibrium state for $\varphi$.
As $\mu\in\ce(f,\ell)$, $\mu$ is a $(\alpha,\delta,\ell)$-zooming measure for some exponential zooming contraction $\alpha=\{\alpha_n\}_n$ and some $\delta>0$. 
Taking $\lambda>0$ be such that $\alpha_n(r)=e^{-\lambda n}r$, we get that $\sum_{n}(e^{-\lambda n})^{1/2}<+\infty$ and so, by Proposition~\ref{Propositioiuiouiu877}, there is a $(\widetilde{\alpha},\delta,\ell)$-zooming return map $(F,B,\cp)$, where $\widetilde{\alpha}_n(r)=e^{-\lambda n/2}r$, such that $\overline{\bigcup_{P\in\cp}P}=\overline{B}$ and $\mu$ has an unique $F$-lift $\nu$.

As $\widetilde{\alpha}$ is a exponential zooming contraction, it follows from Lemma~\ref{Lemma001009110} that \begin{equation}\label{Equationjhgds456}
  \cm^1(f,F)\subset\ce(f).
\end{equation}
By Lemma~\ref{Lemmalihbdcw334}, $\sum_{n\in\NN}V_n(\overline{\varphi})<+\infty$, where $\overline{\varphi}$ is the $F$-lift of $\varphi$.
As $\mu$ is a $F$-liftable expanding equilibrium state for $\varphi$, it follows from Proposition~\ref{Propositionioboi78ihbbnala} and \eqref{Equationjhgds456} that $\nu=\nu_0$, where $\nu_0$ is given by \eqref{Equationkjgi754}.
Thus, $\supp\nu=\supp\nu_0=\overline{\bigcup_{P\in\cp}P}$, proving that $\supp\nu=\overline{B}$.
That is, $\mu\in\ce^*(f,\ell)$.
\end{proof}

The next result shows uniqueness of expanding measures provided that  all $F$-induced invariant probability measures generate $f$-invariant 
expanding probability measures.

\begin{Theorem}\label{Theoremuhb0892877y}
Suppose $\cl\cz(f)=\ce(f)$.
If $f$ is strongly transitive and $\varphi$ is a piecewise Hölder continuous potential, then $f$ has at most one expanding equilibrium state for $\varphi$.
\end{Theorem}

\begin{proof}Using Proposition~\ref{PropositionTotTransPer} of Appendix~I and following the same argument of the proof of Corollary~\ref{Corollarykbiohg4k234}, without loss of generality, changing $f$ by $f^{\ell}$ if necessary, we may assume that $f^n$ is strongly transitive for every $n\ge1$. 

Let $\mu_1, \mu_2\in\ce(f)$ be two ergodic $f$-invariant and expanding equilibrium states for $\varphi$.
Suppose that $\mu_1\in\ce(f,t)$ and $\mu_2\in\ce(f,t')$.
Hence, $\mu_1,\mu_2\in\ce(f,s)$, where $s=$lcm$(t,t')$ is the least common multiple of $t$ and $t'$.
As $(f^s)^n$ is strongly transitive for every $n\ge1$, changing $f$ by $f^s$ if necessary, $\mu_1$ and $\mu_2$ by normalized $f^s$-ergodic components of $\mu_1$ and $\mu_2$ respectively, we may assume without loss of generality that $\mu_1,\mu_2\in\ce(f,1)$ are ergodic probability measures.

Let $\lambda,\delta>0$ and $(F,B,\cp)$ be the $(\beta,\lambda,1)$-zooming return map  given by Theorem~\ref{Theoremuhbkj100201}
 in such a way that $\mu_1$ is $F$-liftable, where $\beta=\{\beta_n\}_n$ and $\beta_n(r)=e^{-\lambda\sqrt{n}}r$. Let  $\overline{\mu}_1$ be the $F$-lift of $\mu_1$.
The assumption $\cl\cz(f)=\ce(f)$, combined with Lemma~\ref{Lemma001009110}, implies that 
every $F$-liftable probability measure is a Lipschitz zooming measure, we get that $\cm^1(f,F)\subset\ce(f)$. 
Then Proposition~\ref{Propositionioboi78ihbbnala} ensures that $\mu_1$ is the unique expanding equilibrium state for $\varphi$ that is $F$-liftable
(the requirements of the proposition are satisfied because of Lemma~\ref{Lemmalihbdcw334}). 
Nevertheless, as $\mu_2\in\ce(f,1)$ is an expanding equilibrium state, it follows from  Lemma~\ref{LemmaInduhbiuty6} that $\mu_2\in\ce^*(f,1)$ and so,  $\mu_2$ is $F$-liftable, proving that $\mu_2=\mu_1$. The conclusion of the theorem now follows as a consequence from the fact that, by ergodic decomposition and the variational principle, almost every element in the ergodic decomposition of an equilibrium state is an ergodic equilibrium state.  \end{proof}

\subsection{Proof of Theorem~\ref{Maintheorem00}}

Let $f$ be strongly transitive, $\cc$ be non-degenerated
and $\varphi$ be a Hölder continuous potential $\varphi$. 
Corollary~\ref{Corolkbihvg45678}, which ensures that $\cl\cz(f)=\ce(f)$, combined with 
Theorem~\ref{Theoremuhb0892877y} guarantees that there exists at most one equilibrium state
with respect to $\varphi$.

\section{Existence of expanding equilibrium states}\label{sec:proofsmain2}

In this section we will study sufficient conditions, on the potentials, which ensure the existence of equilibrium states. 
This will then be used to prove Theorem~\ref{Maintheorem11},~\ref{TheoOLDCorollaryMAINknlielMI}~and ~\ref{TheoCorollaryU8883jljlEMI}. 

\subsection{Criteria for the existence of expanding equilibrium states}

Throughout this subsection we shall consider a slightly more general framework which may be of independent interest
in the context of applications. Suppose that $\XX$ is a compact metric space, $\cc\subset\XX$ is compact set  with empty interior and 
$f: \XX \setminus \cc \to \XX$ is strongly transitive, it has a continuous extension $\bar{f}:\XX\to \XX$ and $\sup_{x\notin\cc}\DD f(x)<+\infty$. 
Assume further that $\cl\cz(f)=\ce(f)$, which is to say that all Lipschitz zooming measures are expanding measures.

Let also $\cm^1(\bar{f}\,)$ denotes the set of all $\bar{f}$-invariant probability measures.
The first results on the existence of expanding equilibrium states explore upper semicontinuity of entropy among expanding measures.

Theorem~\ref{Maintheorem22} below gives a criterium for the existence of expanding equilibrium states for Hölder continuous potentials with small oscillation. 

\begin{Theorem}\label{Maintheorem22} 
If $f$ has a probability measure of maximal expanding entropy $\mu_0$, then
there exists $\delta_0>0$ such that  $f$ has a unique expanding equilibrium state $\mu_{\psi}$ for any piecewise Hölder continuous potential $\psi$ with oscillation smaller that $\delta_0$.
\end{Theorem}
\begin{proof}
Proceeding as in the proof of Theorem~\ref{Theoremuhb0892877y}, we may assume that $f^n$ is strongly transitive for every $n\in\NN$. 
Changing $f$ by $f^{\ell}$ if necessary, we may assume that $\mu_0\in\ce^1(f,1)$.
Let $(F,B,\cp)$ be the $(\beta,\delta,1)$-zooming return map  given by Theorem~\ref{Theoremuhbkj100201}, with $\delta>0$, $\beta=\{\beta_n\}_n$ and $\beta_n(r)=e^{-\lambda\sqrt{n}}r$ for some $\lambda>0$ $($\footnote{ Note that, instead of Proposition~\ref{Propositionioboi78ihbbnala}, we can use Theorem~\ref{theo.technical1} applied to $(F,B,\cp)$ to conclude that $\mu_0$ is the unique probability measure of maximal expanding entropy for $f$, providing a proof of the unicity of the probability measure of maximal expanding entropy that is independent of Theorem~\ref{Theoremuhb0892877y}/Proposition~\ref{Propositionioboi78ihbbnala}.}$)$.

By Theorem~\ref{theo.technical1}, the $F$ lift of $\mu_0$ is $\nu_0$.
In particular, $\int R d\nu_0<+\infty$, where $R$ is the induced time of $F$.
Take $$\delta_0:=\frac{1}{2}\frac{1}{\int R d\nu_0}\sum_{n\in\NN} H(\nu_0(\{R=n\})).$$

If a piecewise Hölder continuous potential $\varphi$ has an expanding equilibrium, this must be the unique expanding equilibrium for $\varphi$ by Theorem~\ref{Theoremuhb0892877y}.
Hence, we may assume by contradiction that exists a sequence $\varphi_k$ of piecewise Hölder functions such that $osc(\varphi_k)\to0$ and $\varphi_k$ does not have an expanding equilibrium state for every $k\ge1$.
By Lemma~\ref{Lemmalihbdcw334}, $\sum_{n\in\NN}V_n(\varphi_k)<+\infty$ for every $k\ge1$.
Taking $a_n=\inf_x\varphi_n(x)$ and considering the piecewise Hölder functions $\psi_n=\varphi_n-a_n\ge0$, we get that $\|\psi_n\|_{\sup}\to0$, $\sum_{n\in\NN}V_n(\psi_n)<+\infty$ and $\psi$ does not have an expanding equilibrium state for every $n\in\NN$.

As $\cl\cz(f)=\ce(f)$ and $\beta$ is a Lipschitz zooming contraction, it follows from  Lemma~\ref{Lemma001009110} that $$\cm^1(f,F)\subset\ce(f).$$
Hence, as $\psi_k$ does not have an expanding equilibrium state, one can choose $\mu_k\in\cm^1(f,F)$ such that $h_{\mu_k}(f)+\int\psi_k d\mu_k>P(\psi_k,f,F)-1/k$ $\to h(\ce(f))$ and
\begin{equation}\label{Equationlkjlhyres6}
  \int R\;d\,\overline{\mu}_k\ge k,
\end{equation}
where $\overline{\mu_k}\in\cm^1(F)$ is the $F$-lift of $\mu_k$.
As $\int\psi_k d\mu_k\to0$, we have that $h_{\mu_k}(f)\to h(\ce(f))$.
Hence, $k_0\ge1$ such that $h_{\mu_k}(f)\ge \gamma:=h(\ce(f))-\delta_0$ for every $k\ge k_0$.
By Theorem~\ref{theo.technical1}, $\int R d \overline{\mu}_k\le C_{\gamma}$ for some  $C_{\gamma}>0$ and every $k\ge k_0$, a contradiction with \eqref{Equationlkjlhyres6}. 
\end{proof}

\begin{Corollary}\label{CorollaryUPPkjljlEMI}
 Suppose that there exists a finite cover $\cu$ of $\XX\setminus\cc$ by open sets such that $f|_{U}$ is injective $\forall U\in\cu$ and $\mu(\cc)=0$ for every $\mu\in\overline{\ce(f)}$ $($\footnote{ Theses additional hypothesis ensure the ``semicontinuity of the entropy for expanding measures'' (Corollary~\ref{CorollaryUPPERSEMI}). 
Moreover, the hypothesis of $\mu(\cc)=0$ for every $\mu\in\overline{\ce(f)}$ is automatically satisfied when $\XX$ is a Riemannian manifold and $\bar{f}$ is a $C^1$ map, see Lemma~\ref{Lemmacritico} of Appendix I.}$)$. 
If  $h_{\nu}(f)<h(\ce(f))$ for every $\nu\in\cm^1(\bar{f}\,)\setminus\ce(f)$ then there exists $\delta_0>0$ such that
\begin{enumerate}
\item $f$ has a unique probability measure of maximal entropy and it is an expanding measure.
\item $f$ has a unique expanding equilibrium state $\mu_{\psi}$ for any piecewise Hölder continuous potential $\psi$ with oscillation smaller that $\delta_0$. 
\end{enumerate}
\end{Corollary}
\begin{proof}
Let $\mu_n\in\ce(f)$ be such that $\lim_n h_{\mu_n}(f)=h(\ce(f))$.
Taking a subsequence, we may assume that $\lim_n\mu_n=\mu\in\cm^1(\bar{f}\,)$.
By Corollary~\ref{CorollaryUPPERSEMI}, $h_{\mu}(f)\ge\lim_n h_{\mu_n}(f)=h(\ce(f))$.
As we are assuming that $h_{\nu}(f)<h(\ce(f))$ for every $\nu\in\cm^1(\bar{f}\,)\setminus\ce(f)$, we have that $\mu\in\ce(f)$ and $h_{\mu}(f)=h(\ce(f))$ is a measure of maximal entropy and so, we can apply  Theorem~\ref{Maintheorem22} to conclude the proof.
\end{proof}

\begin{Lemma}\label{LemmaUlknlielEMI} 
Let $\varphi$ be a piecewise Hölder continuous expanding potential, i.e.,  
$ h_{\nu}(\bar f)+ \int\varphi \,d\nu<P_{\ce(f)}(\varphi)$ for every $\nu\in\cm^1(\bar{f}\,)\setminus\ce(f)$. 
If there exists a finite cover $\cu$ of $\XX\setminus\cc$ by open sets such that $f|_{U}$ is injective $\forall U\in\cu$ and $\mu(\cc)=0$ for every $\mu\in\overline{\ce(f)}$, 
 then $f$ has a unique equilibrium state for $\varphi$, and it is an expanding measure.
\end{Lemma}
\begin{proof}
Let $(\mu_n)_n$ be a sequence in $\ce(f)$ such that  $\lim_{n\to\infty} h_{\mu_n}(f)+\int\varphi \,d\mu_n= P_{\ce(f)}(\varphi)$.
Up to consider a subsequence, if necessary, we may assume that $(\mu_n)_n$ converges to $\mu\in\cm^1(\bar{f}\,)$.
As $\cm^1(\XX)\ni\nu\mapsto\int\varphi d\nu$ is a continuous map it follows from upper semicontinuity of the entropy map for expanding measures (cf. Corollary~\ref{CorollaryUPPERSEMI}) that 
$$
h_{\mu}(\bar f) +\int\varphi \,d\mu\ge P_{\ce(f)}(\varphi).
$$

Combined with the assumption, this ensures that $\mu\in\ce(f)$ (hence $h_{\mu}(\bar f)=h_{\mu}(f)$) and
\begin{align*}
P_{\ce(f)}(\varphi) \le \Ptop(\bar f,\varphi)
	& = \sup\Big\{h_{\nu}(\bar f) +\int\varphi \,d\nu \, ;\,  \nu\in\cm^1(\bar{f}\,) \Big\} \\
	& = \sup\Big\{h_{\nu}(f) +\int\varphi \,d\nu \, ;\, \nu\in\ce(f) \Big\}  \\
	& = P_{\ce(f)}(\varphi) \leq h_{\mu}(f) +\int\varphi \,d\mu.
\end{align*}
This shows that all quantities above coincide and that 
$\mu$ is an expanding equilibrium state for $\varphi$.
Uniqueness is a direct consequence of Theorem~\ref{Theoremuhb0892877y}. This finishes the proof of the lemma.
\end{proof}

\begin{Corollary}\label{CorollaryUPPkk34524er} 
Suppose that there exists a finite cover $\cu$ of $\XX\setminus\cc$ by open sets such that $f|_{U}$ is injective $\forall U\in\cu$ and $\mu(\cc)=0$ for every $\mu\in\overline{\ce(f)}$. 
If $\sup\{h_{\nu}(f)\,;\,\nu\in\cm^1(\bar{f}\,)\setminus\ce(f)\}<h(\ce(f))$ 
then $f$ has a unique equilibrium state $\mu_{\psi}$ for any piecewise Hölder continuous potential $\psi$ with oscillation smaller that $\delta_0$, and it is an expanding measure. 
\end{Corollary}
\begin{proof}
Using the compactness of $\XX$ and continuity of $\varphi$, one can note that
\begin{align*}
h(\ce(f)) + \min \varphi \le P_{\ce(f)}(\varphi) 
		 = \sup\Big\{h_{\nu}(f) +\int\varphi \,d\nu \, ;\, \nu\in\ce(f) \Big\} 
		 \le h(\ce(f)) + \max \varphi.
\end{align*}
By the assumption, there exists $\delta_0>0$ so that 
$h_{\nu}(\bar f) < h(\ce(f)) -  \delta_0$ 
for every ${\nu\in\cm^1(\bar{f}\,)\setminus\ce(f)}$.
In case $\osc(\varphi)<\delta_0$, we obtain that for each $\nu\in\cm^1(\bar{f}\,)\setminus\ce(f)$,
$$
h_{\nu}(\bar f)+ \int\varphi \,d\nu
  < h(\ce(f)) -\delta_0 + \max\varphi
  \le  h(\ce(f)) + \min\varphi
	\le P_{\ce(f)}(\varphi). 
$$ 
In other words, every potential $\varphi$ verifying $\osc(\varphi)<\delta_0$ also satisfies the requirements of
Lemma~\ref{LemmaUlknlielEMI}, hence it admits a unique equilibrium state, which is an expanding measure. 
\end{proof}

\subsection{Proof of Theorem~\ref{Maintheorem11} }
Let $f$ be strongly transitive, $\cc$ is non-degenerated and assume that $\sup_{x\in M\setminus\cc}\|D f(x)\|<$ $+\infty$.
The theorem is a direct consequence of Theorem~\ref{Maintheorem22}, because 
all Lipschitz zooming measures are actually expanding (recall Proposition~\ref{Propkbihvg45678}).
\hfill $\square$

\subsection{Proof of Theorem~\ref{TheoOLDCorollaryMAINknlielMI}}
Let $f$ be strongly transitive,  $\cc$ be non-degenerated and assume that the extension $\bar{f}$ of $f$ to $M$ is $C^1$.
In particular, we have that $\sup_{x\in M\setminus\cc}\|D f(x)\|<$ $+\infty$.
It follows from Corollary~\ref{Corolkbihvg45678} and Lemma~\ref{Lemmacritico} of Appendix that $\cl\cz(f)=\ce(f)$ and  $\mu(\cc)=0$ for every $\mu\in\overline{\ce(f)}$.
Thus, if $\varphi$
is a Hölder continuous and expanding potential then it follows from Lemma~\ref{LemmaUlknlielEMI} that $f$ has a unique equilibrium state for $\varphi$, and it is an expanding measure.
\hfill $\square$

\subsection{Proof of Theorem~\ref{TheoCorollaryU8883jljlEMI}}
Let $f$ be strongly transitive and $\cc$ be non-degenerated
so that $\sup_{x\in M\setminus\cc}\|D f(x)\|<+\infty$.
In view of Corollary~\ref{Corolkbihvg45678} and Lemma~\ref{Lemmacritico} of Appendix, we have that  $\cl\cz(f)=\ce(f)$ and  $\mu(\cc)=0$ for every $\mu\in\overline{\ce(f)}$.
Hence, the first item of the corollary follows from  Corollary~\ref{CorollaryUPPkjljlEMI}, as well as the  second item follows from Corollary~\ref{CorollaryUPPkk34524er}.
\hfill $\square$

\section{Applications}\label{SectionApplications}

\subsection{Uniformly expanding maps}
An immediate application of Theorem~\ref{Theoremuhb0892877y} is when $f:M\to M$ is $C^1$ expanding map defined on a connected compact manifold $M$. In this case every measure is expanding. Indeed, there is a $\delta>0$, the ``radius of the inverse branches'', and a exponential zooming contraction $\alpha$ such that every $f$ invariant measure is an $(\alpha,\delta,1)$-zooming measure.
Moreover, as it is well known that such map is strongly transitive, Theorem~\ref{Theoremuhb0892877y} implies the unicity of the equilibrium state of a Hölder continuous potential for a $C^1$ expanding map on compact connected manifold.
As the existence of a equilibrium state follows from Corollary~\ref{CorollaryUPPERSEMI} $($\footnote{ As, in this case,  all measures are expanding, Corollary~\ref{CorollaryUPPERSEMI} implies the upper semi-continuity of $\mu\mapsto h_{\mu}(f)$ and this implies the upper semi-continuity of $\mu\mapsto h_{\mu}(f)+\int\varphi d\mu$ for every continuous potential $\varphi$ and, therefore, the existence of equilibrium states follows from the compactness of $\cm^1(f)$.}$)$, we conclude a proof of Ruelle's Theorem for (uniformly) expanding maps using only local full induced Markov maps (i.e., countable full shifts instead of finite subshifts).

\subsection{Manneville-Pommeau maps}\label{se:MP}
Consider the Manneville-Pomeau family of interval maps $f_{\alpha}:[0,1]\setminus\{\frac{1}{2}\}\to [0,1]$, $\alpha>0$, given by
\begin{equation}\label{eq:LSV}
f_{\alpha}(x)= \left\{\begin{array}{ll}
x\,(1+2^\alpha x^{\alpha}) & \quad , \text{if } x<\frac{1}{2} \\[0.2cm]
2x-1 & \quad , \text{if } x>\frac{1}{2}
\end{array}
\right.
\end{equation}
Even though these maps cannot be continuously extended to the interval, such construction can be realized as a bi-Lipschitz local homeomorphism on the circle $S^1=\RR/\ZZ$.
That is, identifying $f_{\alpha}$ with the map $\RR/\ZZ\setminus\{[1/2]\}\ni[x]\mapsto [f_{\alpha}(x)]\in\RR/\ZZ$, it can be continuously extended to the bi-Lipschitz local homeomorphism $\bar{f}_{\alpha}:\RR/\ZZ \to \RR/\ZZ $ by setting $\bar{f}_{\alpha}([1/2])=[0]$, where $[x]=\{y\in\RR\,;\,x-y\in\ZZ\}$.
It is known that the metric entropy is upper-semicontinuous, hence every continuous potential has some equilibrium state. However, this family exhibits phase transitions with respect to the family of (piecewise) Hölder continuous potentials 
$\varphi_{\alpha,t} = -t\,\log \,|f_\alpha^{\prime}|$, parameterized by $t \in \mathbb R$ (see e.g. to \cite{BTT,Lo} for more details). 
For instance, if $\alpha\ge 1$ then: 
\begin{itemize}
\item[(i)] for every $t<1$ there exists a unique equilibrium state $\mu_{\alpha,t}$ with respect to the potential $\varphi_{\alpha,t}$ and $\mu_{\alpha,t}\in \mathcal E(f_\alpha)$
\item[(ii)] for $t =1$ the Dirac measure $\delta_0$ is the unique equilibrium state for $\varphi_{\alpha,t}$ and 
$$
P_{\mathcal E(f_\alpha)}(\varphi_{\alpha,t}) = P(f_\alpha,\varphi_{\alpha,t}),
$$
hence, there exists no expanding equilibrium state with respect to the potential $\varphi_{\alpha,1}$;
\item[(iii)] for $t > 1$ the Dirac measure $\delta_0$ is the unique equilibrium state for $\varphi_{\alpha,t}$. 
\end{itemize}
Furthermore, for $\alpha\ge 1$ the pressure function $\mathbb R \ni t\mapsto P(f_\alpha, \varphi_{\alpha,t})$ is known to be $C^1$-smooth (see e.g. \cite{Lo}), a property which, in view of the upper semicontinuity of the metric entropy, is equivalent to the uniqueness of the equilibrium state for each potential 
$\varphi_{\alpha,t}$, $t\in \mathbb R$
(cf. \cite{Wa2}).   In case $\alpha\in (0,1)$ there exists an $f_\alpha$-invariant probability measure absolutely continuous with respect
to Lebesgue, hence there exist two equilibrium states for the potential  $\varphi_{\alpha,1}$.

\medskip
It is also worth mentioning that there exist sequences of expanding probability measures  converging in the weak$^*$ topology to $\delta_0$  
and, by upper semicontinuity of the measure theoretical entropy,  
their entropies converge to zero
(\footnote{As the Manneville-Pomeau map is topologically conjugate to the doubling map then its space of ergodic probability measures
is pathwise connected and entropy dense, as a consequence of \cite{GP}.}).
In consequence, given $\varphi\in C^0(\RR/\ZZ, \mathbb R)$ then the map $\mathcal M^1(f_\alpha) \ni \mu \mapsto h_{\mu}(f_\alpha)+\int \varphi\, d\mu$ 
is continuous at $\mu=\delta_0$, and so
\begin{align*}
P_{\ce(f_\alpha)}(\varphi) & =\sup\Big\{h_{\mu}(f_\alpha)+\int\varphi \,d\mu\,;\,\mu\in\ce(f_\alpha)\Big\} \\
	& = \sup\Big\{h_{\mu}(f_\alpha)+\int\varphi \,d\mu\,;\,\mu\in\mathcal M^1(f_\alpha)\Big\} = P(f_\alpha,\varphi).
\end{align*}
Therefore, the map
\begin{equation}\label{eq.mpnonc1exp}
\mathbb R \ni t\mapsto P_{\mathcal E(f_\alpha)}(\varphi_{\alpha,t}) = P(f_\alpha,\varphi_{\alpha,t}).  
\end{equation}
is $C^1$-smooth whenever $\alpha \ge1$,  and it is not $C^1$-smooth in case $\alpha\in (0,1).$ This shows that the $C^1$-smoothness of the 
expanding pressure function may depend intrinsically on the space of non-expanding measures.

\subsection{Nonuniformly expanding local diffeomorphisms}
Let $f:M\to M$ be a $C^{1+}$ local diffeomorphism defined on compact Riemannian manifold $M$.
As $f$ does not have critical set then $\ce(f)$ is the set of all $f$-invariant probability measures $\mu$ having only positive Lyapunov exponents, i.e.
$$
\ce(f)=\bigg\{\mu\in\cm^1(f)\,;\,\lim_n\frac{1}{n}\log\|(Df^n(x))^{-1}\|^{-1}>0\text{ for $\mu$-almost all }x\in M\bigg\}.
$$
Recall that $h(\ce(f)):=\sup\{h_{\mu}(f)\,;\,\mu\in\ce(f)\}.$
As $\cc=\emptyset$, Theorems~\ref{TheoOLDCorollaryMAINknlielMI} and \ref{TheoCorollaryU8883jljlEMI}
have the following immediate consequences:

\begin{Theorem}\label{Theoremkhviv}
Assume that $f$ is strongly transitive and $\ce(f)\neq\emptyset$. The following hold:
\begin{itemize}
\item[(i)] if $f$ satisfies
\begin{equation}\label{Equatioouiih}
  h_{\mu}(f)<h(\ce(f))\text{ for every }\mu\in\cm^1(f)\setminus\ce(f),
\end{equation}
 then there is $\delta>0$ such that $f$ has a unique equilibrium state $\mu_{\varphi}$ for any Hölder continuous potential $\varphi$ with $\sup\varphi-\inf\varphi<\delta$. 
\item[(ii)] if $\varphi$ is a Hölder continuous potential satisfying 
\begin{equation}\label{Equationljfrhjq}
  h_{\mu}(f)+\int\varphi \,d\mu<P_{\ce(f)}(\varphi)\;\text{ for every }\;\mu\in\cm^1(f)\setminus\ce(f)
\end{equation}
then $\varphi$ has a unique equilibrium state $\mu_{\varphi}$.
\end{itemize}
Moreover, in both cases $\mu_{\varphi}\in\ce(f)$ and $\supp\mu_{\varphi}=M$.
\end{Theorem}

This theorem can be applied to the open class of multidimensional nonuniformly expanding maps studied by Oliveira, Ramos, Varandas and Viana \cite{OV08,RV17,VV}. In fact, Theorem~\ref{Theoremkhviv} allows to extend these previous results. Item (i) above is an extension of the results
in \cite{OV08,VV},  since the assumption ~\eqref{Equatioouiih} is weaker 
than the requirement
$$\sup_{\mu\in\cm^1(f)\setminus\ce(f,1)} h_{\mu}(f) \, <\, h(\ce(f,1))$$ 
used in those papers, where  
 $$\ce(f,1)=\bigg\{\mu\in\cm^1(f)\,;\,\lim_{n \to\infty} \frac{1}{n}\sum_{j=0}^{n-1}\log\|(Df\circ f^j(x))^{-1}\|^{-1}>0\text{ for $\mu$-a.e. }x\in M\bigg\}$$ and $h(\ce(f,1))=\sup\{h_{\mu}(f)\,;\,\mu\in\ce(f,1)\}$, 
 and the assumption ~\eqref{Equatioouiih} is weaker than the assumption 
Item (ii) above offers an extension to \cite{RV17}, as instead of \eqref{Equationljfrhjq}, Ramos and Viana 
assumed the stronger hypothesis 
$$  \sup_{\mu\in\cm^1(f)\setminus\ce(f,1)} \; \Big[h_{\mu}(f)+\int\varphi \,d\mu \Big] \,<\, P_{\ce(f,1)}(\varphi) $$ 
where $P_{\ce(f,1)}(\varphi)=\sup\{h_{\mu}(f)+\int\varphi \,d\mu\,;\,\mu\in\ce(f,1)\}$.

\begin{Remark}
The function $C^0(\RR/\ZZ, \mathbb R) \ni \varphi\mapsto P_{\ce(f)}(\varphi)$ defined by
$$
P_{\ce(f)}(\varphi):=\sup\Big\{h_{\mu}(f)+\int\varphi \,d\mu\,;\,\mu\in\ce(f)\Big\}
$$
is a pressure function in the sense in \cite{BCMV} (i.e. it is monotone, translation invariant and convex).
Theorem~5 in \cite{BCMV} implies that an upper semicontinuous entropy function $\mathfrak{h}: \mathcal M^1(f) \to \mathbb R$ so that 
$$
P_{\ce(f)}(\varphi)=\max\Big\{{\mathfrak h}_{\mu}(f)+\int\varphi \,d\mu\,;\,\mu\in\mathcal M^1(f)\Big\},
$$
where 
$
\mathfrak h_\mu(f) :=\inf \Big\{ P_{\ce(f)}(\psi) - \int \psi\, d\mu \;;\; \psi \in C^0(\RR/\ZZ)  \Big\}.
$
is such that $0\le h_\mu(f) \le \mathfrak h_\mu(f)$ for every $\mu\in \mathcal M^1(f)$.
In particular, there always exist $\mathfrak h$-generalized equilibrium states.
Theorem~2 and Corollary~4 in \cite{BCMV}
say that
$
\mathbb R \ni t \mapsto P_{\ce(f)}(t \varphi)
$
is $C^1$-smooth if and only if there exists a unique tangent functional at every point $t_0\in \mathbb R$.
In case of the Manneville-Pommeau maps the expanding pressure can be non-smooth. It is an open question whether the $C^1$-smoothness of the expanding pressure function can be established among potentials $\varphi$ satisfying the open
condition $P_{\mathcal E(f)}(\varphi) < P(f,\varphi)$.
\smallskip
\end{Remark}

In the next subsection we provide a specific example which illustrates the applicability of our results in this context.

\subsubsection{Maps with a curve of indifferent fixed points}

Consider the skew-product map
$$
\begin{array}{rccc}
g: & S^1\times [0,1] & \to & S^1\times [0,1] \\
	& (\theta,x) & \mapsto &  (4\theta (\,\text{mod} 1\,), f_{\alpha(\theta)}(x))
\end{array}
$$
considered by Gou\"ezel in \cite{Go07}, where $\alpha: S^1 \to [\alpha_{min}, \alpha_{max}]$ is $C^2$-smooth for some constants satisfying
$0<\alpha_{min} <  \alpha_{max} <1$ and $\alpha_{max} <\frac32 \alpha_{min}$, the map $\alpha(\cdot)$ attains the minimum value at a single non-flat point $\theta_0$, and $f_\alpha$ is the family of Manneville-Pommeau maps defined by ~\eqref{eq:LSV}. 
The map $g$ has a unique $g$-invariant probability measure $\mu_{SRB}$ absolutely continuous with respect to Lebesgue (cf. \cite[Theorem~2.10]{Go07}), and this is an expanding and ergodic probability measure.  In particular, as $g$ is $C^{1+}$-smooth the Pesin formula ensures that
$$
\htop(g)\ge h_{\mu_{SRB}}(g) =  \sum_{i=1,2} \lambda^+_i(\mu_{SRB}) > \log 4,
$$
where $\log 4= \lambda^+_1(\mu_{SRB})> \lambda^+_2(\mu_{SRB})>0$ are the (positive) Lyapunov exponents of $\mu_{SRB}$.
The probability measure $\mu_{SRB}$ is liftable to the full induced Markov map obtained as the first return time map $F$ 
to the set $S^1\times (\frac12,1]$. 
Note that the set of points that fail to return to $S^1\times (\frac12,1]$ are either wandering points (lying in some preimage of 
the curve $S^1 \times \{\frac12\}$) or belong the the curve with indifferent fixed points. Thus, 
as $F$ is a first return map one concludes that every expanding measure for $g$ lifts to $F$.
Furthermore, there exists an integer $q\ge 1$ so that $\# \{R=n\} = 4^{q+n}$ for every $n\ge 1$ (see \cite[Section~2.3]{Go07} for more details), and so
$$
\limsup_{n\to\infty} \frac1n \log \#\{R=n\} < \htop(g).
$$
Finally, as all probability measures with entropy larger or equal than $h_{\mu_{SRB}}(g)$ are expanding measures one can use 
 Theorem~\ref{theo.technical1} to conclude the following:

 \begin{Corollary}
 The following properties hold:
\begin{itemize}
\item[(i)] There exists a unique measure of maximal entropy $\mu_0$ for $g$, $\mu_0$ is an expanding measure and $\mu_0$  
	has exponential decay of correlations in the space of Hölder continuous observables.
\item[(ii)]
There exists a unique and expanding equilibrium state for every Hölder continuous potential $\varphi$ with small oscillation.
\end{itemize}
 \end{Corollary}

\subsection{ Interval maps with critical points}\label{SectioINTmaps}

Let $f:[0,1]\to[0,1]$ be a $C^{1+}$ transitive map with a non-flat critical set $\cc$.
It is known that every continuous transitive interval map is strongly transitive (see e.g. Proposition~4.10 in \cite{Pi21b}) and $h_{top}(f)>0$ (cf. Corollary~4.6.11 in \cite{ALM} or 
\cite{BlC}).
By the Ruelle inequality, every ergodic invariant probability measure with positive entropy also has positive Lyapunov exponent.
As the critical region is non-flat, we have that
\begin{equation*}\label{Equationnn12321}
  \nu\in\cm_{erg}^1(f)\text{ and }h_{\nu}(f)>0\implies\nu\in\ce(f).
\end{equation*}
This implies that $f$ has infinitely many expanding periodic orbits $($\footnote{ By \cite{Pi11}, for each ergodic expanding probability measure there is a full induced Markov map $(F,B,\cp)$ such that $\mu$ is $F$-liftable.
If $h_{\mu}(f)>0$ then we must have that $\#\cp\ge2$ and this implies that exists $\Lambda\subset\cap_{n}F^{-n}(B)$ such that $F|_{\Lambda}$ is conjugated to the shift $\sigma:\Sigma_2^+\to \Sigma_2^+$, proving that $F$ and also $f$ has infinitely many expanding periodic orbits.}$)$.
As $\cc$ is non-flat and $f$ is an interval map, $\cc$ is a finite set and so, there is an expanding periodic point $p$ such that $p\notin\co_f^+(\cc)$.
Hence, by Theorem~5 in \cite{Pi11}, $f$ has an ergodic expanding probability measure $\mu_f$ such that $\supp\mu_f=[0,1]$.
 We can use Lemma~\ref{Lemmahhhh1b6} at Appendix II  to conclude that the set $\cf(f)$ of $f$-free points, which are those with dense pre-orbits not intersecting the critical region (see Definition~\ref{def:freeconfined})  
is an open and dense subset of $[0,1]$. 
Letting $\cs:=\cm_{erg}^1(f|_{\co_f^+(\cc)})=\{\nu\in\cm_{erg}^1(f)\,;\,\nu(\co_f^+(\cc))>0\}$, we have that 
\begin{equation}\label{Equationqqwqw}
  \nu\in\cs\iff\nu=\frac{1}{n}\sum_{j=0}^{n-1}\delta_{f^j(p)},
\end{equation} for some $p=f^n(p)\in\co_f^+(\cc)$ and $n\in\NN$.
In particular, $\#\cs\le\#\cc$ and $h_{\nu}(f)=0$ for every $\nu\in\cs$.
Furthermore,
\begin{equation}\label{Eqautionnnmnmn}
  \cm_{erg}^1(f) = \{\mu \in \cm_{erg}^1(f) ; \mu(\cf(f))=1 \} \cup \{\mu \in \cm_{erg}^1(f) ; \mu\in\cs \}.
\end{equation}

 Theorem~\ref{Theoremhgjvin09} at Appendix II shows that the open and denseness of the set of free points serves as a sufficient condition
to carry the thermodynamic formalism in case of maps with singularities and free expanding potentials (see Theorem~\ref{Theoremhgjvin09} for the precise statements). In the special case of the interval map $f$, this can be used to obtain the following:

\begin{Theorem}\label{TheoremTFtoIM}
Let $f:[0,1]\to [0,1]$ be a $C^{1+}$ transitive map with a non-flat critical set $\cc$. If $\varphi$ a Hölder continuous potential such that 
\begin{enumerate}
\item[(i)] $\sup\varphi-\inf\varphi<h_{top}(f)$ or 
\item[(ii)] $\int\varphi \,d\mu<P(f,\varphi)$ for every $\mu\in\cm^1(f)$,
\end{enumerate}
then $f$ has a unique equilibrium state $\mu_{\varphi}$ for $\varphi$.
Moreover, $h_{\mu_{\varphi}}(f)>0$ and $\supp\mu_{\varphi}=[0,1]$.
\end{Theorem}
\begin{proof}
First assume $\varphi$ satisfies condition (ii) above. Note that
\begin{equation}\label{Equationkjbr4678001}
  \gamma:=\sup\bigg\{\int\varphi \,d\mu\,;\,\mu\in\cm^1(f)\bigg\}<P(f,\varphi).
\end{equation}
Otherwise, there would exist a sequence $\mu_n\in\cm^1(f)$ such that 
$  \lim_{n\to+\infty}\int\varphi \,d\mu_n= P(f,\varphi)$ and any 
accumulation point $\mu$ of $\{\mu_{n}\}_n$ would satisfy $ \int\varphi \,d\mu=P(f,\varphi)$, contradicting $(ii)$.
By \eqref{Equationqqwqw} and \eqref{Eqautionnnmnmn}, one gets that $h_{\mu}(f)=0$ for every $\mu\in\cm^1(f)\setminus\ce(f|_{\cf(f)})$.
Combined with 
\eqref{Equationkjbr4678001} this implies that 
$$
h_{\mu}(f)+\int\varphi \,d\mu\le\gamma \quad\text{ for every}\quad \mu\in\cm^1(f)\setminus\ce(f|_{\cf(f)}).
$$
As this implies that $P_{\ce(f|_{\cf(f)})}(\varphi)=P(f,\varphi)$, we conclude that 
$$h_{\mu}(f)+\int\varphi \,d\mu<P_{\ce(f|_{\cf(f)})}(\varphi)\text{ for every }\mu\in\cm^1(f)\setminus\ce(f|_{\cf(f)}),$$
 proving that $\varphi$ is a free expanding potential.
Therefore, by item $(8)$ of Theorem~\ref{Theoremhgjvin09}, $\varphi$ has a unique equilibrium state $\mu_{\varphi}$.
Moreover, $\mu_{\varphi}\in\ce(f)$ and $\supp\mu_{\varphi}=\supp\mu_f=[0,1]$.
In particular, $f$ has a unique $\mu_0\in\cm^1(f)$ with maximal entropy,  $\mu_0\in\ce(f)$ and $\mu_0(\cf(f))=1$.

\medskip
Now, assume that $\varphi$ satisfies condition (i).
By the assumption, the Hölder continuous potential $\psi$ defined by $\psi(x)=\varphi(x)-\inf\varphi$ satisfies
$$
  0\le\psi(x)<h_{top}(f)\text{ for every }x\in[0,1].
$$
Now, observe that if $h_{\mu}(f)=0$ then 
$$
0\le h_{\mu}(f)+\int\psi d\mu=\int\psi d\mu< h_{top}(f)=h_{\mu_0}(f) \le  h_{\mu_0}(f)+\int\psi d\mu_0\le P_{\ce(f|_{\cf(f)})}(\psi).
$$
In other words, every zero entropy probability measure $\mu$ satisfies
\begin{equation*}\label{Equationlgyr7789}
\;h_{\mu}(f)+\int\psi d\mu<P_{\ce(f|_{\cf(f)})}(\psi).
\end{equation*}
Thus, by \eqref{Equationqqwqw} and \eqref{Eqautionnnmnmn}, we get that
$$
h_{\mu}(f)+\int\psi d\mu<P_{\ce(f)}(\psi)\text{ for every }\mu\in\cm^1(f)\setminus\ce(f|_{\cf(f)}),
$$ proving that $\psi$ is an expanding potential.
Using once more item $(8)$ of Theorem~\ref{Theoremhgjvin09}, we conclude that $\psi$ has a unique equilibrium state $\mu_{\psi}$, that$\mu_{\psi}\in\ce(f)$ and $\supp\mu_{\psi}=\supp\mu_f=[0,1]$. The proof of the theorem is now complete, by noting that $\varphi$ and $\psi$ share the same equilibrium states.
 \end{proof}

\begin{Remark}
There are many results in the literature for interval maps closed related with Theorem~\ref{TheoremTFtoIM}.
For instance, Bruin and Todd  \cite{BT} consudered potentials satisfying condition $(i)$ and Li and Rivera-Letelier \cite{LR-L} considered potentials satisfying condition $(ii)$ $($\footnote{ Note that, if $\sup_x\varphi(x)<h_{top}(f)$, or $\sup_x\frac{1}{n}\sum_{j=0}^{n-1}\varphi\circ f^j(x)<h_{top}(f)$ for some $n\ge1$, then $\sup\{\int\varphi \,d\mu\,;\,\mu\in\cm^1(f)\}<h_{top}(f)$. Thus, condition $(2)$ is weaker than $\varphi$ being a  hyperbolic potential as defined in \cite{LR-L}}$)$.
As the constant potential $\varphi\equiv0$ always meets the assumptions of Theorem~\ref{TheoremTFtoIM}, 
in the special case of $C^{1+}$ non-flat interval maps this theorem gives an alternative proof of the
intrinsic ergodicity of piecewise monotonic transformations with positive entropy established more generally by Hofbauer \cite{Ho}. 
The intrinsic ergodicity of $C^{\infty}$ maps was established by Buzzi \cite{Bu97}, even for maps with infinitely many critical points (and infinitely many intervals of monotonicity). For $C^1$ maps that are not piecewise monotonic see \cite{BR06}. 
\end{Remark}

\subsection{Viana maps}\label{SectionVmaps}
Let $S^1=\RR/\ZZ$ be the unitary circle, $d\ge 16$, $\alpha>0$, $\sigma:S^1\to S^1 $ given by $\sigma(\theta)=d\,\theta$ mod $\ZZ$ and $g_{\alpha}:S^1\times\RR\to S^1\times\RR$ given by $$g_{\alpha}(\theta,x)=(\sigma(\theta),a_0+\alpha\sin(2\pi\theta)-x^2),$$
where $a_0$ is such that the point $0\in\RR$ is pre-periodic to the quadratic map $q(x):= a_0+x^2$.
Viana \cite{Vi} proved that there exists  $\alpha>0$ small, a closed  interval $I\subset(-2,2)$ and   $C^3$ small neighborhood $\cn$ of $g_\alpha$ such that if $g \in\cn$ then
\begin{enumerate}
\item $g(S^1\times I)\subset S^1\times I$;
\item $\bigcap_{n\ge0} g^n(S^1\times I)$ is a forward invariant compact set with nonempty interior;
\item Lebesgue almost every point $p\in\bigcup_{n\ge0} g^n(S^1\times I)$ has all its Lyapunov exponents positive (with respect to $g$);
\item the critical set of $\cc_\phi=\{x\,;\,\det Dg(x)=0\}$ is the graph of a $C^2$ function $c_g:S^1\to\RR$ arbitrarily close to the null function. In particular, the critical set of $g$ is non-flat. 
\end{enumerate}
In Theorem~C in \cite{AV}, it was  proved that 
\begin{enumerate}
\item[(5)] $g|_{ \bigcap_{n\ge0}  g^n(S^1\times I)}$ is strongly transitive, for every $g\in\cn$.
\end{enumerate}
A map $f:=g|_{ \bigcap_{n\ge0}  g^n(S^1\times I)}$, where in $g\in \cn$, is called a {\bf\em Viana map}.
Let us denote the domain of $f$ by $J_f$, that is, $$J_f=\bigcap_{n\ge0}   g^n(S^1\times I).$$

It follows from Section~2.5 in \cite{Vi}, that there exists an invariant foliation $\cf^c$ by nearly vertical smooth curves.
The map $\pi:J_f\to S^1$ defined by $(\pi(x),0)=\cf^c(x)\cap( S^1\times\{0\})$, where $\cf^c(x)$ is the element of $\cf^c$ containing $x\in J_f$, is a continuous map.
Moreover, if we take $h:S^1\to S^1 $ given by  $h(x)=\pi(f(\pi^{-1}(x))=\pi(\cf^c(f(x))$, we have that $h$ is a local homeomorphism topologically conjugated to $\sigma$ and $$h\circ\pi=\pi\circ f.$$

As $h$ is conjugated with $\sigma$, then $h_{top}(h)=\log d$ and  $h_{\pi_*\mu}(f)\le\log d$ for every $\mu\in\cm^1(f)$.
Moreover, if the Lyapunov exponent of along the central direction $\cf^c$ is zero for $\mu$-almost every point, 
then combining the Ruelle inequality for fibered dynamical systems (see e.g. \cite{Arl}) with the  Ledrappier-Walters formula \cite{LW} 
for the entropy of fibered systems (see alternatively \cite{Bo71})
one guarantees that $h_{\mu}(f)\le\log d$.
This shows that
 \begin{equation}\label{Eqautionmnmiuyt}
  \{\mu\in\cm_{erg}^1(f)\,;\, h_{\mu}(f)>\log d\}\subset\ce(f).
\end{equation}

\begin{Lemma}\label{Lemmaihbihb2e328}
If $f$ is a Viana map with critical set $\cc$ and $\mu$ is an $f$-invariant, ergodic probability measure such that $\mu(\co_f^+(\cc))>0$ then $0\le h_{\mu}(f)\le\log d$.
\end{Lemma}
\begin{proof}
Note that $J_f$ is a compact subset of $S^1\times\RR$ with $S^1\times\{0\}\subset\interior(J_f)$.
As $\mu(\co_f^+(\cc))=\mu(\bigcup_{n\ge0}f^n(\cc))>0$, there is $\ell\ge0$ such that $\mu(\Lambda)>0$, where $\Lambda=f^{\ell}(\cc)$.
Let $F$ be the first return map to $\Lambda$ by $f$ and $R$ be the first return time to $\Lambda$ by $f$.
In addition, if $H$ is the first return map to $\pi(\Lambda)\subset S^1$ by $h$ and $r$ its induced time, we get that  $$H\circ\pi=\pi\circ F\text{ and }r\circ\pi=R$$

As $\Lambda$ is the image by $f^{\ell}$ of the admissive curve $\cc$  
(see \cite{Vi} for definitions and more details), 
we get that 
$$
  \#\cf^c(x)\cap\Lambda\le d^{\ell}\;\quad \forall\,x\in J_f
$$ and so,
$$\#(\pi|_{\Lambda}^{-1}(\theta))\le d^{\ell}\quad\;\forall\theta\in S^1.$$
 The latter suffices to ensure that $h_{\overline{\mu}}(F)=h_{\pi_*\overline{\mu}}(H)$. 
As a consequence, taking $\overline{\mu}$ as the $F$-lift of $\mu\in\cm^1(f)$ and $\overline{\pi_*\mu}$ as the $H$-lift of $\pi_*\mu\in\cm^1(h)$, we get that $\int R\,d\overline{\mu}=\int r\,d\overline{\pi_*\mu}$ and 
$$h_{\mu}(f)\int R\,d\overline{\mu}=h_{\overline{\mu}}(F)=h_{\pi_*\overline{\mu}}(H)=h_{\pi_*\mu}(h)\int r\,d\overline{\pi_*\mu}.$$
Therefore, $0\le h_{\mu}(f)=h_{\pi_*\mu}(h)\le h_{top}(h)=\log d$, which proves the lemma.
\end{proof}

Let $f:J_f\to J_f$ be a Viana map with critical region $\cc$.
It follows from \cite{Al}, there exists an absolutely continuous invariant probability measure $\mu_f$, equivalent with respect to $\leb|_{J_f}$, and such that $\mu_f$ is ergodic and expanding.
As $\supp\mu_f=J_f$, it follows from  Theorem~\ref{Theoremhgjvin09} at Appendix II that the set $\cf(f)$ of the free points of $f$ is an open and dense subset of $J_f$, and that $\leb(J_f\setminus\cf(f))=0$.

\begin{Lemma}\label{Lemmaugvfvy}
$0\le h_{\mu}(f)\le\log d$ for every ergodic probability measure $\mu\in\cm^1(f)\setminus\ce(f|_{\cf(f)})$.
\end{Lemma}
\begin{proof}
Given an ergodic probability measure $\mu\in\cm^1(f)\setminus\ce(f|_{\cf(f)})$,  either $\mu(\cf(f))=0$ or else  $\mu\notin\ce(f)$.
If $\mu\notin\ce(f)$ then  \eqref{Eqautionmnmiuyt} implies that $h_{\mu}(f)\le\log d$.
On the other hand, as $f$ is strongly transitive, it follows from item $(4)$ of Theorem~\ref{Theoremhgjvin09} that $\partial(\cf(f))\subset\co_f^+(\cc)$. Hence, if $\mu(\cf(f))=0$, it follows from  Lemma~\ref{Lemmaihbihb2e328} above, that $\mu(\cf(f))=0$ is equivalent to $\mu(\partial(\cf(f))=1$. The latter implies 
$\mu(\co_f^+(\cc))=1$, which ultimately can only occur if $0\le h_{\mu}(f)\le\log d$. This concludes the proof.
\end{proof}

\begin{proof}[\bf Proof of Corollary~\ref{CorollaryVIANA}] 
If $\varphi$ is a free expanding Hölder continuous potential, we can use  item $(8)$ of Theorem~\ref{Theoremhgjvin09} to conclude that $\varphi$ has a unique equilibrium state $\mu_{\varphi}$, also that $\mu_{\varphi}\in\ce(f)$ and $\supp\mu_{\varphi}=\supp\mu_f=J_f$.
Thus, we are reduced to prove that $\varphi$ is a free expanding potential.

First, assume condition $(ii)$ holds, that is 
 $\int\varphi \,d\mu<P(f,\varphi)-\log d$ for every $\mu\in\cm^1(f)$.
An argument similar to the one in the proof of Theorem~\ref{TheoremTFtoIM} ensures that
\begin{equation}\label{Equationkjbr4qq001}
  \gamma:=\sup\bigg\{\int\varphi \,d\mu\,;\,\mu\in\cm^1(f)\bigg\}<P(f,\varphi)-\log d.
\end{equation}

If $\mu\in\cm^1(f)\setminus\ce(f|_{\cf(f)})$ is ergodic then  $h_{\mu}(f)\le\log d$ (recall Lemma~\ref{Lemmaugvfvy}) and so, by \eqref{Equationkjbr4qq001},  
 $h_{\mu}(f)+\int\varphi \,d\mu\le\log d+\gamma<P(f,\varphi)$.
This implies that $P(f,\varphi)=P_{\ce(f|_{\cf(f)})}(\varphi)$.
Finally,  using the ergodic decomposition and   Jacobs Theorem \cite{J63, OV08}, we get that 
$$h_{\mu}(f)+\int\varphi \,d\mu<P_{\ce(f|_{\cf(f)})}(\varphi)\;\;\forall\mu\in\cm^1(f)\setminus\ce(f|_{\cf(f)}),$$ proving that $\varphi$ is a free expanding potential.
Noting that $\varphi\equiv0$ satisfies condition $(ii)$, there is a measure $\mu_0\in\ce(f)$, with $\mu_0(\cf(f))=1$, such that $h_{\mu_0}(f)=h_{top}(f)$.

Now, assume condition $(i)$, i.e., $\sup\varphi-\inf\varphi<h_{top}(f)-\log d$. 
Taking $\psi(x)=\varphi(x)-\inf\varphi$, we have that $\psi$ is a piecewise Hölder continuous potential and $0\le\psi(x)<h_{top}(f)-\log d$.
If $\mu\in\cm^1(f)\setminus\ce(f|_{\cf(f)})$ is ergodic, we get from Lemma~\ref{Lemmaugvfvy} that $h_{\mu}(f)\le\log d$ and so,
$$0\le h_{\mu}(f)+\int\psi d\mu<h_{top}(f)=h_{\mu_0}(f)\le h_{\mu_0}(f)+\int\psi d\mu_0\le P_{\ce(f|_{\cf(f)})}(\varphi).$$
Again, using the ergodic decomposition and Jacobs result once more, we conclude that 
$$h_{\mu}(f)+\int\psi d\mu<P_{\ce(f|_{\cf(f)})}(\varphi)\;\;\forall\mu\in\cm^1(f)\setminus\ce(f|_{\cf(f)}),$$ proving that $\psi$ is a free expanding Hölder continuous potential.
As observed above,  this implies that $\psi$ has a single equilibrium state $\mu_{\psi}$.
Moreover, $\mu_{\psi}\in\ce(f)$ and $\supp\mu_{\psi}=\supp\mu_f=J_f$.
Noting that $\mu\in\cm^1(f)$ is an equilibrium state for $\psi$ if and if $\mu$ an equilibrium state for $\varphi$, we conclude the proof of the theorem.
\end{proof}

\begin{Remark}
The first item of Corollary~\ref{CorollaryVIANA} proved above implies the existence and uniqueness of equilibrium states for Hölder continuous potential with small oscillation. The existence of  equilibrium states for potentials with small oscillation  for Viana maps appeared first in a preprint by Arbieto, Matheus 
and Senti \cite{AMS06}  
The uniqueness of the measure of maximal entropy Viana maps was stated in a preprint by Santana \cite{Sa21}.
\end{Remark}

\section{Appendix I:  Strong transitivity, slow recurrence and upper semicontinuity of entropy}

\subsection{Strongly transitive maps}
As in Section~\ref{sec:def}, let $\XX$ be a Baire metric space and $\cc\subset\XX$ a closed set with empty interior. 
Assume that for every $x\in\XX$ there is $\gamma_x>0$ such that 
$$
  \overline{B_{\gamma_x}(x)}\text{ is compact and } B_{\varepsilon}(x)\text{ is connected and for every }0<\varepsilon\le\gamma_x$$
and let  $f:\XX\setminus\cc\to\XX$ is 
local homeomorphism.

Given a map $h:W\to\XX$, where $W\subset\XX$ set $h^*:2^{\XX}\to 2^{\XX}$ by $h^*(U)=h(U\cap W)=\{h(x)\,;\,x\in U\cap W)$, where $2^{\XX}$ is the set of all subsets of $\XX$ including the empty set.
In particular, $f^*(U)=f(U\setminus\cc)$ for every $U\subset\XX$, as defined in Section~\ref{sec:def}.

\begin{Lemma}\label{Lemmajhg86ryckjh}
 If $f$ is transitive  then, given $\ell\ge1$ there are $1\le s\le\ell$,  $\ell/s\in\NN$, and an open set $U\subset\XX$ such that
\begin{enumerate}
\item $(f^*)^{s}(U)\subset U$;
\item $(f^*)^j(U)\cap (f^*)^k(U)=\emptyset$ for every $0\le j<k<s$;
\item $f^{\ell}$ is transitive on $(f^*)^j(U)$ for every $0\le j<s$;
\item $\overline{\bigcup_{j=0}^{s-1}(f^*)^j(U)}=\XX$.
\end{enumerate}	
\end{Lemma}

\begin{proof}
Assume that $f$ is transitive and fix $\ell\ge1$.
Since $f$ is transitive, $f$ is a local homeomorphism, $\cc$ is a meager set and $\XX$ a Baire space, one can see that $\XX'=\bigcap_{n\ge0}f^{-n}(\XX)$ is a residual subset of $\XX$ and a Baire space (with respect to the induced metric). Moreover, as one can check that $f|_{\XX'}$ is also transitive, there $p\in\XX'$ such that $\co_f^+(p)=\co_{f|_{\XX'}}^+(p)$ is dense on $\XX'$. Thus, $\overline{\co_f^+(p)}=\XX$.
Let 
$$
U_j=\interior\bigg(\overline{\co_{f^{\ell}}^+(f^j(p))}\bigg), \quad \text{for  } 0\le j < \ell. 
$$
Observe that $\XX=\bigcup_{j=0}^{\ell-1}\overline{U_j}$. Hence, there exists $0\le m<\ell$ such that $U_{m}\ne\emptyset$.
Moreover, as $f^{\ell}(\co_{f^\ell}^+(p))\triangle\co_{f^{\ell}}^+(f^{\ell}(p))\subset\{p\}$, $f^{\ell-m}(\co_{f^\ell}^+(f^m(p))=\co_{f^{\ell}}^+(f^{\ell}(p))$ and $f$ is a local homeomorphism, we get that \begin{equation}\label{Equationhbibbub}
  (f^*)^{\ell}(U)\subset U:=U_0\ne\emptyset.
\end{equation}
 
The same reasoning proves that $(f^*)^j(U)\ne\emptyset$ for every $j\ge1$.

Now, observe that it follows from $(f^*)^{\ell}(\overline{U})\subset\overline{U}$ and the fact that $f$ is an open map that 
$(f^*)^{\ell}(U)\subset U$. In consequence,
$$(f^*)^{\ell}((f^*)^j(U))\subset (f^*)^j(U) \quad \text{for every } j\ge1.
$$
As $(f^*)^j(U)=\interior\big(\overline{\co_{f^{\ell}}^+(f^j(p))}\big)$ is a $f^\ell$-forward invariant and non-empty open set,
there exists $n\ge0$ such that $y_j:=f^{n\ell}(f^j(p))\in (f^*)^j(U)$ and $\overline{\co_{f^{\ell}}^+(y_j)}=\overline{(f^*)^j(U)}$,  proving that $f^{\ell}$ is transitive on $(f^*)^j(U)$.

Consider $s=\min\{j\ge1\,:\,(f^*)^j(U)\cap U\ne\emptyset\}$. Note that $s\le\ell$ by (\ref{Equationhbibbub}). We first claim that $(f^*)^{s}(U)\subset U$.
Indeed, as $W:=(f^*)^s(U)\cap U\ne\emptyset$ is an open set, there exists a point $y\in W\cap\co_{f^{\ell}}^+(f^s(p))$.
By construction $(f^*)^{\ell}(W)\subset W$.
On the one hand $\co_{f^\ell}^+(y)\subset W$ and so, $\interior\big(\,{\overline{\co_{f^{\ell}}^+(y)}}\,\big)=W\subset U$.
On the other hand, as $y\in \co_{f^{\ell}}^+(f^s(p))$, we have that $\interior\big(\,\overline{\co_{f^{\ell}}^+(y)}\,\big)=\interior\big(\,\overline{\co_{f^{\ell}}^+(f^s(p))}\,\big)=(f^*)^s(U)$, proving that $(f^*)^s(U)\subset U$, as claimed.
Now, as $1\le s\le \ell$, we can write $\ell=a s+b$ with $0\le b<s$.
Thus, $(f^*)^{\ell}(U)=(f^*)^{a s+b}(U)=(f^*)^{b}((f^*)^{a s}(U))\subset(f^*)^{b}(U)$.
As $(f^*)^{\ell}(U)\subset U$ and $(f^*)^b(U)\cap U=\emptyset$ for every $0<b<s$, we must have $b=0$ and $\ell/s=a\in\NN$. 
By construction, if $0\le j<k<s$ then $0\le j+(s-k)<s$ and $$(f^*)^{s-k}((f^*)^j(U)\cap (f^*)^k(U))\subset (f^*)^{j+s-k}(U)\cap (f^*)^s(U)\subset(f^*)^{j+s-k}(U)\cap U=\emptyset,$$
and so $(f^*)^j(U)\cap (f^*)^k(U)=\emptyset$ for every $0\le j<k<s$.

Finally, as $\bigcup_{j=0}^{s-1}(f^*)^j(U)\subset V$ is open and $f^*\big(\bigcup_{j=1}^{s-1}(f^*)^j(U))=\bigcup_{j=1}^{s}(f^*)^j(U)\subset \bigcup_{j=0}^{s-1}(f^*)^j(U)$, it follows from the transitivity of $f$ that $\overline{\bigcup_{j=0}^{s-1}(f^*)^j(U)}=\XX$.
\end{proof}

\begin{Lemma}\label{Lemmaammama}
Suppose that $f$ is strongly transitive. If $f^{\ell}$ is transitive for some $\ell\ge1$ then $f^{\ell}$ is strongly transitive.
\end{Lemma}
\begin{proof}Suppose that $f$ is a strongly transitive and  
$f^{\ell}$ is topologically transitive, for some $\ell\ge1$.
For any arbitrary $x\in\XX$, one can write
$$
\co_f^-(x)=\bigcup_{j=0}^{\ell-1}\bigcup_{y\in f^{-j}(x)}\co_{f^{\ell}}^-(y), \quad\text{and so}\quad 
\alpha_f(x)=\bigcup_{j=0}^{\ell-1}\bigcup_{y\in f^{-j}(x)}\alpha_{f^{\ell}}(y).
$$

As $\alpha_f(x)=\XX$ when $x\notin\cc$, there exist $0\le s<\ell$ and $y\in f^{-s}(x)$ such that $\interior(\alpha_{f^{\ell}}(y))\ne\emptyset$.
As  $f^*$ is continuous and open, we have that 
$$
\emptyset\ne (f^*)^s(\interior(\alpha_{f^{\ell}}(y)))\subset \interior(\alpha_{f^{\ell}}(f^s(y)))=\interior(\alpha_{f^{\ell}}(x)).
$$
Furthermore, as $(f^*)^{\ell}(\alpha_{f^{\ell}}(x))\subset \alpha_{f^{\ell}}(x)$ and $f$ is open, we have that 
$(f^*)^{\ell}(W)\subset W$, where $W=\interior(\alpha_{f^{\ell}}(x))\ne\emptyset$.
Then, the transitivity of $f^{\ell}$ on $U$ 
ensures that $\alpha_{f^{\ell}}(x)=\overline{W}=\XX$ for every $x\notin\cc$, proving that $f^\ell$ is strongly transitive for every $\ell\in\NN$.
\end{proof}

\begin{Lemma}\label{LemmaihiiFIX}
If $f$ is strongly transitive and $\fix(f)\ne\emptyset$ then $f^{\ell}$ is strongly transitive for every $\ell\ge1$.
\end{Lemma}
\begin{proof}
Let $\ell\ge1$ and consider $U$ as the open set given by Lemma~\ref{Lemmajhg86ryckjh}.
Choose $p\in\fix(f)$. As $f$ is strongly transitive, there is a smaller $n\ge0$ such that $p\in (f^*)^n(U)$.
As $(f^*)^s(U)\subset U$, we get that $0\le n<s\le\ell$.
As $f^s(p)=p$, it follows from item (2) of Lemma~\ref{Lemmajhg86ryckjh} that $s=1$ and so, by item (4), $\overline{U}=\XX$.
As, by item (3), $f^{\ell}$ is transitive on $U$, we can conclude that $f^{\ell}$ is transitive on $\XX$.
Thus, it follows from Lemma~\ref{Lemmaammama} that $f^{\ell}$ is strongly transitive.
\end{proof}

\begin{Proposition}\label{PropositionTotTransPer}
If $f$ is strongly transitive and $\per(f)\ne\emptyset$, then there is an open set $V\subset\XX$ such that
\begin{enumerate}
\item $V\cup f^*(V)\cdots \cup (f^*)^{\ell-1}(V)\supset\XX\setminus\cc$;
\item $(f^*)^{\ell}(V)\subset V$;
\item $f^{\ell j}|_V$ is strongly transitive for every $j\ge1$,
\end{enumerate}  
where $\ell=\min\{n\in\NN\,;\,\fix(f^n)\ne\emptyset\}$.
\end{Proposition}
\begin{proof}Let $\cf_0$ be the set of all local homeomorphisms $h:U\setminus\cc\to U$ defined on nonempty open sets $U\subset\XX$.
Given $n\in\NN$, let $\cf_n$ be the set of all strongly transitive $h\in\cf_0$ such that $\bigcup_{j=1}^n\fix(h^j)\ne\emptyset$.

It follows from Lemma~\ref{LemmaihiiFIX} that this proposition is true for every $h\in\cf_1$. By induction, assume that this proposition is true for every $h\in\cf_{\ell-1}$, where $\ell\ge2$.

Let $h\in\cf_{\ell}$ with $h:W\setminus\cc\to W$, where $W\subset\XX$ is a nonempty open set. 
Let $U$ the open set given by Lemma~\ref{Lemmajhg86ryckjh} applied to $h$ and $1\le s\le\ell$, with $\ell/s\in\NN$, such that $(h^*)^{s}(U)\subset U$.

We may assume that $h\in\cf_{\ell}\setminus\cf_{\ell-1}$. In this case, there is a periodic point $p$ for $h$ with period $\ell$. As $h$ is strongly transitive, there is a smaller $n\ge0$ such that $p\in (h^*)^n(U)$.
As $(h^*)^s(U)\subset U$, we get that $0\le n<s\le\ell$.
Taking $q=h^{s-n}(p)$ we get that $q$ is a periodic point for $g:=h^s|_U$ with period $\ell/s$.
Note that $g$ is strongly transitive.
Indeed, given a nonempty  open set $A\subset U$, it follows from $h$ being strongly transitive,  $(h^*)^j(U)\cap (h^*)^k(U)=\emptyset$ $\forall\,0\le j<k<s$ and $(h^*)^s(U)\subset U$ that $$U=U\cap\bigcup_{i\ge0}(h^*)^i(A)=\bigcup_{i\ge0}(h^*)^{s\,i}(A)=\bigcup_{i\ge0}(g^*)^i(A),$$
showing that $g$ is strongly transitive.

If $s\ge2$ then $g\in\cf_{\ell-1}$, it follows from the induction hypothesis that the exists a nonempty open set $V\subset W$ and $1\le s_1\le\ell/s$ such that $g^{\frac{\ell}{s}j}|_V=h^{\ell j}|_V$ is strongly transitive for every $j\ge1$, with $(h^*)^{\ell}(V)=(g^*)^{\ell/s}(V)\subset V$ and $$W\setminus\cc\subset V\cup \cdots\cup(g^*)^{\frac{\ell}{s}-1}(V)=V\cup(h^*)^{s}(V)\cdots\cup(f^*)^{\ell-s}\subset V\cup\cdots\cup(f^*)^{\ell-1}(V),$$ proving the induced step when $s\ge2$. 

Hence, we may assume that $s=1$.
In this case, it follows from item (4) of Lemma~\ref{Lemmajhg86ryckjh} that $U$ is an open and dense subset of $W$ and this implies that $h^{\ell}$ is transitive.
Thus, by Lemma~\ref{Lemmaammama}, $h^{\ell}$ is strongly transitive. As $q\in\fix(h^{\ell})$ it follows form Lemma~\ref{LemmaihiiFIX} that $h^{\ell j}$ is strongly transitive for every $j\ge1$, concluding the proof of the induced step.
\end{proof}

\begin{Proposition}\label{Propositioniougoygh} Let $\alpha$ be a zooming contraction and $\delta>0$. If  $\mu\in\cm^1(f)$ is an ergodic $(\alpha,\delta,1)$-weak zooming probability measure with $\interior(\supp\mu)\ne\emptyset$ then $$U:=\{x\in\supp\mu\,;\,\alpha_f(x)\supset\supp\mu\}$$ contains an open and dense subset of $\supp\mu$, with $\mu(U)=1$ and $f^*(U)=U$. In particular, $f|_{U}$ is strongly transitive. 
\end{Proposition}

\begin{proof}
Let $V=\interior(\supp\mu)$. As $f^*(\supp\mu)\subset\supp\mu$ and $f^*$ is an open map (the image of an open set is an open set), we get that $f^*(V)\subset V$.
Thus, it follows from the invariance and ergodicity of $\mu$ that $\mu(V)=1$.

Let $\omega_{_{\cz}}(x)$ be the set of all $y\in\XX$ such that $y=\lim_{j\to\infty} f^{n_j}(x)$, where each $n_j\ge 1$ is such that $x\in\cz_{n_j}(\alpha,\delta,1)$
and $n_j\to\infty$.
In other words, the set $\omega_{_{\cz}}(x)$ is formed by the accumulation points of the iterates of $x$ at $(\alpha,\delta,1)$-zooming times.
It follows from Lemma~3.9 of \cite{Pi11}  that there exists a compact set $\ca_{_{\cz}}$ such that $\omega_{_{\cz}}(x)=\ca_{_{\cz}}$ for $\mu$-almost every $x$.
Since $V$ is an open set, for $\mu$-almost every $x$ there exists $n\ge 1$ so that 
$V_{n+\ell}({\alpha},\delta,1)(x) \subset V$. 
This fact, together with the invariance of $V$ and the fact that every point of $\ca_{_{\cz}}$ is accumulated by the centers of zooming balls of radius $\delta$, guarantees that 
$$
B_{\delta}(\ca_{_{\cz}}):=\Big\{x\in\XX\,:\,\dist(x,\ca_{_{\cz}})<\delta\Big\}\subset V.
$$

The latter put us in a position to use the construction of nested sets and induced maps from \cite[Section~5]{Pi11}. In brief terms, nested sets are a generalization of the concept of nice sets and enjoy the key property that any two pre-images at zooming times are either disjoint or nested, which makes possible to build natural induced maps (see Definition~5.9 and 
Theorem~2 in \cite{Pi11}). 
Indeed, consider a small $r\in(0,\delta/2)$, a point $p\in\ca_{_{\cz}}$, $B=B_r^{\star}(p)$ the zooming nested ball of radius $r$ centered at $p$ 
and
$$R(x)=\min\Big\{n\in\NN\,:\,x\in (f^n|_{V_n(\alpha,\delta,1)(y)})^{-1}(B)\text{ for some }y\in\cz_n(\alpha,\delta)\Big\}.$$ 
We recall that $B$ is a connected open set containing $B_{r/2}(p)$ and $B\cap \ca_{_{\cz}} \ne\emptyset$.
Then Corollary~6.6 and Lemma~6.7 in \cite{Pi11} ensure that there exists a $(\alpha,\delta,1)$-zooming return map $(F,B,\cp)$ such that 
$$F:A:=\bigcup_{P\in\cp}P\to B, \quad F(x)=f^{R(x)}\quad \text{and}\quad A=B \, (\text{mod}\;\mu).
$$ 
We claim  that $\supp\mu \subset \alpha_f(x)$ for every $x\in B$. 
Observe first that $B\subset B_r(p) \subset B_\delta(A_\cz)$
are open subsets in $\supp \mu$.
Hence $\alpha_f(x)\supset\alpha_F(x)=\overline{B}$ for every $x\in B$.
It follows from the invariance and ergodicity of $\mu$ that $\mu(\bigcup_{n\ge0}f^{-n}(B))=1$.
Thus, $\bigcup_{n\ge0}f^{-n}(B)$ is an open and dense subset of $\supp\mu$.
As $\alpha_f(x)\supset \bigcup_{n\ge0}f^{-n}(B)$ for every $x\in B$, we get that $\alpha_f(x)\supset\supp\mu$ for every $x\in B$, as claimed.
As a consequence, $U\supset B$.

\begin{Claim}\label{Claimnblhviy}
Given $p\in U$ there is an open set $V\subset U$, with $V\cap\cc=\emptyset$ such that $f(V)$ is an open neighborhood of $p$.
\end{Claim}
\begin{proof}[Proof of the claim]
Given $p\in U$, it follows from $\alpha_f(p)=\supp\mu$ that $p=f^n(q)$ for some $q\in B$ and $n\ge1$. Now, the proof follows the same steps of the proof of Claim~\ref{Claimnbl000hviy}. 
\end{proof}

It follows from the Claim~\ref{Claimnblhviy} above that $U$ is an open set and $f^*(U)=U$. Finally, as $\mu$ is ergodic, $f$-invariant and $\mu(U)>0$, we get that $\mu(U)=1$. From $\mu(U)=1$ follows that $U$ is dense on $\supp\mu$, concluding the proof.
\end{proof}

 \subsection{Slow recurrence to the critical set} Let $M$ be a Riemannian manifold with $\diameter(M)<+\infty$ and $f:M\setminus\cc\to M$ a $C^{1+}$ local diffeomorphism.
\begin{Lemma}\label{LemmaDoLIftAndS}
 If $\cc$ is non-degenerated and either $\lim_{x\to c}|\log|\det Df(x)||=0$ for every $c\in\cc$ or $\lim_{x\to c}|\log|\det Df(x)||=+\infty$ for every $c\in\cc$ then $\int_{x\in M}\log\dist_1(x,\cc)d\mu>-\infty$.
\end{Lemma}
\begin{proof} 
The proof follows the proof of Lemma~B.1 of \cite{Pi20}.
Indeed, 
Consider the function $\varphi:M\to[0,+\infty)$ defined as
$$\varphi(x)=
\begin{cases}
	0 & \text{ if }x\in\cc\\
	\det Df(x) & \text{ if }x\notin\cc \text{ and $\cc$ is purely critical}\\
	\frac{1}{\det Df(x)} & \text{ if }x\notin\cc \text{ and $\cc$ is purely singular}
\end{cases}
$$
As $f$ is $C^{1+}$, we get that $\varphi$ is a Hölder function.
We may assume
that $\cc\ne\emptyset$.
As $\varphi$ is Holder, $\exists\,k_0,k_1>0$ such that  $|\varphi(x) - \varphi(y)|$ $\le$ $k_0\dist(x,y)^{k_1}$ $\forall\,x,y\in M$.
As $\cc$ is closed, given $x\in M$ there is $y_x\in\cc$ such that $\dist(x,y_x)=\dist(x,\cc)$.
Thus, we get $|\varphi(x)|$ $=$ $|\varphi(x) - \varphi(y_x)|$ $\le$ $k_0\dist(x,y_x)^{k_1}$ $=$ $k_0\dist(x,\cc)^{k_1}$.
That is,
\begin{equation}\label{EquationJHGJKJH7}
  \log|\varphi(x)|\le\log k_0+k_1\log\dist(x,\cc)\;\,\forall\,x\in M.
\end{equation}

Let $m=dimension(M)$ and note that $\|A^{-1}\|^{-m}\le|\det A|\le\|A\|^{m}$ for every $A\in GL(m,\RR)$.
That is, $$m\log\big(\|A^{-1}\|^{-1}\big)\le\log|\det A|\;\;\text{ and }\;\;\log\|A\|\ge-\frac{1}{m}\log\bigg|\frac{1}{\det A}\bigg|.$$
Thus, 
if $\int\log|\varphi|\,d\mu=-\infty$, it follows from Birkhoff that either $$\limsup_{n\to\infty}\frac{1}{n}\log\|(Df^{n}(x))^{-1}\|^{-1}\le\frac{1}{m}\limsup_{n\to\infty}\frac{1}{n}\log|\det Df^n(x)|=$$
$$=\frac{1}{m}\,\lim\frac{1}{n}\sum_{j=0}^{n-1}\log|\det Df(f^j(x))|=\frac{1}{m}\,\lim\frac{1}{n}\sum_{j=0}^{n-1}\log|\varphi\circ f^j(x)|=-\infty$$
for $\mu$-almost every $x$ (when $\cc$ is purely critical) or, when $\cc$ is purely singular, 
$$\limsup_{n\to\infty}\frac{1}{n}\log\|(Df^{n}(x))\|\ge-\frac{1}{m}\liminf_{n\to\infty}\frac{1}{n}\log\bigg|\frac{1}{\det Df^n(x)}\bigg|=$$
$$=-\frac{1}{m}\,\lim\frac{1}{n}\sum_{j=0}^{n-1}\log\bigg|\frac{1}{\det Df(f^j(x))}\bigg|=-\frac{1}{m}\,\lim\frac{1}{n}\sum_{j=0}^{n-1}\log|\varphi\circ f^j(x)|=+\infty$$
for $\mu$-almost every $x$. In any case, we have a contradiction to our hypothesis.
So, $\int\log|\varphi|d\mu>-\infty$ and, by (\ref{EquationJHGJKJH7}), we get that  $$-\infty<\int\log|\varphi|\,d\mu-\log k_0\le k_1\int_{x\in M}\log\dist(x,\cc)\,d\mu.$$
Hence $\int_{x\in M}\log\dist(x,\cc)\,d\mu>-\infty$ and so, $$\int_{x\in M}\log\dist_1(x,\cc)d\mu=\int_{x\in M}\log\dist(x,\cc)d\mu-\int_{x\in M\setminus B_1(\cc)}\log\dist(x,\cc)d\mu\ge$$
$$\ge\int_{x\in M}\log\dist(x,\cc)d\mu-\log\diameter(M)>-\infty,$$ concluding  the proof. \end{proof}

\subsection{Entropy upper semi-continuity for zooming measures}

Let $\XX$ be a metric space and $f:\XX\to\XX$ a continuous map.
Suppose that there exist a closed set $\cc\subset\XX$ with empty interior and a finite cover $\{X_1,\cdots,X_t\}$ of $\XX\setminus\cc$ by open sets such that $f|_{X_j}$ is injective for every $1\le j\le t$ (i.e., $f|_{\XX\setminus\cc}$ is piecewise injective).

\begin{Lemma}\label{Lemmakjsvj300}
If $\mu$ is probability measures $\mu$ of $\XX$ with $\mu(\cc)=0$ then there exists open sets $\{Y_1,\cdots,Y_t\}$ such that $\XX\setminus\cc=Y_1\cup\cdots\cup Y_t$, $Y_j\subset X_j$ and $\mu(\partial Y_j)=0$ for every $1\le j\le t$. 
\end{Lemma}
\begin{proof}
Given $x\in\XX$ and an open set $V$ let $r(V,x)=\sup\{\delta\ge0\,;\,B_{\delta}(p)\subset V\}$.
For each $\varepsilon\in[0,1/2]$ and open sets $U,V\subset\XX\setminus\cc$ let
$$U[\varepsilon,V]=U\setminus\bigg(\cc\cup \overline{\bigcup_{x\in V\cap \partial U} B_{\varepsilon r(V,x)}(x)}\bigg)\subset\XX\setminus\cc$$
and note that $$U\cup V=U[\varepsilon,V]\cup V.$$
Thus, setting $X_{1,\varepsilon}=X_1[\varepsilon,\bigcup_{j=2}^tX_j]$, we have that $X_{1,\varepsilon}\subset X_1$ and $X_{1,\varepsilon}\cup X_2\cup\cdots\cup X_t=\XX\setminus\cc$.
Moreover, as $\partial X_{1,\varepsilon_1}\cap\partial X_{1,\varepsilon_2}\subset\cc$ when $\varepsilon_1\ne\varepsilon_2$, there exists a countable set $R_1\subset[0,1/2]$ such that $\mu(\partial X_{1,\varepsilon})=0$ for every $\varepsilon\in[0,1/2]\setminus R_1$.
Choose $\varepsilon_1\in[0,1/2]\setminus  R_1$. 
Inductively, suppose that $1\le n\le t$, $X_{1,\varepsilon_1}\subset X_1$,$...$, $X_{n,\varepsilon_n}\subset X_n$, $(X_{1,\varepsilon_1}\cup\cdots\cup X_{n,\varepsilon_n})\cup (X_{n+1}\cup\cdots\cup X_t)=\XX\setminus\cc$ and $\mu(\partial X_{1,\varepsilon_1})=\cdots=\mu(\partial X_{n,\varepsilon_n})=0$.
Taking $V=(X_{1,\varepsilon_1}\cup\cdots\cup X_{n,\varepsilon_n})\cup (X_{n+2}\cup\cdots\cup X_t)$ and $$X_{n+1,\varepsilon}=X_{n+1}[\varepsilon,V],$$
we get that $X_{n+1,\varepsilon}\subset X_{n+1}$ and  $X_{n+1,\varepsilon}\cup V =\XX\setminus\cc$.
As $\partial X_{n+1,\varepsilon_1}\cap\partial X_{n+1,\varepsilon_2}\subset\cc$ when $\varepsilon_1\ne\varepsilon_2$, there exists a countable set $R_{n+1}\subset[0,1/2]$ such that  $\mu(\partial X_{n+1,\varepsilon})=0$ for every $\varepsilon\in[0,1/2]\setminus R_{n+1}$.
Taking $\varepsilon_{n+1}\in[0,1/2]\setminus R_{n+1}$, we complete the induction step as well as the proof.
\end{proof}

\begin{Lemma}\label{Lemmadfghjye32zzx}
Let $\alpha$ be a zooming contraction, $\ell\in\NN$ and $\delta>0$.
If $\mu_n\in\cm^1(f)$ is a sequence of $(\alpha,\delta,\ell)$-zooming measures converging to some $\mu_0\in\cm^1(f)$ with $\mu_0(\cc)=0$ then $
  h_{\mu_0}(f)\ge\limsup_n h_{\mu_n}(f)$.
\end{Lemma}
\begin{proof}
Taking a subsequence if necessary, we can assume that $\lim_n h_{\mu_n}(f)=a_0$, for some $a_0\ge0$.

It follows from Lemma~\ref{Lemmakjsvj300} that there exists open sets $\mathcal{Y}=\{Y_1,\cdots,Y_t\}$ such that $\XX\setminus\cc=Y_1\cup\cdots\cup Y_t$, $Y_j\subset X_j$ and $\mu_0(\partial Y_j)=0$ for every $1\le j\le t$.
As, for $1\le j\le t$, $f|_{X_j}$ is injective, also $f|_{Y_j}$ is injective. 

Let $\cq$ be the partition of $\XX\setminus\cc$ generated by $\mathcal{Y}$, that is, $\cq=\{\mathcal{Y}(x)\,;\,x\in\XX\setminus\cc\}$, where $\mathcal{Y}(x)=\bigcap_{Y\in\cu(x)}Y$ and $\cu(x)=\{Y\in\mathcal{Y}\,;\,x\in Y\}$.
Note that $\cq$ is a finite partition, $\mu_0(\partial Y)=0$ and $f|_Y$ is injective for every $Y\in\cq$.

Given a partition $\cp$ of $\XX$, $x\in\XX$ and $i\in\NN$, recall that  $\cp(x)$ is the element of $\cp$ containing $x$ and $\cp_i(x)=\{y\in\XX\,;\,\cp(f^j(x))=\cp(f^j(y))\text{ for every }0\le j<i\}$.

\begin{Claim}\label{Claimkhfjudf4}
If $p$ has infinitely many  $(\alpha,\delta,\ell)$-zooming times
and $\cp$ is a partition of $\XX\setminus\cc$ with $\diameter(\cp)<\delta$ and $\cq\preceq\cp$ then $$  \lim_{i\to\infty}\diameter(\cp_i(p))=0.$$
\end{Claim}
\begin{proof}[Proof of the claim]
Let $\NN(p)$ be the set of all $(\alpha,\delta,\ell)$-zooming times of $p$.
Thus, if $i\in\NN(p)$ then there is a zooming pre-ball $V_i(p)$ such that $f^{i\,\ell}|_{V_i(p)}$ is a homeomorphism of $V_i(p)$ with $B_{\delta}(f^{i\,\ell}(p))$ and $(f^{i\,\ell}_{V_i(p)})^{-1}$ is a $\alpha_{i}$ contraction.

As $f^{i\ell}|_{\cp_{i\ell}(p)}$ is injective, $f^{i\ell}(\cp_{i\ell}(p))\subset \cp(f^{i\ell}(p))\subset B_{\delta}(f^{i\ell}(p))=f^{i\ell}(V_i(p))$, we get that $\cp_{i\ell}(p)\subset V_i(p)=(f^{i\,\ell}_{V_i(p)})^{-1}(B_{\delta}(f^{i\,\ell}(p)))$ and $\diameter(V_i(p))\le \alpha_i(\diameter(\cp(f^{i\ell}(p))))\le\alpha_i(\delta)$.
Thus, $$\lim_i \diameter(\cp_{i\ell}(p))=\lim_i\alpha_i(\delta)=0$$ and so, as $\{\diameter(\cp_i(p))\}_{i\in\NN}$ is a decreasing sequence, $\lim_i \diameter(\cp_i(p))=0$.

\end{proof}

It follows from Claim~\ref{Claimkhfjudf4} above that if $\cp$ is a  partition of $\XX\setminus\cc$ with $\diameter(\cp)<\delta$ and $\cq\preceq\cp$ then $\cp$  generates mod\,$\mu$ the $\sigma$-algebra of the Borel sets  for every $(\alpha,\delta,\ell)$-zooming probability measure $\mu$.
As a consequence, 
\begin{equation}\label{EQCjvbinoihv2342}
h_{\mu}(f)=h_{\mu}(f,\cp),
\end{equation}
for every $(\alpha,\delta,\ell)$-zooming probability measure $\mu$.

Let $\PP$ be the set of all finite measurable partition of $\XX\setminus\cc$ 
such  $\cq\preceq\cp$, $\diameter(\cp)<\delta$ and  $\mu_0(\partial P)=0<\mu_0(P)$ for every $P\in\cp$.
As $\cq$ is a finite partition of $\XX\setminus\cc$ with $\mu_0(\partial Q)=0$ $\forall Q\in\cq$, $\cp\vee \cq\in\PP$ for every finite partition $\cp$ of $\XX\setminus\cc$ with $\diameter(\cp)<\delta$.
This implies that $h_{\mu_0}(f)=\sup\{h_{\mu_0}(f,\cp)\,;\,\cp\in\PP\}$ and, by \eqref{EQCjvbinoihv2342}, $h_{\mu_n}(f)=h_{\mu_n}(f,\cp)$ for every $n\in\NN$ and $\cp\in\PP$.

Given $\varepsilon>0$, choose $\cp\in\PP$ such that $h_{\mu_0}(f,\cp)>h_{\mu_0}(f)-\varepsilon$.
As $\mu_0(\partial P)=0$ for every $P\in\cp$, we get that $\cm^1(f)\ni\nu\mapsto h_{\nu}(f,\cp)\in[0,+\infty)$ is upper semi-continuous at $\mu_0$.
Thus, $h_{\mu_0}(f)>h_{\mu_0}(f,\cp)-\varepsilon\ge\lim_n h_{\mu_n}(f,\cp)\ge a_0-\varepsilon$,
proving that $h_{\mu_0}(f)>a_0-\varepsilon$ for every $\varepsilon>0$, i.e., $h_{\mu_0}(f)\ge a_0$, which concludes  the proof of the lemma.
\end{proof}

\subsection{Semicontinuity of the entropy of Lipschitz zooming  measures for $C^1$ maps}

Let $M$ be a compact Riemannian manifold (possibly with boundary), let $\cc\subset M$ be a closed subset with empty interior.
Let $f:M\to M$ be a $C^{1}$ map with critical set $\cc=\{x\in M\,;\,\det Df(x)=0\}$. 
Suppose that $\interior(\cc)=\emptyset$ and that there exists a finite cover of $M\setminus\cc$ by open sets $\{M_1,\cdots,M_n\}$ such that $f|_{M_j}$ is injective for every $1\le j\le n$.

\begin{Lemma}\label{Lemmacritico}
If $\mu\in\overline{\ce(f)}$ then $\mu(\cc)=0$.
\end{Lemma}
\begin{proof}
If $\cc$ is pure critical set $\varphi:M\to[-\infty,+\infty)$ by
$$\varphi(x)=
\begin{cases}
\log|\det Df(x)| & \text{ if }x\not\in\cc\\
-\infty & \text{ if }x\in\cc
\end{cases}.
$$
As $\int\varphi d\mu$ is the sum of the Lyapunov exponents of $\mu$, we have that  $\int\varphi d\mu\ge0$ for every $\mu\in\ce(f)$.
Thus, by continuity $\int\varphi d\mu\ge0$ for every $\mu\in\overline{\ce(f)}$.
On the other hand, as $\varphi$ is bounded from above, $\int\varphi d\mu=-\infty$ for  every $\mu\in\cm^1(M)$ with $\mu(\cc)>0$, proving that $\mu(\cc)=0$ whenever $\mu\in\overline{\ce(f)}$.
\end{proof}

\begin{Proposition}
Let $\alpha$ be a Lipschitz zooming contraction, $\ell\in\NN$ and $\delta>0$.
If $\mu_n\in\cm^1(f)$ is a sequence of $(\alpha,\delta,\ell)$-zooming measures converging to some $\mu_0\in\cm^1(f)$ then $
  h_{\mu_0}(f)\ge\limsup_n h_{\mu_n}(f)$.\end{Proposition}
\begin{proof}
It follows from Proposition~\ref{Propkbihvg45678} that $\mu_n\in\ce(f)$ and so, $\mu_0\in\overline{\ce(f)}$.
Thus, it follows from Lemma~\ref{Lemmacritico} that $\mu_0(\cc)=0$. Therefore, we can apply Lemma~\ref{Lemmadfghjye32zzx} to conclude the proof.
\end{proof}

\subsection{Pressure for the shift $\Sigma_{\infty}^+$}

\begin{Proposition}
\label{PropositionPFSSIGMA}
Let $\LL\subset\NN$ and $\{\beta_n\}_{n\in\LL}$ be a sequence of positive real numbers such that $\sum_{n\in\LL}\beta_n<+\infty$.
Set $\LL(k)=\LL \cap\{1,\cdots,k\}$.
If $ a_n\in[0,1]$ is such that $\sum_{n\in\LL} a_n=1$ then 
$$\limsup_{k\to\infty}\sum_{n\in \LL(k), a_n\ne0} a_n\log (\beta_n/ a_n)\le\log\bigg(\sum_{n\in\LL}\beta_n\bigg).$$
Furthermore, $$\sum_{n\in\LL} a_n\log (\beta_n/ a_n)=\log\bigg(\sum_{n\in\LL}\beta_n\bigg)\iff a_n=\frac{\beta_n}{\sum_{m\in\LL}\beta_m}>0\;\forall n\in\LL.$$
\end{Proposition}
\begin{proof}
This proposition is a consequence of Lemma~\ref{LemmaPFSSIGMA} below applied to the sequence$$b_n=
\begin{cases}
\beta_n/ a_n & \text{ if } a_n\ne0\\
1 & \text{ if } a_n=0
\end{cases}.
$$
To see that, let us suppose that exist $n_1\ne n_2\in\NN$ such $a_{n_1}\ne0\ne a_{n_2}$ and $b_{n_1}\ne b_{n_2}$. In this case, it follows from Lemma~\ref{LemmaPFSSIGMA} that 
$$\limsup_{k\to\infty}\sum_{1\le n\le k, a_n\ne0}^k a_n\log (\beta_n/ a_n)=$$
$$=\limsup_{k\to\infty}\sum_{n=1}^ka_n\log b_n<\log\bigg(\limsup_{k\to\infty}\sum_{n=1}^ka_nb_n\bigg)=\log\bigg(\sum_{n\in\NN, a_n\ne0}\beta_n\bigg)<\log\bigg(\sum_{n\in\NN}\beta_n\bigg).$$
Now, suppose that  $a_nb_n=a_n\beta$ for some $\beta>0$ and every $n\in\NN$.
In this case, $\beta_n= a_n\beta$ when $ a_n\ne0$ and so, $$\sum_{n\in\NN, a_n\ne0}\beta_n=\sum_{n\in\NN, a_n\ne0} a_n\beta=\sum_{n\in\NN} a_n\beta=\beta.$$
Thus, in this case, 
$$\sum_{n\in\NN, a_n\ne0} a_n\log (\beta_n/ a_n)=\sum_{n\in\NN} a_n\log (\beta)=\log\beta=\log\bigg(\sum_{n\in\NN, a_n\ne0}\beta_n\bigg).$$
If $ a_n\ne0$ for every $n\in\NN$, we get that
$$\sum_{n\in\NN} a_n\log (\beta_n/ a_n)=\log\bigg(\sum_{n\in\NN}\beta_n\bigg).$$
On the other hand, if $ a_n=0$ for some $n\in\NN$, we have that $$\limsup_{k\to\infty}\sum_{1\le n\le k, a_n\ne0}^k a_n\log (\beta_n/ a_n)=\sum_{n\in\NN, a_n\ne0} a_n\log (\beta_n/ a_n)=\log\bigg(\sum_{n\in\NN, a_n\ne0}\beta_n\bigg)<\log\bigg(\sum_{n\in\NN}\beta_n\bigg).$$
\end{proof}

\begin{Lemma}
\label{LemmaPFSSIGMA}
Let $\{b_n\}_{n\in\NN}$ be a sequence of positive real numbers. 
If $a_n\in[0,1]$ is such that $\sum_{n\in\NN}a_n=1$ and $\beta:=\sum_{n\in\NN}a_n b_n<+\infty$ then either 
$$\limsup_{k\to\infty}\sum_{n=1}^ka_n\log b_n<\log\beta$$
or $a_nb_n=a_n\beta$  $\forall\,n\in\NN$ and $\sum_{n\in\NN}a_n\log b_n=\log \beta$.
\end{Lemma}
\begin{proof}
Since $\log$ is concave, we get that
\begin{equation}\label{Eqljhytg008}
a_1\log b_1+a_2\log b_2<\log(a_1b_1+a_2b_2)
\end{equation}
for every $a_1,b_2\in(0,1)$ and $b_1,b_2\in(0,+\infty)$ with $a_1+a_2=1$ and $b_1\ne b_2$.
Of course that $a_1\log b_1+a_2\log b_2=\log(a_1b_1+a_2b_2)=\log\beta$ when $b_1=b_2=\beta$.
\begin{Claim}\label{Claimie456} If $\alpha_1,\cdots,\alpha_k\in(0,1)$ and $\alpha_1+\cdots+\alpha_k=1$, then 
$$\sum_{n=1}^k\alpha_n\log \beta_n\le\log\bigg(\sum_{n=1}^k\alpha_n\beta_n\bigg)$$
for every $\beta_1,\cdots,\beta_k\in(0,+\infty)$.
\end{Claim}
\begin{proof}[Proof of the claim]
By \eqref{Eqljhytg008}, the claim is true for $k=2$ when $b_1\ne b_2$.
As claim is trivial for $k=2$ when $b_1=b_2=b$, the claim is true for $k=2$.

Let $\ell\ge3$ and assume by induction that the claim is true for every $2\le k<\ell$.
Now, suppose that $\alpha_1,\cdots,\alpha_\ell\in(0,1)$, $\alpha_1+\cdots+\alpha_\ell=1$ and $\beta_1,\cdots,\beta_\ell\in(0,+\infty)$.
Taking $\gamma=\sum_{n=2}^k\alpha_n=\sum_{n=1}^{k-1}\alpha_{n+1}$, it follows from the induction hypothesis that $$\sum_{n=2}^{k}\frac{\alpha_{n}}{\gamma}\log\beta_{n}=\sum_{n=1}^{k-1}\frac{\alpha_{n+1}}{\gamma}\log\beta_{n+1}\le \log\bigg(\sum_{n=1}^{k-1}\frac{\alpha_{n+1}}{\gamma}\beta_{n+1}\bigg)=\log\bigg(\sum_{n=2}^{k}\frac{\alpha_{n}}{\gamma}\beta_{n}\bigg).$$
From the case $k=2$, we have that $$\alpha_1\log\beta_1+\gamma\log\bigg(\sum_{n=2}^{k}\frac{\alpha_{n}}{\gamma}\beta_{n}\bigg)\le\log\bigg(\alpha_1\log\beta_1+\gamma\sum_{n=2}^{k}\frac{\alpha_{n}}{\gamma}\beta_{n}\bigg).$$
Thus, 
$$\sum_{n=1}^k\alpha_n\log \beta_n=\alpha_1\log\beta_1+\gamma\sum_{n=2}^{k}\frac{\alpha_{n}}{\gamma}\log\beta_{n}\le\alpha_1\log\beta_1+\gamma\log\bigg(\sum_{n=2}^{k}\frac{\alpha_{n}}{\gamma}\beta_{n}\bigg)\le$$
$$\le\log\bigg(\alpha_1\log\beta_1+\gamma\sum_{n=2}^{k}\frac{\alpha_{n}}{\gamma}\beta_{n}\bigg)=\log\bigg(\sum_{n=1}^{k}\alpha_{n}\beta_{n}\bigg).$$
\end{proof}

Let us consider, in the claim below, a weaker version of the lemma. 

\begin{Claim}\label{Claimhdhkh444}
If $\sum_{n\in\NN}a_n=1$ for a sequence $a_n\in[0,1]$ then $$\limsup_{k\to\infty}\sum_{n=1}^ka_n\log b_n\le\log\beta$$
for every sequence $b_n>0$.
\end{Claim}
\begin{proof}[Proof of the claim]
Taking $\gamma_k:=\sum_{n=1}^ka_n$, it follows from Claim~\ref{Claimie456} above  that 
$$\sum_{n=1}^ka_n\log b_n=\gamma_k\sum_{n=1}^k\frac{a_n}{\gamma_k}\log b_n\le\gamma_k\log\bigg(\sum_{n=1}^k\frac{a_n}{\gamma_k}b_n\bigg)=\gamma_k\log(1/\gamma_k)+\gamma_k\log\bigg(\sum_{k=1}^na_n b_n\bigg).$$
Since, $\lim_k\gamma_k=1$, we get that $\lim_k\gamma_k\log(1/\gamma_k)=0$ and so,
$$\limsup_{k\to\infty}\sum_{n=1}^ka_n\log b_n\le \limsup_{k\to\infty}\log\bigg(\sum_{n=1}^ka_n b_n\bigg)=\log\bigg(\limsup_{k\to\infty}\sum_{n=1}^ka_n b_n\bigg)=\log\beta,$$ proving the claim.
\end{proof}

Clearly, if $a_nb_n=a_n\beta$ for every $n\in\NN$ then $\sum_{n\in\NN}a_n\log b_n=\log(\sum_{n\in\NN}a_nb_n)=\log \beta$.
Hence, to conclude the proof of the lemma, we can suppose that there exist $n_1<n_2\in\NN$ such that $a_{n_1}b_{n_1}\ne a_{n_2}b_{n_2}$. In this case, $a_{n_1}\ne0\ne a_{n_2}$ and $b_{n_1}\ne b_{n_2}$.
If $a_{n_1}+a_{n_2}=1$ then $\sum_{n\in\NN}a_n\log b_n=a_{n_1}\log(b_{n_1})+a_{n_2}\log(b_{n_1})$ and the proof follows from \eqref{Eqljhytg008}. 
So, we can assume that $\gamma:=(\sum_{n\in\NN}a_n)-(a_{n_1}+a_{n_2})>0$.
In this case, it follows from Claim~\ref{Claimhdhkh444}  that
$$\limsup_{k\to\infty}\sum_{n=1}^ka_n\log b_n=a_{n_1}\log b_{n_1}+a_{n_2}\log b_{n_2}+\gamma\limsup_{k\to\infty}\sum_{n\in\{1,\cdots,k\}\setminus\{n_1,n_2\}}\frac{a_n}{\gamma}\log b_n\le$$
$$\le a_{n_1}\log b_{n_1}+a_{n_2}\log b_{n_2}+\gamma\log\bigg(\limsup_{k\to\infty}\sum_{n\in\{1,\cdots,k\}\setminus\{n_1,n_2\}}\frac{a_n}{\gamma} b_n\bigg).$$
Taking $\alpha=a_{n_1}+a_{n_2}$, it follows from \eqref{Eqljhytg008} that , we have that 
$$a_{n_1}\log b_{n_1}+a_{n_2}\log b_{n_2}=\alpha\bigg(\frac{a_{n_1}}{\alpha}\log b_{n_1}+\frac{a_{n_2}}{\alpha}\log b_{n_2}\bigg)<\alpha\log\bigg(\frac{a_{n_1}}{\alpha}b_{n_1}+\frac{a_{n_2}}{\alpha}b_{n_2}\bigg),$$
showing that  
$$\limsup_{k\to\infty}\sum_{n=1}^ka_n\log b_n<\underbrace{\alpha\log\bigg(\frac{a_{n_1}}{\alpha}b_{n_1}+\frac{a_{n_2}}{\alpha}b_{n_2}\bigg)+\gamma\log\bigg(\limsup_{k\to\infty}\sum_{n\in\{1,\cdots,k\}\setminus\{n_1,n_2\}}\frac{a_n}{\gamma} b_n\bigg)}_{\star}$$
By Claim~\ref{Claimie456}, 

$$\star\le\log\bigg(\alpha\bigg(\frac{a_{n_1}}{\alpha}b_{n_1}+\frac{a_{n_2}}{\alpha}b_{n_2}\bigg)+\gamma\bigg(\limsup_{k\to\infty}\sum_{n\in\{1,\cdots,k\}\setminus\{n_1,n_2\}}\frac{a_n}{\gamma} b_n\bigg)\bigg)=$$
$$=\log\bigg(\limsup_{k\to\infty}\sum_{n=1}^ka_n b_n\bigg)=\log\beta,$$
proving that $\limsup_{k}\sum_{n=1}^ka_n\log b_n<\log\beta$.
\end{proof}

\section{Appendix II: Maps with criticalities, free points and support of ergodic expanding measures}\label{SectionOnSuppExp}

In this section, we decompose the support of an ergodic expanding probability measure in two disjoint subsets, namely an open and dense set of free points and a compact meager set of confined points (see definitions below).
The {thermodynamic formalism for the expanding measures} giving full weight to the set of free points will be well understood. Moreover,  even if in general 
we cannot say much about the measures lying in the confined part of the system, this is a small region where in some cases one can hope to have additional geometric informations (as in Section~ \ref{SectioINTmaps} and \ref{SectionVmaps}) to complete the picture of the Thermodynamic Formalism of the whole system.

\medskip
Let $M$ be a Riemannian manifold and $N$ a nonempty open subset of $M$.
Let $\Lambda$ be a closed set such that $\Lambda=\overline{\interior(\Lambda)}$ and $f:\Lambda\to\Lambda$ be a continuous map and a $C^{1+}$ local diffeomorphism on $\Lambda$ except on a non-flat critical/singular set $\cc$.
Given $x\in\Lambda$, define the {\bf\em non-critical alpha-limit set} of a point $p\in\Lambda$, denoted by  $\alpha_f^0(x)$, as the set of all $y\in\Lambda$ such that exist $n_j\to+\infty$ and $x_j\in f^{-n_j}(x)$ such that $\{x_j,f(x_j),\cdots,f^{n_j-1}(x_j)\}\cap\cc=\emptyset$ and $\lim_jx_j=y$.
That is,
\begin{equation}\label{Equationkutderyui1}
	\alpha_f^0(x)=\alpha_{f_{_0}}(x),\;\text{ where }f_0:=f|_{\Lambda\setminus\cc}
\end{equation}
\begin{Definition}\label{def:freeconfined}
A point $p\in\Lambda$ is called {\bf\em $f$-confined} if $\alpha_f^0(p)\ne\Lambda$, otherwise  $p$ is called a {\bf\em free point}.
Denote  {\bf\em the set of all $f$-free points} of $f$ by $\cf(f)$.
\end{Definition}

Recall that a set $V$ is called {\bf\em meager} if it is contained on a countable union of closed sets with empty interior, in which case 
$\Lambda\setminus V$ contains a residual set.
A set $V\subset\Lambda$ called {\bf\em fat} if it is  not a meager.

\begin{Lemma}\label{Lemmahhhh1b6}
If $f$ is transitive and 
$\interior(\cf(f))\ne\emptyset$ then $\cf(f)$ is an open and dense subset of $\Lambda$ such that $f(\cf(f)\setminus\cc)=\cf(f)$. In particular, $f|_{\cf(f)\setminus\cc}$ is strongly transitive. Furthermore,
\begin{enumerate}
\item the set $\partial(\cf(f))$ coincides with the set of all $f$-confined points of $f$;
\item $\alpha_f^0(x)\subset\partial(\cf(f))$ for every $x\in\partial(\cf(f))$;
\item $\partial \Lambda\subset\partial(\cf(f))$;
\item if $f$ is strongly transitive then $\partial \Lambda\subset\partial(\cf(f))\subset\co_f^+(\cc)$.
\end{enumerate}
\end{Lemma}
\begin{proof} First consider the following claim.
\begin{Claim}\label{Claimnbl000hviy}
Given $p\in\cf(f)$ there is an open set $V\subset\cf(U)$ so that $V\cap\cc=\emptyset$ and $f(V)$ is an open neighborhood of $p$.
\end{Claim}
\begin{proof}[Proof of the claim]
Given $p\in\cf(f)$, it follows from $\alpha_f^0(p)=\Lambda$ that $p=f^n(q)$ for some $q\in\interior(\cf(f))\setminus\cc$ and $n\ge1$.
Taking $\varepsilon>0$ small enough, we get that $B_{\varepsilon}(q)\subset \interior(\cf(f))$, that $f^n|_{B_{\varepsilon}(q))}$ is an homeomorphism and $W:=f^n(B_{\varepsilon}(q))$ is an open set containing $p$.
As $f^n|_{B_{\varepsilon}(q))}$ is an homeomorphism, we get that $\alpha_f^0(f^n(x))\supset\alpha_f^0(x)\supset \Lambda$ for every $x\in B_{\varepsilon}(q)$ proving that $W\subset\cf(f)$.
As the same, we can conclude that $V=f^{n-1}(B_{\varepsilon}(q))$ is an open set containing in $\cf(f)$.
Moreover, $V\cap\cc=\emptyset$.
As $f(V)=W$ contains $p$ we conclude the proof of the claim.
\end{proof}
It follows from the Claim~\ref{Claimnbl000hviy} above that $\cf(f)$ is an open set and $f(\cf(f)\setminus\cc)=\cf(f)$, which is the first statement in the lemma.
Using that $f$ is transitive and $\cf(f)$ is a forward invariant open set, we conclude that $\cf(f)$ is also dense in $\Lambda$.
As $\alpha_{f|_{\cf(f)\setminus\cc}}(x)=\alpha_{f}^0(x)=\Lambda\supset\cf(f)$ for every $x\in\cf(f)$, $f|_{\cf(f)\setminus\cc}$ is strongly transitive.
As $\cf(f)$ is an open set, $\partial(\cf(f))=\Lambda\setminus\cf(f)$, proving that $\partial(\cf(f))$ is the set of $f$-confined points. This proves item (1) above.

If $\alpha_f^0(p)\cap\cf(f)\ne\emptyset$ for some $p\in \Lambda$, we get $p=f^n(q)$ for some $q\in\cf(f)=\interior(\cf(f))$ and so, one can use the argument of the proof of Claim~\ref{Claimnbl000hviy} above to conclude that $p\in\cp(f)$.
Hence, $\alpha_f^0(p)\subset\partial(\cf(f))$ for every $p\in\partial(\cf(f))$, proving item (2).

Item (3) follows as Claim~\ref{Claimnbl000hviy} and the definition of $\partial \Lambda$ imply that $\partial \Lambda  \cap\cf(f)=\emptyset$. 
Finally, if $f$ is strongly transitive and $p\in\partial(\cf(f))$,  as $\alpha_f(p)=\Lambda$ and $p\notin\cf(f)$ we get that $\co_f^-(p)\cap\cc\ne\emptyset$, 
proving that $p\in\co_f^+(\cc)$. This proves item (4) and completes the proof of the lemma.
\end{proof}

\begin{Definition}\label{def:expandingpotential} A continuous function $\varphi:\Lambda\to\RR$ is  a {\bf\em free expanding potential} for $f$ if
$$h_{\mu}(f)+\int\varphi \,d\mu<P_{\ce(f|_{\cf(f)})}(\varphi),\;\;\forall\mu\in\cm^1(f)\setminus\ce(f|_{\cf(f)}).$$
\end{Definition}

\begin{Theorem}\label{Theoremhgjvin09}
Let $M$ be a  Riemannian manifold and $f:M\to M$ be a $C^{1+}$ map with a non-flat critical set $\cc$.
Let $\mu$ be an ergodic invariant expanding probability measure with a fat support and write $g=f|_{\supp\mu}$. The set $\cf(g)$ is an open and dense subset of $\supp\mu$, with $\mu(\cf(g))=1$, and such that $g(\cf(g)\setminus\cc)=\cf(g)$.
In particular, $g|_{\cf(g)\setminus\cc}$ is strongly transitive.
Furthermore, the following properties hold:
\begin{enumerate}
\item $\partial(\cf(g))$ is the set of all $g$-confined points of $\supp\mu$.
\item $\alpha_{g}^0(x)\subset\partial(\cf(g))$ for every $x\in\partial(\cf(g))$.
\item $\partial\supp\mu\subset\partial(\cf(g))$.
\item If $g$ is strongly transitive then $\partial \supp\mu \subset\partial(\cf(g))\subset\co_g^+(\cc)$.
\item If $\varphi$ is a Hölder continuous potential $\varphi$ then $g$ has at most one expanding equilibrium state $\mu_{\varphi}$ with $\supp\mu_{\varphi}=\supp\mu$. 
Moreover, if $\nu$ is an ergodic expanding equilibrium state for $\varphi$ with $\supp\nu\subsetneq\supp\mu$, then $\supp\nu\subset\partial(\cf(g))$.
\item Suppose that $g$ has a measure of expanding maximal entropy with a fat support.
If $\varphi$ is a Hölder continuous potential  with a small enough oscillation, then there exists a unique expanding equilibrium state $\mu_{\varphi}$ such that $\supp\mu_{\varphi}=\supp\mu$. Moreover, if $\nu$ is an ergodic expanding equilibrium state for $\varphi$ with $\supp\nu\subsetneq\supp\mu$, then $\supp\nu\subset\partial(\cf(g))$.
\item Given a Hölder continuous potential $\varphi$, there is a $\mu_0\in\cm^1(g)$ such that 
$$h_{\mu_0}(g)+\int\varphi \,d\mu_0\ge P_{\ce(g|_{\cf(g)})}(\varphi).$$
\item If $\varphi$ is a free expanding  Hölder continuous potential for $g$, then $g$ has a unique equilibrium state $\mu_{\varphi}$ for $\varphi$.
Moreover, $\mu_{\varphi}\in\ce(g)$ and $\supp\mu_{\varphi}=\supp\mu$.
\end{enumerate}
\end{Theorem}
\begin{proof}
Taking $g_0=g|_{\supp\mu\setminus\cc}$, we get that $\mu$ is a $g_0$ ergodic expanding $g_0$-invariant probability measure and $\interior(\supp\mu)\ne\emptyset$.
Thus, as $\alpha_g^0(x)=\alpha_{g_{_0}}(x)$, it follows from Proposition~\ref{Propositioniougoygh} of Appendix I   that $\cf(g)$ is an open and dense subset of $\supp\mu$ with $\mu(\cf(g))=1$ and $g(\cf(g)\setminus\cc)=\cf(g)$.

\smallskip
{\bf\em Proof of items $(1)$ - $(6)$}. 
As $\interior(\cf(g))\ne\emptyset$, the proof of items $(1)$ to $(4)$ follows from   Lemma~\ref{Lemmahhhh1b6} above.
As $g_1:=g|_{\cf(g)\setminus\cc}$ is strongly transitive items $(5)$ and  $(6)$ follow from Theorems~\ref{Maintheorem00} and \ref{Maintheorem11} applied to $g_1$ and the fact that 
$$
\nu\in\cm^1_{erg}(g)\setminus\cm^1(g_1)\implies\nu(\cf(g))=0 \quad \text{or}\quad \nu(\cc)=1
$$ 
and so, any probability measure $\nu\in\ce(g)\cap\cm^1_{erg}(g)\setminus\cm^1(g_1)$ satisfies $\nu(\cf(g))=0$.
Note also that, by Lemma~\ref{LemmaInduhbiuty6}, every ergodic expanding equilibrium state $\mu_0$ for $g_1$ belongs to $\ce^*(g_1)$ and this implies that $\supp\mu_0\supset\cf(g)$, that is, $\supp\mu_0=\supp\mu$.

\smallskip
{\bf\em Proof of item $(7)$}. Given a piecewise Hölder continuous potential $\varphi$, 
let $(\mu_n)_n$ be a sequence of ergodic $g$-invariant probability measures in $\ce(g)$ with $\mu_n(\cf(g))=1$ and such that $h_{\mu_n}(g)+\int\varphi \,d\mu_n\to P_{\ce(g|_{\cf(g)})}(\varphi)$ as $n$ tends to infinity.
Taking a subsequence if necessary, we may assume that $\lim_k\mu_n=\mu_0$ for some $\mu_0\in\cm^1(g)$.

As $\ce(g_1)\ne\emptyset$ then $\per(f)\ne\emptyset$.
By Proposition~\ref{PropositionTotTransPer} in Appendix, replacing $g_1$ by $h:=g_1^{\ell}|_V$, where $\ell=\min\{j\ge1\,;\,\fix(g_1^j)\ne\emptyset\}$ and $V$ is an open set such that $(g_1^*)^{\ell}(V)\subset V$, we get that $h^n$ is strongly transitive for every $n\ge1$ and $\frac{1}{\mu(V)}\mu|_V\in\ce(h,1)$.
Moreover, $$\cm^1(g_1)\ni\nu\mapsto\frac{1}{\nu(V)}\nu|_V\in\cm^1(h)$$ is a bijection sending $\ce(g|_{\cf(g)})=\ce(g_1)$ onto $\ce(h)$.

Writing $\nu_n=\frac{1}{\mu_n(V)}\mu_n|_V$ and $\widetilde{\varphi}(x)=\sum_{j=0}^{\ell-1}\varphi\circ g^j(x)$, it follows from Proposition~\ref{prop:reduction-III} that exists a sequence $\overline{\nu}_n\in\ce^*(h)$ such that $d(\nu_n,\overline{\nu}_n)\to0$, $|\int\widetilde{\varphi}d\nu_n-\widetilde{\varphi}d\overline{\nu}_n|\to0$ and $\liminf h_{\overline{\nu}_n}(h)\ge \liminf h_{\nu_n}(h)$.
As $\sup|Dh|\le\max|Dg^{\ell}|<+\infty$, it follows from item $(7)$ of Theorem~\ref{Theoremuhbkj100201} that exist a zooming contraction $\beta$ and $\delta>0$ such that all $\overline{\nu}_n$ are $(\beta,\delta/2,1)$-zooming measures for $h$.
Thus,  $\overline{\mu}_n:=\frac{1}{\ell}\sum_{j=0}^{\ell-1}\overline{\nu}_n\circ g^{-j}\in\cm^1(g)$ is a sequence of $(\beta,\delta/2,\ell)$-zooming probability measures for $g$ such that $d(\mu_n,\overline{\mu}_n)\to0$, $|\int\varphi \,d\mu_n-\int\varphi d\overline{\mu}_n|\to0$ and $\liminf h_{\overline{\mu}_n}(g)\ge\liminf h_{\mu_n}(g)$.
As $\overline{\mu}_n\to\mu_0\in\cm^1(g)$, it follows from Lemma~\ref{Lemmadfghjye32zzx} of Appendix that $h_{\mu_0}(g)\ge\liminf h_{\overline{\mu}_n}(g)\ge\liminf h_{\mu_n}(g)$. Hence, the continuity of $\cm^1(g)\ni\nu\mapsto\int\varphi d\nu$ implies that $h_{\mu_0}(g)+\int\varphi \,d\mu_0\ge P_{\ce(g|_{\cf(g)})}(\varphi)$.

\smallskip
{\bf\em Proof of item $(8)$}. 
Let $\varphi$ be a free expanding Hölder continuous potential. 
The existence of an equilibrium state for $\varphi$ is a consequence from item $(7)$ and the definition of a free expanding potential. 
The uniqueness of the equilibrium state for $\varphi$ follows from the item $(5)$.
Indeed, the hypothesis of $\varphi$ being a free expanding potential, implies that  every equilibrium state $\eta$ for $\varphi$ must belongs to $\ce(f)$ and satisfies $\eta(\cf(g))=1$.
As a consequence, $\supp\eta\not\subset\partial(\cf(g))$ and, as a free expanding potential is an expanding potential, it follows from item $(5)$ and that $\eta\in\ce(g)$ satisfies $\supp\eta=\supp\mu$ and it is the unique possible equilibrium state for $\varphi$.
\end{proof}

\end{document}